\documentclass[11pt]{amsart}

\usepackage{fullpage}
\usepackage{amssymb, amsmath, wasysym, mathrsfs, mathtools, color, stmaryrd, enumerate}
\usepackage[all, cmtip]{xy}
\usepackage{url}
\usepackage{tikz-cd}

\definecolor{hot}{RGB}{65,105,225}
\usepackage[pagebackref=true,colorlinks=true, linkcolor=hot ,  citecolor=hot, urlcolor=hot]{hyperref}

\newcommand{\commentOut}[1]{}

\theoremstyle{plain}
\newtheorem{theorem}{Theorem}[section]
\newtheorem{prop}[theorem]{Proposition}

\newtheorem{que}[theorem]{Question}

\newtheorem{cor}[theorem]{Corollary}
\newtheorem{conj}[theorem]{Conjecture}
\newtheorem{lemma}[theorem]{Lemma}
\newtheorem{thrm}[theorem]{Theorem}

\theoremstyle{definition}

\newtheorem{defn}[theorem]{Definition}
\newtheorem{rmk}[theorem]{Remark}

\newtheorem*{ex*}{Example}

\newtheorem{subs}[theorem]{ }
\newtheorem{propdef}[theorem]{Proposition-Definition}

\newcommand{\cO}{{\mathcal O}}

\newcommand{\cH}{{\mathcal H}}

\newcommand{\cJ}{{\mathcal J}}
\newcommand{\cF}{{\mathcal F}}

\newcommand{\cM}{{\mathcal M}}

\newcommand{\cV}{{\mathcal V}}
\newcommand{\cL}{\mathcal{L}}
\newcommand{\cD}{\mathcal{D}}

\def\cN{\mathcal N}

\newcommand{\sV}{\mathscr{V}}

\newcommand{\bP}{{\mathbb{P}}}
\newcommand{\bQ}{{\mathbb{Q}}}
\newcommand{\bZ}{{\mathbb{Z}}}
\newcommand{\bR}{{\mathbb{R}}}
\newcommand{\bC}{{\mathbb{C}}}
\newcommand{\bA}{{\mathbb{A}}}
\newcommand{\bN}{{\mathbb{N}}}

\def\bz{{\mathbf{0}}}

\newcommand{\ubul}{{\,\begin{picture}(-1,1)(-1,-3)\circle*{2}\end{picture}\ }}

\newcommand{\lbul}{{\,\begin{picture}(-1,1)(-1,-2)\circle*{2}\end{picture}\ }}

\DeclareMathOperator{\codim}{codim}

\DeclareMathOperator{\id}{id}

\DeclareMathOperator{\im}{Im}

\DeclareMathOperator{\pic}{Pic}
\DeclareMathOperator{\lct}{lct}
\DeclareMathOperator{\rank}{rank}

\DeclareMathOperator{\Aut}{Aut}

\DeclareMathOperator{\homo}{Hom}
\DeclareMathOperator{\enmo}{End}
\DeclareMathOperator{\spec}{Spec}

\DeclareMathOperator{\Def}{Def}
\DeclareMathOperator{\gr}{gr}                    
                    
\DeclareMathOperator{\length}{length}                    
\DeclareMathOperator{\ord}{ord}                    
\DeclareMathOperator{\Hom}{Hom}
\DeclareMathOperator{\End}{End}
\DeclareMathOperator{\HH}{H}
\DeclareMathOperator{\Trace}{Trace}
\DeclareMathOperator{\Span}{Span}
\DeclareMathOperator{\MC}{MC}
\DeclareMathOperator{\IC}{IC}
\DeclareMathOperator{\mld}{mld}

\def\ra{\rightarrow}

\def\al{\alpha}

\newcommand\be{\begin{equation}}
\newcommand\ee{\end{equation}}
\def\xa{\xrightarrow}
\def\Linf{L_\infty}

\def\Art{{\mathcal{A}rt}}
\def\Set{{\mathcal{S}et}}
\def\eps{\epsilon}

\newcommand{\rS}{{\mathrm{S}}}
\newcommand{\rrS}{{ \bar{\mathrm{S}}}}
\def\sf{\mathsf{f}}
\def\dec{{\mathrm{dec}}}
\def\mm{{\mathfrak{m}}}

\def\wh{\widehat}

\def\length{{\rm{length}}}
\def\sQ{\mathcal Q}
\def\sR{\mathcal R}
\def\sM{\mathcal M}
\newcommand{\mb}{\mathcal{M}_{\textrm{B}}}
\newcommand{\mdol}{\mathcal{M_{\textrm{Dol}}}}
\newcommand{\mdr}{\mathcal{M}_{\textrm{DR}}}
\def\ti{\tilde}

\title[Deformations with cohomology constraints: a review]{Deformations with cohomology constraints: a review}

\author{Nero Budur}\address{KU Leuven, Celestijnenlaan 200B, B-3001 Leuven, Belgium} \email{nero.budur@kuleuven.be}
\author{An-Khuong Doan}\address{KU Leuven, Celestijnenlaan 200B, B-3001 Leuven, Belgium}\email{an-khuong.doan@kuleuven.be}

\begin{document}

\begin{abstract} 
Deformation problems with cohomology constraints over a field of characteristic zero are controlled by $L_\infty$ pairs. In this largely expository article we review this theory and focus on recent applications.
\end{abstract}

\maketitle

\setcounter{tocdepth}{1}

\numberwithin{equation}{section}

\tableofcontents

\section{Introduction}

Deformation theory has been well-developed for a while. Over a field of characteristic zero, a principle of Deligne \cite{De86} says that  every deformation problem is controlled by a differential graded Lie algebra, and two equivalent dgla's describe the same deformation theory. The technique was developed by Goldman-Milson \cite{GM}. A more flexible version in terms of $L_\infty$ algebras is due to Fukaya, Kontsevich, Soibelman, Manetti, and others, cf. \cite[p.451]{Mane}.

Recently deformation theory was enhanced to handle cohomology constraints. Deformation problems with cohomology constraints are controlled by dgl pairs, or better, $L_\infty$ pairs, according to a generalization of Deligne's principle by Budur-Wang \cite{BW}, Budur-Rubi\'o \cite{BR}.  A pair means an algebra together with a module.  The main feature is that only one $L_\infty$ pair is necessary to control the local structure at a fixed point $E$ of the pairs $(\cM,\cV_k^i)$ for all $k,i \in\bZ$, where $\cM$ is a fixed moduli scheme of objects with a cohomology theory and $\cV_k^i=\{E\in\cM\mid h^i(E)\ge k\}$ are the cohomology jump  subschemes. Equivalent $L_\infty$ pairs describe the same $(\cM,\{\cV_k^i\}_{i,k})$ locally at $E$. 

The goal of this article is to give a review of the deformation theory with cohomology constraints and of its applications.  We start in Section \ref{secDFT} describing the theory as a black box leaving the technical details to Part \ref{part2}.  Section \ref{secA} is a short guide to the applications we  cover in this paper. The next sections, grouped in Part \ref{part1}, are concerned with applications of this theory to stable vector bundles on  curves following \cite{theta}. In Part \ref{part3} we review other applications of this theory from \cite{BW, BR}. 
 In Part \ref{part2} we review $\Linf$ structures and the technical aspects of  the deformation theory with cohomology constraints from \cite{BW, BR}.

 \subs{\bf Acknowledgement.} We thank M. Aprodu, E. Arbarello, C. Chiu, M. Coppens, R. Docampo,  G. Farkas, M. Musta\c{t}\u{a}, J. Nicaise, M. Rubi\'o, M. Saito, C. Schnell, R. Yang, N. Zhang, and the referee for useful comments.
The work was  supported by the Methusalem grant METH/21/03, and the grants G097819N, G0B3123N from FWO.

\subs{\bf Notation.}\label{subNot} We work over a field $K$ of characteristic zero. When needed, we also require $K$ to be algebraically closed. An {\it algebraic variety} $X$ over $K$ is a geometrically irreducible, reduced, separated scheme of finite type over $K$. By $Sing(X)$ we denote the  singular locus of $X$ with the induced reduced scheme structure. If $x$ is a point of $X$, we denote by $T_xX$ the Zariski tangent space of $X$ at $x$, and by $TC_xX$ the tangent cone of $X$ at $x$, cf. Section \ref{appSI}.
We freely switch between vector bundles and locally free sheaves. We denote by $E^\vee$ the dual of a vector bundle $E$ on $X$, by $\enmo (E)$ the vector bundle of endomorphisms, and  set $h^i(E)=h^i(X,E):=\dim_K \HH^i(X,E)$. 

By convention, if $k\in\bZ$ and $l,l'>0$, then the ideal generated by the $k$-minors of an $l'\times l$ matrix of linear forms is 0 if $k>\min\{l,l'\}$, and is the ideal $\langle 1\rangle$ if $k\le 0$.

\section{Deformations with cohomology constraints as a black box}\label{secDFT}

We start by summarizing the deformation theory with cohomology constraints over a field $K$ of characteristic zero from \cite{BW, BR}, leaving the technical details for Section \ref{apxLinf}. 
By this theory one means to understand the local structure of cohomology jump loci 
$
\cV^i_k=\{L\in \cM\mid \dim_K \HH^i(L)\ge k\}
$
of objects in a moduli space $\cM$ endowed with a cohomology theory over $K$. When the moduli space $\cM$ is a scheme,  $\cV^i_k$ are  to be considered with a natural closed subscheme structure. 

\subs{\bf Dgl pairs.}
Classical deformation theory studies the local structure of $\cM$ at an object $L$ by  attaching a differential graded Lie algebra (dgla) $C$ such that the formal completion of $\cM$ at $L$ is isomorphic to the deformation functor of $C$
$$
\widehat{\cM}_L\simeq \Def ({C}) :\Art \ra \Set
$$ 
as functors from the category of local Artinian finite type $K$-algebras to the category of sets. 

Recall that for a scheme $\cM$ and a point $L$ on $\cM$ corresponding to a maximal ideal $\mathfrak{m}$, the formal completion $\widehat{\cM}_L$ is the formal spectrum of the completion of the structure sheaf $\cO$ of $\cM$ at $\mathfrak{m}$,
$$
\widehat{\cM}_L = {\rm{Spf}}(\widehat{\cO}_{\mathfrak{m}}).
$$
It defines a functor by associating to $A\in\Art$ the set of local $K$-algebra homomorphisms  $\widehat{\cO}_{\mathfrak{m}}\ra A$. By Yoneda Lemma, this functor is pro-represented by $\widehat{\cO}_{\mathfrak{m}}$ and therefore it causes no harmful ambiguity to denote this functor by $\widehat{\cM}_L$ too.

 The deformation functor of the dgla $C$ is defined by associating to every $A$ in $\Art$ with maximal ideal $\mathfrak{m}_A$ the set of Maurer-Cartan elements of $C\otimes \mm_A$ modulo the gauge action
$$
\Def(C; A) =\{\omega\in C^1\otimes_K \mathfrak{m}_A\mid d_C\omega+\frac{1}{2}[\omega,\omega]_C=0\}/ (C^0\otimes_K\mathfrak{m}_A),
$$
where $d_C$ is the differential of $C$ extended by identity on $A$, and $[.\,,.]_C$ is the Lie bracket of $C$ extended by the usual multiplication on $A$. Two quasi-isomorphic dgla's have isomorphic deformation functors, by \cite{GM}.

The moduli space $\cM$ plays an illustrative role in the above. In general, for an object $L$ of a category one has a well-defined deformation functor of $L$, and deformation subfunctors for each $i, k$, which play the role of $\widehat{\cM}_L$, respectively of $\widehat{(\cV^i_k)}_L$, in the above.

To the object $L$ one also attaches a (left) dgl  module $M$ over $C$, that is,  a dgl pair  $(C,M)$ in the terminology from \cite{BW}. Typically this process goes hand in hand with finding the dgla $C$, and the cohomology pair $(HC,HM)$ is typically the pair of cohomologies of $(\enmo(L),L)$, where $\enmo(L)$ is the endomorphisms object attached to $L$, which  exists since we dispose of a cohomology theory.  When $C, M$ are graded by $\bN$, and as cochain complexes they are bounded-above and have finite-dimensional cohomology, one has well-defined deformation subfunctors $\Def^i_k(C,M)$ of $\Def(C)$ for every integral $i$ and $k$ such that
$$
\widehat{(\cV^i_k)}_L\simeq \Def^i_k(C,M),
$$
and any two equivalent dgl pairs give the same cohomology jump deformation subfunctors, by \cite[\S 3]{BW}. This is by definition what it means for $(C,M)$ to {\it control} the deformations of $L$ with cohomology constraints. The deformation subfunctors send $A\in\Art$ to
\begin{align*}
\Def^i_k(C,M;A)  = \{\omega\in C^1\otimes_K & \mathfrak{m}_A\mid  d_C\omega+\frac{1}{2}[\omega,\omega]_C=0\text{ and }\\
& J^i_k(M\otimes_KA, d_M+\omega)=0 \}/ (C^0\otimes_K\mathfrak{m}_A),
\end{align*}
where $d_M$ is the differential of $M$ extended by identity on $A$ to $M\otimes_KA$, and the cohomology jump ideals $J^i_k\subset A$ of the {Aomoto} complex $(M\otimes_KA, d_M+\omega)$ of $A$   modules are defined as follows.

\begin{defn}\label{defCJI} (\cite[\S 2]{BW})
 Let $R$ be a noetherian commutative ring and  $N$ a complex of $R$-modules, bounded above, with finitely generated cohomology. There always exists a bounded above complex $F$ of finitely generated free $R$-modules and a quasi-isomorphism of complexes $F \xrightarrow{\sim} N$. The {\it cohomology jump ideals} of $N$ are the ideals in $R$ defined as 
$$
J^i_k(N) = I_{{\rank}(F^i)-k+1}(d^{i-1} \oplus d^i),
$$
where $d^i: F^i \rightarrow F^{i+1}$ are the differentials of $F$, and $I_r$ is the ideal generated by the  $r\times r$ minors. The cohomology jump ideals do not depend on the choice of the free resolution.
\end{defn}

\begin{rmk} (i) The definition is based on the following elementary linear algebra fact. If $F'\xa{d'}F\xa{d''}F''$ is a complex of finite-dimensional vector spaces, then its homology has dimension $\ge k$ if and only if the rank of the matrix $\left(\begin{matrix} d' & 0\\ 0 & d'' \end{matrix}\right)$ is at most $\dim F-k$.

(ii) The cohomology jump ideals generalize the notion of Fitting ideals from modules to complexes of modules. For example, they are compatible with base change: if $S$ is a noetherian $R$-algebra and $N$ as above is a complex of flat $R$-modules, then $J^i_k(N)\cdot S=J^i_k(N\otimes_R S)$  as ideals in $S$.

(iii) In particular, the cohomology jump ideals deserve their name: if $N$ is a complex of flat $R$-modules and $\mathfrak m\subset R$ is a maximal ideal, $J^i_k(N)\subset \mathfrak m$ if and only if $\dim_{R/\mathfrak m}H^i(N\otimes_RR/\mathfrak m)\ge k$.

(iv) Cohomology jump ideals arise in many geometric situations. For example, if $\cN$ is a bounded-above complex of coherent sheaves on a projective scheme $X$ of finite type of $K$, then the set $\{x\in X\mid \dim H^i(X,\cN\otimes k(x))\ge k\}$ is endowed with a closed subscheme structured by the cohomology jump ideals. An application of this example is the proof of the semicontinuity theorem: if $X\to Y$ is a projective morphism of noetherian schemes, and $\cF$ is a coherent sheaf on $X$ flat over $Y$, then the locus $\{y\in Y\mid \dim H^i(X_y,\cF_y)\ge k\}$ carries a natural closed subscheme structure.
\end{rmk}

\subs{\bf $\Linf$ pairs.} A  more efficient theory has been developed in \cite{BR} by passing from dgl pairs to $\Linf$ pairs. The dgl approach has the disadvantage that typically the dgl pairs involved are infinite dimensional in each degree, while the problem one deals with involves typically finite dimensional cohomology vector spaces. Even computing tangent spaces to cohomology jump deformation functors is difficult using dgl pairs. 

The category of dgla's is a subcategory of  the category of  $\Linf$ algebras, and the category of dgl pairs is a subcategory of the category of $\Linf$ pairs, that is, pairs consisting of an $\Linf$ algebra together with an $\Linf$ module, with morphisms appropriately defined.
We refer to Section \ref{apxLinf} for  details and definitions. Here we only recall that an {\it $\Linf$ algebra} is a graded vector space $ C$ together with a collection of graded anti-symmetric multilinear maps
$$
l_n: C^{\otimes n}\ra  C
$$
of degree $2-n$ for every $n\ge 1$, satisfying a generalized Jacobi identity. The dgla's are precisely the $\Linf$ algebras with $l_n=0$ for $n\ge 3$, in which case $l_1$ is the differential and $l_2$ is the Lie bracket. An {\it $\Linf$ module} over $ C$ is a graded vector space $ M$ together with a collection of graded linear maps
$$
m_n: C^{\otimes n-1}\otimes  M\ra  M
$$
of degree $2-n$ for every $n\ge 1$, satisfying a certain compatibility with the maps $l_n$. The notion of {\it weak equivalence} between $\Linf$ pairs is recalled in  Section \ref{apxLinf}.

The homotopy transfer theorem  guarantees that for every dgla $C$ and  dgl module $M$, the cohomology graded vector spaces $HC$ and $HM$ can be endowed with an $\Linf$ algebra structure $l_*$ and, respectively, a $\Linf$ module structure $m_*$, such that: $l_1=0$, $m_1=0$, $l_2$ and $m_2$ are induced from  the Lie bracket and the dgl module structure, and the dgl pair $(C,M)$ is weakly equivalent as an $\Linf$ pair with $(HC,HM)$, see Theorem \ref{thrmTLP}.

There is a well-defined deformation functor $\Def (HC)$ attached to $HC$ such that to  every $A$ in $\Art$ one attaches
 \be\label{eqLMC}
\Def(HC; A)=\Bigl\{\omega\in H^1C\otimes_K \mathfrak{m}_A\mid \sum_{n\ge 2}\frac{1}{n!}l_n(\omega^{\otimes n})=0\Bigr\}/_\sim
\ee
where $\sim$ is the homotopy equivalence relation, see Definition \ref{defMCequi}. Via the homotopy transfer theorem, the implication for deformation theory is that there is one more  isomorphism of functors
\be\label{eqMOD}
\widehat{\cM}_L\simeq \Def (C)\simeq \Def (HC)
\ee
This is due to Fukaya, Kontsevich, Soibelman,
Manetti, etc., see Theorem \ref{corFKM}.

It was shown in \cite{BR} that, under the assumption that the cochain complex $M$ is bounded above, there are well-defined subfunctors $\Def^i_k(HC,HM)$ of $\Def(HC)$ such that
\begin{equation}\label{eqLJI}
\begin{split}
\Def^i_k(HC,HM;A)=&\left\{\omega\in  H^1C\otimes_K \mathfrak{m}_A\mid \sum_{n\ge 2}\frac{1}{n!}l_n(\omega^{\otimes n})=0\text{ and }\right.\\
&\left.J^i_k\biggl(HM, \sum_{n\ge 1}\frac{1}{n!}m_{n+1}(\omega^{\otimes n}\otimes \_)\biggr)=0 \right\}/_\sim
\end{split}
\end{equation}
with the cohomology jump ideals $J^i_k\subset A $ defined as above, see Definition \ref{defDik}. A weak equivalence of  $\Linf$ pairs induces an isomorphism of deformation functors restricting to isomorphisms of the cohomology jump deformation subfunctors, see Theorem \ref{thmHCHMa}. 

For deformation theory with cohomology constraints this implies that there is one more   isomorphism of functors
\be\label{eqTISOM}
\widehat{(\cV^i_k)}_L\simeq \Def^i_k(C,M)\simeq \Def^i_k(HC,HM),
\ee
see  Theorem \ref{thmHCHMa}, cf. \cite[Thm. 1.6]{BR}. The price one pays for gaining finite-dimensionality is thus the introduction of higher degree terms in the equations.

Recall that the tangent space to a deformation functor $F$ is $TF:=F(K[\eps]/(\eps^2))$, see \cite[3.5]{Mane}.

\begin{theorem}\label{thmTgxxy}{{\rm{(}}\cite[Thm 1.7]{BR}{\rm{)}}}
\label{propTDef}
Let $(C,M)$ be a dgl pair or, more generally, an $\Linf$ pair, over a field of characteristic zero. Assume that $C, M$ are $\bN$-graded and that $M$ is bounded above as a cochain complex. Let $h_i=\dim H^iM$. The Zariski tangent spaces to the functors 
$$\Def^i_0(C,M)=\Def(C)\supset\ldots\supset\Def^i_k(C,M)\supset\ldots \supset  \Def^i_{h_i+1}(C,M)=\emptyset$$ 
are: the full Zariski tangent space $T\Def(C)=H^1C$ if $k<h_i$; empty if $k>h_i$; and if $k=h_i$, equal to the kernel of the linear map
$$
H^1C\ra \bigoplus_{j=i-1,i}\Hom(H^{j}M,H^{j+1}M)
$$
induced from the $\Linf$ module multiplication maps
$
H^1C\otimes H^jM\ra H^{j+1}M.
$
\end{theorem}

\section{Guide to applications}\label{secA}

Deformation theory with cohomology constraints provides a conceptual and technical framework for such deformation problems. To successfully apply the black box presented above, the next step is to find controlling dgl or $\Linf$ pairs with enough simplifying conditions to allow one to say something useful. The applications surveyed here fall into three classes; $(C,M)$ denotes a dgl pair:

\begin{itemize}
\item Formal cohomology $\Linf$ pairs $(HC,HM)$, that is, the only non-zero $\Linf$ operations are $l_2:HC^{\otimes 2}\to HC$ and $m_2:HC\otimes HM\to HM$. Equivalently, a dgl pair $(C,M)$ is {\it formal} if it is equivalent to its  cohomology dgl  pair $(HC,HM)$.
\item Cohomology $\Linf$ pairs $(HC,HM)$ such that only finitely many $\Linf$ module multiplication maps $m_{n+1}:HC^{\otimes n}\otimes HM\to HM$ are non-zero;
\item Cohomology $\Linf$ pairs $(HC,HM)$ such that 
$m_2:C\otimes HM\to HM$ is generic.
\end{itemize}

These assumptions simplify the equations in (\ref{eqLMC}) and (\ref{eqLJI}). There are additional simplifying assumptions on the objects to be deformed, such as ``stable", ``semi-simple", ``irreducible", to the effect that one does not have to mod out by any equivalence relation in (\ref{eqLMC}) and (\ref{eqLJI}), or at least that one understands very well this equivalence relation, and, in the last two cases, that $\Def(C)\simeq\widehat{H^1C}_{0}$.

The first and third cases lead to answers to the deformation problem with cohomology constraints that are as nice as possible, as close to linear algebra as one can hope to get.
The second case leads to qualitative conclusions. If the cohomology jump loci one studies are already algebraic, then the finitely many $\Linf$ multiplication maps provide another algebraic structure instead of just an answer in terms of power series. Bi-algebraicity leads to constraints via an Ax-Lindemann type of theorem.

In Part \ref{part1} we illustrate  the last case and look at  stable vector bundles on smooth projective curves with generic Petri maps following \cite{theta}. This subject, Brill-Noether theory, has a long and rich history. We also present here a few new results which do not appear in \cite{theta}.

In Part \ref{part3} we will see applications that fall in the first two categories. Among the first case we cover: stable holomorphic vector bundles with vanishing  Chern classes, irreducible complex local systems, stable Higgs bundles with vanishing  Chern classes, and semisimple representations of the fundamental group of compact K\"ahler manifold, following \cite{BW}. Regarding the second case we look at spaces endowed with a weight filtration such as complex algebraic varieties, links of singularities of complex algebraic varieties, and Milnor fibers of germs of holomorphic functions, all with a vanishing $W_0H^1=0$ constraint, following \cite{BR}.

\bigskip

\part{Brill-Noether loci and $\Linf$ pairs}\label{part1} 

\bigskip

\section{Introduction to Part \ref{part1}} 
Let $C$ be  a smooth projective curve of genus $g$ over an algebraically closed field $K$ of characteristic zero. In this part we consider the Brill-Noether loci of $C$. This is a subject with a long history, see Section \ref{apxBN} for a short review related to the results below.

Let $\omega_C$ be the canonical bundle of $C$. Let $n>0, d\ge 0, k> 0$ in $\bZ$. Fix a vector bundle $F$ on $C$. Let $\cM_{n,d}$ be  the moduli space of stable vector bundles on $C$ of rank $n$ and degree $d$. Let
$$
\cV_{n,d,k}(F):=\{E\in \cM_{n,d}\mid h^{0}(C,E\otimes F)\ge k\}
$$
 endowed with the natural structure of closed subscheme of $\cM_{n,d}$. 
 We set $\cV_{n,d,k}=\cV_{n,d,k}(\cO_C)$. When $\cM_{n,d}$ is fixed from the context, we set $\cV_k(F)=\cV_{n,d,k}(F)$ and $\cV_k=\cV_k(\cO_C)$.  It is known that $\cV_k(F)$ form a filtration of closed subschemes
 $$
 \cM_{n,d}=\cV_0(F)\supset\cV_1(F)\supset\cV_2(F)\supset\ldots.
 $$
  For $E\in\cM_{n,d}$, the {\it Petri map} is a natural map
$$
\pi_{E,F}: \HH^0(C,E\otimes F)\otimes \HH^0(C,E^\vee\otimes F^\vee\otimes \omega_C) \to \HH^0(C,E\otimes E^\vee\otimes\omega_C),
$$
see Definition \ref{defBNP}. We set $\pi_{E}=\pi_{E,\cO}$. Set 
$l=h^0(C,E\otimes F)$, $l'=h^1(C,E\otimes F)$. Then $l-l'$ is given by (\ref{eqChi}). 
We assume $l\ge 1$. If $E$ is a line bundle, that is $n=1$, we denote it by $L$ to stress this fact and use the classical notation $\pic^d(C)=\cM_{1,d}$, $W^{k-1}_d=\cV_k$.

Using deformation theory with cohomology constraints in terms of $\Linf$ pairs the following was shown in \cite{theta}. 

\begin{theorem}\label{thmGenToConeEt} There is a canonical isomorphism of $K$-vector spaces between the tangent space $T_E\cM_{n,d}$ and  $\HH^1(C,E\otimes E^\vee)$. Assume $\pi_{E,F}$ is  injective. Then:
\begin{itemize}
\item There is a local $K$-isomorphism for the \'etale topology between $(\cM_{n,d},E)$ and  $(\HH^1(C,E\otimes E^\vee),\bz)$ inducing  for every $1\le k\le l$ local $K$-isomorphisms for the \'etale topology between $(\cV_{k}(F),E)$
 and $(TC_E\cV_k(F),\bz)$.
\item Moreover,  $TC_E\cV_{k}(F)$ is the closed 
  subscheme defined by the ideal generated by the minors of size $l-k+1$ of the  $l'\times l$ matrix of linear forms on $\HH^1(C,E\otimes E^\vee)$ given by $\pi_{E,F}$. 
  \end{itemize}
 \end{theorem}
 
 The injectivity condition holds for example in the following cases.
  
 \begin{theorem}\label{thrmInj}
Assume that $C$ is generic among curves with same genus. If
\begin{itemize}
\item $\rm{(}$\cite{Gie}$\rm{)}$   $F=\cO_C$, or
\item $\rm{(}$\cite{TiBtw}$\rm{)}$   $F$ is generic among vector bundles with same rank and degree,
\end{itemize}
then the Petri map
$
\pi_{L,F}$
is injective for every $L\in \pic^d(C)$. 
\end{theorem}

\begin{rmk}\label{rmkBad}
Theorem \ref{thmGenToConeEt} can fail for $n=1$ and non-generic curves. There are curves $C$ of genus 4 such that the scheme $W^1_3$ is supported only on $L$ and such that $\pi_L$ is the matrix $\big(\begin{smallmatrix}
  x_1 & x_2\\
  x_2 & x_3
\end{smallmatrix}\big)$ of linear forms on $\HH^1(C,\cO_C)$, see \cite[p196, A.3 (ii)]{A+}. The 1-minors define a 1-dimensional linear subspace of the 4-dimensional affine space $\HH^1(C,\cO_C)$.
\end{rmk}

To prove Theorem \ref{thmGenToConeEt}, of which only the first was new,  its formal analog, Theorem \ref{thmGenToCone}, was proven first. Theorem \ref{thmGenToCone} is essentially due to
 \cite[Thm. 0.1]{Po} (a slight improvement presented here might already follow from \cite{Po}). The latter was proved using $A_\infty$-categories. Theorem \ref{thmGenToCone} was proved in \cite{theta} using deformation theory with cohomology constraints in terms of $\Linf$ pairs. This gave a quick, conceptual proof. Note that as a first step, this also reproves with $\Linf$-algebras the classical statement that $\cM_{n,d}$ is smooth of the right dimension at $E$, cf.  Theorem \ref{thmLePot}. We give in Section \ref{secTecCore} a second proof of Theorem \ref{thmGenToCone}. For this, we show in Theorem \ref{thrmLinf1Gen} that a  partial formality result similar to \cite[Thm. 3.1]{Po} holds for $\Linf$ pairs. This will require more background on $\Linf$ pairs which is the reason behind postponing this proof to the last section.


We state some consequences.
By  Theorem   \ref{thmGenToConeEt} the local models for Brill-Noether loci on curves at stable bundles with injective Petri maps are generic determinantal schemes.   Generic determinantal varieties have been abundantly studied. We  gathered in Section \ref{secDet}  some known results on singularities of generic determinantal schemes. By Theorem   \ref{thmGenToConeEt} information about singularities of the  local models passes automatically to  the Brill-Noether loci if the Petri map is injective. We find it convenient to assume  that
\be\label{eqAssu}
 l=h^0(C,E\otimes F)\le l'=h^1(C,E\otimes F).
\ee
Equivalently, $n\deg(F)-\rank (F)(n(g-1)-d)\le 0$, so the condition is independent of $E$. This is for simplicity only, since one can always reduce to this case, see \ref{subAss}. In case $F=\cO_C$, (\ref{eqAssu}) becomes
\be\label{eqAssuF}
n(g-1)-d\ge 0.
\ee
If $E=L$ is a line bundle and $F=\cO_C$ this is equivalent to $d<g$. The terminology from singularity theory used below is recalled in Section \ref{appSI}. 

\begin{thrm}\label{thrmLCThrk}\label{thrmLCT} In the setup of Theorem \ref{thmGenToConeEt}, let $K=\bC$,
let   $E\in \cV_k(F)\subset\cM_{n,d}$ with $1\le k\le l$, satisfying (\ref{eqAssu}), and such that $\pi_{E,F}$ is injective.
Then the following hold in a   Zariski open neighborhood of $E$ in $\cM_{n,d}$:

\begin{enumerate}

\item[(o)] $\cV_k(F)$  is variety with at most rational singularities, it has dimension $\rho_{n,d,k}(F)$, see Definition \ref{defBNP} {\color{hot}(4)},  and  the singular locus of $\cV_k(F)$ is $\cV_{k+1}(F)$.

\item[(oo)] The multiplicity of $\cV_k(F)$ at $E$ is 
$$
\prod_{i=0}^{k-1}\frac{(l'+i)!i!}{(l-k+i)!(l'-l+k+i)!}.
$$

\item[(i)] The multiplier ideal at $E$ of the pair $(\cM_{n,d},\cV_k(F))$ with coefficient $c\in \bR_{>0}$ is the intersection of symbolic powers 
$$
\bigcap_{j=0}^{l-k} {J_{k+j}}^{(\lfloor c(j+1)\rfloor+1-(k+j)(l'-l+k+j))}
$$
where $J_k$ is the ideal sheaf defining $\cV_k(F)$ in $\cM_{n,d}$. If $k=1$ the intersection simplifies to
$
{J_{1}}^{\lfloor c\rfloor +l-l'}.
$

\item[(ii)]  The  irreducible components of the scheme of $m$-jets of $\cV_k(F)$ centered at $E$ can be explicitly described.


\item[(iii)] If $k=1$ the local Bernstein-Sato polynomial at $E$ of the ideal defining  $\cV_1(F)$ in $\cM_{n,d}$ is
$$
\prod_{i=l'-l+1}^{l'} (s+i).
$$

\item[(iv)] If $d=n(g-1-\deg(F)/\rank (F))$, equivalently $l=l'$,   the topological zeta function at $E$ of the pair $(\cM_{n,d},\cV_k(F))$ is $$\prod_{\al\in \Omega}\frac{1}{1-\al^{-1}s}$$
where
$$
\Omega\subset\left\{ -\frac{l^2}{l-k+1}, -\frac{(l-1)^2}{l-k}, -\frac{(l-2)^2}{l-k-1},\ldots, - k^2 \right\}.
$$

\item[(v)] If $k=1$ and $d=n(g-1-\deg(F)/\rank (F))$, the  monodromy conjecture relating the local topological zeta function with the local Bernstein-Sato polynomial, see Conjecture \ref{conjMC}, holds   for the generalized theta divisor $\cV_1(F)\subset\cM_{n,d}.$

\item[(vi)] Consider $f:Y\to \cM_{n,d}$ the composition of blowups of (strict transforms) of $\cV_l(F)$, $\cV_{l-1}(F)$, $\cV_{l-2}(F)$, $\ldots$, $\cV_k(F)$, in this order. Then:

\begin{itemize}
\item At each stage this is the blowup of a smooth center. 
\item The composition $f$ is a log resolution of $(\cM_{n,d},\cV_{k}(F))$. 

\item The pullback of the ideal sheaf defining $\cV_{k}(F)$ is $\cO_Y(-\sum_{i=0}^{l-k}(l-k+1-i)E_i)$, where $E_i$ is the (strict transform of the) divisor introduced by blowing up the (strict transform of) $\cV_{l-i}(F)$.
\end{itemize}

\item[(vii)] The stratification of $\cV_k(F)$ given by $\cV_{t}(F)\setminus \cV_{t+1}(F)$ with $k\le t$ is a Whitney stratification, and the local Euler obstruction at $E$ of $\cV_k(F)$  is $\binom{l}{l-k}$.

\item[(viii)] The log canonical threshold of $(\cM_{n,d},\cV_k(F))$ at $E$ is $$
\min\left\{\frac{(l-i)(l'-i)}{l -k+1-i}\mid i=0,\ldots ,l-k\right\}.
$$

\item[(ix)] The simple holonomic $\cD$-module composition factors of the local cohomology modules $$\cH^j_{\cV_{k}(F)}(\cO_{M_{n,d}})$$ are known, each is equal to the intersection homology module $\cL(\cV_{k'}(F),M_{n,d})$ for some $k'$ with $k\le k'\le l$, in which case the weight  with respect to the weight filtration is also known. The Hodge filtrations on $\cH^j_{\cV_{k}(F)}(\cO_{M_{n,d}})$ and $\cL(\cV_{k}(F),M_{n,d})$ are combinatorially determined, as well as their generating levels. 
If $l=l'$, the same information is available for $\cO_{\cM_{n,d}}(*\cV_1(F))$, and for $p\ge 0$ the $p$-Hodge ideal of $\cV_1(F)$ is 
$$
I_p(\cV_1(F)) = \bigcap_{k=1}^{l-1} J_{k+1}^{\left(k(p-1)-\binom{k}{2}\right)}.
$$

\item[(x)] If $l=l'$, the minimal discrepancies of $\cV_k(F)$ along $\cV_{k+1}(F)$ and, respectively, along a point $E'\in \cV_{k'}(F)\setminus \cV_{k'+1}(F)$ with $k\le k'\le l$ are:
$$
\mld(\cV_{k+1}(F);\cV_k(F)) = k+1,\quad \mld(E';\cV_k(F)) = l^2-kk'.
$$
\end{enumerate}
 \end{thrm}

Parts {\color{hot} (o), (iii)-(viii), (x)} are from  \cite{theta}. 
Parts  {\color{hot} (o)-(oo)} recover and state in a slightly more general way older results  due to  \cite{Ke, AC, A+, TiBtw, CT}, cf. Theorem \ref{thrmLCTa}. 
 Part {\color{hot}(vi)}  recovers \cite[Thm. 3.3]{Mul} which says that for $n=1$, $F=\cO_C$, and $d=g-2$, this blowup process is an embedded resolution (without checking the simple normal crossings condition) of $(\pic^{g-2}(C),W^0_{g-2})$. Part {\color{hot}(viii)} is due to  \cite{Zhu} for $n=1$ and $F=\cO_C$, cf. Theorem \ref{thrmLCTa} {\color{hot} (iv)}.

\subs{\bf Beyond genericity.} If the Petri map is not injective there is less known about the singularities of the Brill-Noether loci. 
 Theorem \ref{thrmLCTb}  collects some known results  including  (extensions of) the Riemann-Kempf singularity theorem due to \cite{Ke, A+, La, Li, CT}.  Next result generalizes Theorem \ref{thrmLCTb} {\color{hot}(iii)}, has the same proof, but
cannot be found in the literature:
 
 \begin{thrm}\label{thmGG2} 
For any curve $C$ and assuming (\ref{eqAssuF}), $\cV_1\subset\cM_{n,d}$ has rational singularities at every point if non-empty.
\end{thrm}

Next, we would like to pose some questions regarding Brill-Noether loci in absence of genericity. When $F=\cO_C$, the Petri map $\pi_E$ is 1-generic, cf. Lemma \ref{lem1genP}. Here {\it 1-generic} means that the multiplication of two non-zero vectors is non-zero.    More generally, {\it $k$-generic} means by definition that the kernel of the multiplication map does not contain a sum of $\le k$ pure tensors. Then injective and 1-generic are the two extremes of being $k$-generic. The associated matrix of linear forms  is also called $k$-generic if the multiplication map is $k$-generic. This terminology is due to \cite{Eis}.

We  gathered in Section \ref{secDet}  some known results on singularities of $k$-generic determinantal schemes. 
Organizing the information in this way, we realized that some of it had escaped attention in the last decades. For example, the fact that determinantal varieties of  Hankel matrices have rational singularities, which is also the title of \cite{Con}, follows  from Kempf's method of well-presented morphisms and rational resolutions \cite{Ke}, see Theorem \ref{thrmHankel}. In  general, without further specialization to specific situations, the properties of determinantal schemes of 1-generic matrices of linear forms depend on the matrix and not only on its size $b\times a$ and the size $m$ of the minors  one uses as ideal generators. However it seems that,  keeping $a,b,m$ fixed, the $k$-generic determinantal schemes become more singular as $k$ decreases. Moreover, Hankel matrices seem to be the most special 1-generic  matrices. We therefore pose the following question about log canonical thresholds:

\begin{que}\label{queKgen}
Let $0<m\le a\le b$ and $M=\bA^{ab}$ be the space of $b\times a$ matrices over an algebraically closed field $K$. Let $N,N'\subset M$ be two  linear subspaces, and let $N_m, N'_m$ be the natural closed subschemes parametrizing the matrices of rank $\le a-m$ in $N, N'$, respectively. 
\begin{itemize}
\item
If $N$ is $k$-generic  for some $1\le k\le a$ then 
$$
\min\left\{\frac{(a-i)(b-i)}{a -m+1-i}\mid i=0,\ldots ,a-m\right\}\ge 
 \lct(N,N_m)\ge  \left\{
\begin{array}{ll}
1 & \text{ if }a=b\text{ and }m=1,\\
1+\displaystyle{\frac{b+m-2}{a-m+1}} & \text{ if }a<b \quad \quad ?
\end{array}
\right.$$ 

\item
If $N$ is $k$-generic and $N'$ is $k'$-generic for some $1\le k'<k\le a$, then 
$$
\lct(N,N_m)\ge 
 \lct(N',N'_m) \quad ?$$ 
\end{itemize}
\end{que}

The upper bound is chosen to be exactly the log canonical threshold of generic determinantal schemes, whereas the
lower bound is chosen to be exactly the log canonical threshold of Hankel determinantal schemes, cf. Theorem \ref{thrmGen} {\color{hot} (iv)} and Theorem \ref{thrmHankel} {\color{hot} (viii)}. 

\begin{rmk}
In the case when $a=b$ and $m=1$ the question is true, and all numbers are equal to 1,  since in this case $N_1$, $N'_1$ are hypersurfaces  with at most rational singularities, by Theorem \ref{thmkGen} essentially due to Kempf. 
In the case $b=a+1$ and $m=1$, the question asks if $\lct(N,N_1)=2$ always.
\end{rmk}

The question is relevant for Brill-Noether loci because of:

\begin{prop}\label{propXWQ}
For any curve $C$ and $E\in\cV_k(F)\subset\cM_{n,d}$, there are inequalities of local log canonical thresholds
$$
\lct_E(\cM_{n,d},\cV_k(F)) \ge \lct_\bz(\HH^1(C,E\otimes E^\vee),TC_E\cV_k(F))\ge \lct_\bz(N,N_k)
$$
where $N=\im(\pi_E)^\vee$ and $N_k$ is the closed subscheme cut out by the $(l-k+1)$-minors of the $l'\times l$ matrix of linear forms determined by the  Petri map $\pi_{E,F}$.
\end{prop}

The second inequality in this proposition is also proven with the $\Linf$ technique.

In the case of line bundles, Hankel matrices arise from Petri maps $\pi_L$ of  line bundles on hyperelliptic curves by Proposition \ref{propHyl}.

If the Petri map is not injective the following questions arise:

\begin{que}$\;$
 If $n=1$ and $d<g$ is there a class of curves $C$ for which Theorem \ref{thmGenToConeEt} for $W^{k-1}_d$ is true with $L$ only satisfying that $\pi_L$ is $k_0$-generic and $1\le k \le k_0\le l$?
\end{que}

If $k_0=1$ we will see below that there is strong evidence that the class of hyperelliptic curves provides a positive answer. We pose a simpler form of the question for them:

\begin{que}\label{queHypo} Does Theorem \ref{thmGenToConeEt} hold for $W^0_d$ for every hyperelliptic curve $C$?
\end{que}

\begin{rmk} 
For an arbitrary curve $C$  one does not necessarily have an isomorphism of analytic germs $(W^0_d,L)\simeq (TC_LW^0_d,\bz)$.  Consider $C$ and $L$ as in Remark \ref{rmkBad}. Then the singular locus of $TC_LW^0_3$ is 1-dimensional whereas the singular locus of $W^0_3$ is the reduced support of $W^1_3$ by \cite[IV, Cor. 4.5]{A+}, hence 0-dimensional. We thank C. Schnell for this remark.
\end{rmk}

Questions \ref{queKgen} and \ref{queHypo} suggest that, from the point of view of the log canonical thresholds, $W^0_d$ are the most singular for hyperelliptic curves and the least singular for generic curves.

We show next that there is compelling evidence for a positive answer to Question \ref{queHypo}. Note that a positive answer to Question \ref{queHypo} would allow one to apply next lemma to hyperelliptic curves:

\begin{lemma}\label{thmW0d} Let $C$ be a smooth projective curve over $K=\bC$, and let $L\in\pic^d(C)$ with $0\neq h^0(L)h^1(L)$, and $d<g$. Suppose that Theorem \ref{thmGenToConeEt} holds for $W^0_d$ at $L$. Then there are equalities of  log canonical thresholds and minimal exponents (and other local analytic invariants)
$$
\lct_L(\pic^d(C),W^0_d) = \lct(N,N_1)\quad\text{ and } \quad\al_L(\pic^d(C),W^0_d) = \al(N,N_1),
$$
where $N=\im(\pi_L)^\vee$ and $N_1$ is the closed subscheme cut out by the maximal minors of the  matrix of linear forms determined by the  Petri map $\pi_{L}$.
\end{lemma}

If $C$ is a hyperelliptic curve and $d<g$, it is known that 
 $W^r_d$ is an irreducible scheme of dimension $d-2r$,  $Sing(W^r_d)= (W^{r+1}_d)_{red}$, and $(W^r_d)_{red}\simeq W^0_{d-2r}$, see Proposition \ref{propHell}. This is compatible with Hankel matrices, see Theorem \ref{thrmHankel}, and these properties would follow from a positive answer to Question \ref{queHypo}.  Additional consequences would be:

\begin{prop}\label{thmHyE} Suppose Question \ref{queHypo} has a positive answer. Let $C$ be a smooth projective hyperelliptic curve over $K=\bC$, and $d<g\ge 2$. Let $L\in W^{k-1}_d$ with $1\le k\le l$. Then:

\begin{enumerate}[(i)]

\item\label{iti0}{\rm{(}\cite{SY}\rm{)}} $W^{k-1}_d$ is reduced. Hence $W^{k-1}_d\simeq W^0_{d-2(k-1)}$.

\item\label{iti} Theorem \ref{thmGenToConeEt} holds for arbitrary hyperelliptic $C$ and all $W^{k-1}_d$.

\item\label{iti2} If  $k\le m\le l$ then the multiplicity of $W^{k-1}_d$ at any point in $W^{m-1}_d\setminus W^{m}_d$ is $$\binom{g-d-2+m+k}{m-k}.$$
Hence if $d=g-1$ then $W^{k-1}_{g-1}\setminus W^{k}_{g-1}$ is the locus of points of $W^0_d$ with multiplicity exactly $k$, cf. \cite{SY}.

\item\label{iti3} Consider $f_{l-k}:Y_{l-k}\to \pic^d(C)$ the composition of blowups of (strict transforms of) $W^{l-1}_d$, $W^{l-2}_d$, $\ldots$, $W^{k-1}_d$, in this order. At each stage this is the blowup of a smooth center, such that $f_{l-k}$ is a log resolution of $(\pic^d(C),W^{k-1}_d)$, cf. \cite{SY} for $d=g-1$ and $k=1$.

\item\label{iti4}{\rm{(}\cite{SY}\rm{)}} If $d=g-1$ then $$f^*_{l-1}(W^0_{g-1})= \sum_{i=0}^{l-1}(l-i)E_i$$
 where $E_i$ is the (strict transform of the) divisor introduced by blowing up the (strict transform of) $W^{l-i-1}_{g-1}$. 
 
\item\label{iti5}  $$
\lct _L(\pic^d(C), W^{k-1}_d)=
\left\{
\begin{array}{ll}
1 & \text{ if }d=g-1\, (\text{that is, }l=l')\text{ and }k=1,\\
1+\displaystyle{\frac{l'+k-2}{l-k+1}} & \text{ if }d\neq g-1\, (\text{that is, }l<l') \text{ or }k\neq 1.
\end{array}
\right.
$$

\item \label{iti6}{\rm{(}\cite{SY2}\rm{)}} If $d=g-1$ and $l>1$, the minimal exponent of the theta divisor is $$\al_L(\pic^{g-1}(C), W^{0}_{g-1})=3/2.$$
\end{enumerate}
\end{prop}
 
The properties citing \cite{SY} are already known to hold unconditionally of a positive answer to Question \ref{queHypo}. The proofs in \cite{SY} do not go through Hankel matrices, apart from (\ref{iti0}). We regard this as compelling evidence that Hankel determinantal varieties are the local \'etale models for Brill-Noether loci of hyperelliptic curves. It would be interesting to check if (\ref{iti2}) and (\ref{iti5}) also hold unconditionally for all hyperelliptic curves.

\subs{\bf Organization of Part \ref{part1}.} 
In Section \ref{subContr} we review how to obtain Theorem \ref{thmGenToCone} describing locally formally the  Brill-Noether loci from  the black box of Section \ref{secDFT}. 
In Section \ref{secApplBN} we address Theorem \ref{thmGenToConeEt}, Theorem \ref{thrmLCThrk},  Proposition \ref{propXWQ},  Lemma \ref{thmW0d}, and Proposition \ref{thmHyE}. Part \ref{part1}  ends with three short survey sections supporting the previous sections. 
In Section \ref{secDet} we collect some known facts about the singularities of spaces of $k$-generic matrices matrices. 
In Section \ref{apxBN} we  collect some known facts about the singularities of Brill-Noether loci and prove  Theorem \ref{thmGG2}. Section \ref{appSI}  recalls some terminology and facts from singularity theory.

\section{The controlling pairs}\label{subContr}


The dgl pairs controlling locally the Brill-Noether loci are given by the following, see \cite{theta}:

\begin{prop}\label{rmkPair} 
Let $E, F$ be two vector bundles over a smooth projective variety $X$ over an algebraically closed field $K$. Assume that $E$ is stable with respect to a fixed polarization. Then the deformations of $E$ with cohomology constraints $h^i(X,E\otimes F)\ge k$ are controlled by the dgl pair $( R\Gamma(X, \enmo  (E)), R\Gamma(X,E\otimes F))$.
\end{prop}

Applying Theorem \ref{thmTgxxy} and homotopy transfer, that is Theorem \ref{thrmTLP}, to the pair in Proposition \ref{rmkPair}, one obtains the controlling $\Linf$ pairs and the tangent spaces to the Brill-Noether loci, cf. \cite{theta}:

\begin{prop}\label{propOurLePot}
Let $E, F$ be two vector bundles over a smooth projective variety $X$ over an algebraically closed field $K$. Assume that $E$ is stable with respect to a fixed polarization. Then:

\begin{enumerate}[(1)]

\item The deformations of $E$ with cohomology constraints $h^i(E\otimes F)\ge k$ are controlled by the $\Linf$ pair $(\HH^\ubul(X,\enmo(E)), \HH^\ubul(X,E\otimes F))$.

\item If $\cM$ denotes the moduli space of stable vector bundles on $X$ of same Hilbert polynomial as $E$,  $\cV^i_k=\{E'\in\cM\mid h^i(E\otimes F)\ge k\}$ denote the cohomology jump loci endowed with the natural closed subscheme structure, and $h^i=h^i(E\otimes F)$, then the Zariski tangent spaces at $E$ to
$$\cV^i_0=\cM\supset\ldots\supset\cV^i_k\supset\ldots \supset  \cV^i_{h_i+1} (=\emptyset \text{ around } E)$$ 
are: the full Zariski tangent space $T_E\cM=\HH^1(X,\enmo(E))$ if $k<h_i$; empty if $k>h_i$; and if $k=h_i$, equal to the kernel of the linear map
$$
\HH^1(X,\enmo(E))\ra \bigoplus_{j=i-1,i}\Hom(\HH^{j}(X,E\otimes F),\HH^{j+1}(X,E\otimes F))
$$
induced from the natural multiplication maps
$
\HH^1(X,\enmo(E))\otimes \HH^{j}(X,E\otimes F)\ra \HH^{j+1}(X,E\otimes F).
$
\end{enumerate}

\end{prop}

Proposition \ref{propOurLePot} is classical for Brill-Noether loci of line bundles when $X$ is a curve, cf. \cite[IV, Prop. 4.2]{A+}, and one can show it implies Theorem \ref{thmDimSing} {\color{hot}(i)} below.

With these preliminaries  we can now state the formal neighborhood version of Theorem \ref{thmGenToConeEt}. 

\begin{theorem}\label{thmGenToCone} Let $E, F$ be as in Theorem \ref{thmGenToConeEt}. There is a canonical isomorphism of $K$-vector spaces between the tangent space $T_E\cM_{n,d}$ and  $\HH^1(C,E\otimes E^\vee)$. If $\pi_{E,F}$ is  injective, there is an isomorphism between the formal neighborhood of $E$ in $\cM_{n,d}$ and the formal neighborhood of the origin in $\HH^1(C,E\otimes E^\vee)$ inducing  for every $1\le k\le l$ isomorphisms between:

\begin{itemize}
\item the formal neighborhood of 
$
\cV_{k}(F)$
 at $E$ in $\cM_{n,d}$,
 \item the formal neighborhood at the vertex of the tangent cone $TC_E\cV_k(F)$  in the tangent space $T_E\cM_{n,d}$.
  \end{itemize}
Moreover,  $TC_E\cV_{k}(F)$ is the closed 
  subscheme defined by the ideal generated by the minors of size $l-k+1$ of the  $l'\times l$ matrix of linear forms on $\HH^1(C,E\otimes E^\vee)$ given by $\pi_{E,F}$. 
 \end{theorem}
 

The proof of Theorem \ref{thmGenToCone} in \cite{theta} applied the following intermediate result to the controlling pair from Proposition \ref{propOurLePot}:

\begin{thrm}\label{thm1GDik}  Let $(M,V)$ be an $L_\infty$ algebra together with a module, both of finite dimension over a field $K$ of characteristic zero, such that:
\begin{itemize} 
\item $M^i=0$ and $V^i=0$ for $i\neq 0,1$,
\item the differentials on $M$ and $V$ are zero, 
\item  the linear map $\pi:V^0\otimes (V^1)^\vee\to (M^1)^\vee$ induced from the multiplication map $m_2:M^1\otimes V^0\to V^1$ is injective.
\end{itemize} 

Assume that the $\Linf$ algebra $M$ is  obtained as a transferred structure from a dgla $C$ with $\iota:M=HC\subset C$ as in Theorem \ref{thmTTL},  
and  $[\iota(M^0),C]=0$.   Let $\bz\in M^1$ denote the origin. For every $k\in \bN$ let $\cV_k\subset M^1$ be
 the closed subscheme   defined by minors of size $\dim V^0 -k+1$ of the matrix of linear forms on $M^1$ determined by $\pi$. 
 Then there is a canonical isomorphism of vector spaces $T\Def(M)= M^1$ and  an isomorphism of functors $\Def(M)\simeq \widehat{(M^1)}_\bz$ compatible with each other, inducing
 isomorphisms of functors  $\Def^0_k(M,V)\simeq \widehat{(\cV_k)}_\bz$ for every $k$.  
 \end{thrm} 
 \begin{proof} We recall the proof from \cite{theta} since in Section \ref{secTecCore} we  give another proof and we would like to point to some steps from here.
 Denote by $l=\{l_n\}_{n\ge 1}$ the $\Linf$ algebra structure on $M$, and by $m=\{m_n\}_{n\ge 1}$ the $\Linf$ module structure on $V$. We have $l_1=0$ and $m_1=0$. Let $\omega\in  M^1$. Since $l_n$ has degree $2-n$, $l_n(\omega^{\otimes n})$ is in $M^2=0$.  Hence $M^1\otimes \mm_A=\MC_M(A)$ for all $A\in\Art$, with the Maurer-Cartan set as in Definition \ref{defnMC}. This gives $T\Def(M)= M^1$, cf. Theorem \ref{propTDef}. The assumption on $M^0$  implies that no two elements in $M^1\otimes \mm_A$ are homotopy equivalent by  Lemma \ref{MCnequi}. Thus
$
\Def(M)\simeq (\widehat{M^1})_{\bz}. 
$
Since there is no homotopy equivalence to mod out by, we also have
\begin{equation}\label{eqLJIp1}
\Def^0_k(M,V;A)=\{\omega\in  M^1\otimes \mm_A\mid   J^0_k(V\otimes A, d_\omega)=0 \}
\end{equation}
where
$$d_\omega:V^0\otimes A\to V^1\otimes A,\quad
d_\omega(\_):=\sum_{n\ge 1}\frac{1}{n!}m_{n+1}^A(\omega^{\otimes n}\otimes\_),
$$
since $V$ is concentrated in degrees 0,1 and $m_1=0$, see Definition \ref{defDik}. It will be slightly more convenient to work with the graded-symmetric version of the $\Linf$ pair structure; by Remarks \ref{rmkSymM0} and \ref{rmkSymM} this amounts to changing $\omega^{\otimes n}$ to its symmetric version $\omega^{\vee n}$ in the formula for $d_{\omega}$ if we keep denoting by $\{m_n\}_n$ the graded-symmetric version of the $\Linf$ module structure on $V$.

We construct now a universal matrix $d_{univ}$ with entries in the completion $\widehat S$ at the maximal ideal at $\bz\in M^1$ of the symmetric algebra $S$ of $(M^1)^\vee$, such that $d_{univ}$ gives all $d_\omega$ for all  $A$ and $\omega$ as above. Let $s=\dim M^1$.
Fix  a basis $e_1,\ldots, e_s$  of the vector space $M^1$.  Let $x_1,\ldots , x_s$ be the dual basis, so that $S=K[x_1,\ldots ,x_s]$ and $\widehat{S}=K\llbracket x_1,\ldots ,x_s \rrbracket$. Let
$
\omega_{univ}=\sum_{i=1}^se_i\otimes x_i\in M^1\otimes S.
$
Define the morphism of free $\widehat{S}$-modules
\be\label{eqDUNI}
d_{univ}:V^0\otimes \widehat{S} \to V^1\otimes \widehat{S},\quad \sigma\otimes 1\mapsto  \sum_{n\geq 1}\frac{1}{n!} (m_{n+1}\otimes\id_{\widehat{S}})((\omega_{univ})^{\vee n}\otimes(\sigma\otimes 1)).
\ee
Fixing bases for $V^0, V^1$, we write $d_{univ}$ as a matrix with entries in $\widehat{S}$.
By construction we have for all $k$ canonical isomorphisms of subfunctors
$
\Def^0_k(M,V)= {\rm{Spf}}(\widehat{S}/J^0_k(d_{univ}))
$
compatible with the inclusion of subfunctors for $k\le k'$.

The  matrix $B$ formed by the linear parts of the entries of $d_{univ}$ is by construction the matrix of linear forms on $M^1$ determined by $\pi$ and the above vector space bases. 
By the injectivity assumption on $\pi$, the entries of $B$ are linearly independent. Hence we can find an isomorphism of $\widehat{S}$ such that 
$d_{univ}$ becomes $B$. This implies the claim since $\widehat{(\cV_k)}_\bz$ is defined by the ideal $J^0_k(B)\subset\widehat{S}$. 
 \end{proof}


One can apply Theorem \ref{thm1GDik} to the context of Theorem \ref{thmGenToCone}. It is here that stability of the vector bundle $E$ becomes crucial. It is needed to guarantee that all conditions from Theorem \ref{thm1GDik} are met. Moreover $\pi$ corresponds to the Petri map $\pi_{E,F}$ and hence its entries are linearly independent linear forms. Thus one can apply a change of formal coordinates to obtain that $J^0_k(\pi_{E,F})$ define locally formally the twisted Brill-Noether loci. This finishes the proof of  Theorem \ref{thmGenToCone}.

In Section \ref{secTecCore} we  give a  proof of Theorem \ref{thmGenToCone} different that the one in \cite{theta}, similar to \cite{Po} but requiring more $\Linf$  background.

\section{Proofs of the applications to  Brill-Noether loci}\label{secApplBN}

The goal of this section is to recall the main step left in the proof of Theorem \ref{thmGenToConeEt}, to prove the claims from Theorem \ref{thrmLCThrk} not stated in \cite{theta},  Proposition \ref{propXWQ}, Lemma \ref{thmW0d}, and Proposition \ref{thmHyE}. Let $K$ be a field  of characteristic zero.  Artin showed:

\begin{thrm}\label{thrmArtin}{\rm{(}\cite{ar}, \cite[Cor. 2.6]{ar1}\rm{)}}  Let $X_1$ and $X_2$ be two $K$-schemes of finite type, and let $x_i\in X_i$ be two points. If the  formal neighborhoods $\widehat{(X_i)}_{x_i}$
are $K$-isomorphic then:
\begin{itemize}
\item $(X_i,x_i)$ are locally isomorphic for the \'etale topology, that is, there exist a $K$-scheme of finite type $X'$, a point $x'\in X'$, and  \'etale maps $X_1\leftarrow X'\to X_2$ sending $x_1\mapsfrom x'\mapsto x_2$, and inducing isomorphisms of residue fields of $x_1,x',x_2$;

\item $(X_i,x_i)$ are locally analytic isomorphic if $K=\bC$. 
 \end{itemize}
\end{thrm}

The key to passing from the formal neighborhood in Theorem \ref{thmGenToCone} to the local \'etale neighborhood in Theorem \ref{thmGenToConeEt} is the following version of Artin's algebraization theorem.

\begin{prop}{\rm{(}\cite[Prop. 3.2]{theta}\rm{)}}\label{propTower}
Let $X$ be a smooth  $K$-variety, $X\supset Y_1\supset Y_2\supset \ldots Y_m$  closed subschemes, and $x\in Y_m$ a point. Let $T=T_xX$, $C_i=TC_xY_i$, and $0\in C_m$ be the vertex. Suppose there exists a $K$-isomorphism of formal neighborhoods $\widehat{X}_x\simeq \wh{T}_0$ inducing  isomorphisms

$$
\begin{tikzcd}
\wh{X}_x \arrow[d,"\simeq"] \arrow[r, hookleftarrow]& \wh{Y}_{1,x} \arrow[d,"\simeq"] \arrow[r, hookleftarrow]& \wh{Y}_{2,x} \arrow[d,"\simeq"] \arrow[r, hookleftarrow] &  \ldots \arrow[r, hookleftarrow]& \wh{Y}_{m,x} \arrow[d,"\simeq"]\\
\wh{T}_0 \arrow[r, hookleftarrow]& \wh{C}_{1,0} \arrow[r, hookleftarrow]& \wh{C}_{2,0} \arrow[r, hookleftarrow]&  \ldots \arrow[r, hookleftarrow]& \wh{C}_{m,0}.\\
\end{tikzcd}
$$

Then:
\begin{itemize}
\item There exists a local isomorphism for the \'etale topology $(X,x)\simeq (T,0)$ inducing  local isomorphisms for the \'etale topology 

\be\label{eqdiT}
\begin{tikzcd}
(X,x) \arrow[d,"\simeq"] \arrow[r, hookleftarrow]& (Y_1,x) \arrow[d,"\simeq"] \arrow[r, hookleftarrow]& (Y_2,x) \arrow[d,"\simeq"] \arrow[r, hookleftarrow] &  \ldots \arrow[r, hookleftarrow]& (Y_m,x) \arrow[d,"\simeq"]\\
(T,0) \arrow[r, hookleftarrow]& (C_1,0) \arrow[r, hookleftarrow]& (C_2,0) \arrow[r, hookleftarrow]&  \ldots \arrow[r, hookleftarrow]&(C_m,0).\\
\end{tikzcd}
\ee

\item If $K=\bC$, there exists a local analytic isomorphism  $(X,x)\simeq (T,0)$ inducing  local analytic isomorphisms   
in the  diagram (\ref{eqdiT}).
\end{itemize}
\end{prop}

\subs{\bf Proof of Theorem \ref{thmGenToConeEt}.} The direct application of Proposition \ref{propTower} and Theorem \ref{thmGenToCone} gives Theorem \ref{thmGenToConeEt}. $\hfill\Box$

\subs{\bf Proof of Theorem \ref{thrmLCThrk}.} We proceed as in \cite{theta}.  By Theorem \ref{thmGenToConeEt}, there exists a tower of cartesian diagrams

\be\label{eqBlubli}
\begin{tikzcd}
(\cM,E) \arrow[d,hookleftarrow]& (X,x) \arrow[l] \arrow[r] \arrow[d,hookleftarrow]& (M ,\bz) = (T_E\cM,\bz)\arrow[d,hookleftarrow]\\
(\cV_1,E) \arrow[d,hookleftarrow]& (X_1,x) \arrow[l] \arrow[r] \arrow[d,hookleftarrow]& (M_1,\bz) = (TC_E\cV_1,\bz) \arrow[d,hookleftarrow]\\
(\cV_2,E) \arrow[d,hookleftarrow]& (X_2,x) \arrow[l] \arrow[r] \arrow[d,hookleftarrow]& (M_2,\bz) =(TC_E\cV_2,\bz)\arrow[d,hookleftarrow]\\
\vdots \arrow[d,hookleftarrow]&\vdots \arrow[d,hookleftarrow]& \vdots \arrow[d,hookleftarrow]\\
(\cV_l,E) & (X_l,x) \arrow[l] \arrow[r] & (M_l,\bz)=(TC_E\cV_l,\bz) \\
\end{tikzcd}
\ee
where: the horizontal maps  are \'etale, sending $E\mapsfrom x\mapsto \bz$, inducing isomorphisms of residue fields at $E,x,\bz$; the vertical maps are closed embeddings of subschemes; $\cM=\cM_{n,d}$,
$\cV_k=\cV_k(F)$, $M=\HH^1(C,E\otimes E^\vee)$, $\bz$ is the origin; $M_k$ is the closed subscheme defined by the ideal generated by the minors of size $l-k+1$ of matrix of linear forms on $\HH^1(C,E\otimes E^\vee) $ determined by $\pi_{E,F}$. 
By the injectivity of $\pi_{E,L}$ this matrix is generic  of size $l'\times l$, cf. Definition \ref{defnGen}.

To simplify notation, we denote by the same symbols, and work with them so from now on, the restriction of the diagram to two Zariski open neighborhoods of $E$ and $x$, respectively in $\cM_{n,d}$ and $X$, respectively. By shrinking these open neighborhoods, we can and will assume that $\cM, \cV_k, X, X_k$ are connected.

Parts {(o), (iii)-(viii), (x)} are proved in \cite{theta} using the corresponding features of the local models from Theorem \ref{thrmGen}.

(oo) Follows from Theorem \ref{thrmGen} {\color{hot}(ii)}, with $a=l$, $b=l'$. 
Note that $M$ and $M_k$ here are the same as $\bA^{ab}$ and $M_k$, respectively, from Theorem \ref{thrmGen} up to the product with an affine space of dimension equal to $h^1(E\otimes E^\vee) - ll'$. In any case the codimension of $M_k$ here agrees with the codimension of $M_k$ from Theorem \ref{thrmGen}.

(ii) The scheme of $m$-jets of a variety passing through a fixed point only depends on the formal neighborhood of that point. Hence the problem is reduced to describing the  scheme of $m$-jets of $M_k$ passing through $\bz$. This was done in \cite{Roi} using pre-partitions.



 (i) We use the definition in terms of the log resolution from (vi) of multiplier ideals, see Definition \ref{defnLR}. To compute that relative canonical divisor of $f:Y\to \cM$ we use the formula for the codimension of $\cV_k$ from (o). We obtain the multiplier ideal at $E$ of the pair $(\cM,\cV_k)$ with coefficient $c\in \bR_{>0}$  is
 $$
f_*\cO_Y\left(A-\sum_{j=0}^{l-k} (\lfloor c(j+1)\rfloor+1-(k+j)(l'-l+k+j))E_{l-k-j}\right)=\star
$$
where $A$ is an effective exceptional divisor such that $K_{Y/\cM}-A$ is also effective. Hence 
$$
\star=f_*\cO_Y\left(-\sum_{j=0}^{l-k} (\lfloor c(j+1)\rfloor+1-(k+j)(l'-l+k+j))E_{l-k-j}\right)
$$
which gives the claim, cf. Theorem \ref{thrmGen} {\color{hot}(v)}.

(ix)  Consider a triplet $(X,Z,x)$ consisting of a smooth affine algebraic variety, a closed subvariety together with a point in it, and  $(X',Z',x')$  another such triplet such that there exists an analytic isomorphism $(X,Z,x)\simeq (X',Z',x')$, that is, the germs of $Z$ and $Z'$ at $x$ and $x'$, respectively, are embedded analytically equivalent. Then under this isomorphism the analytic mixed Hodge modules $j_*\bQ_{X\setminus Z}^H$ (if $Z$ has codimension one), $\HH^k(i_*i^!\bQ_X^H)$, $\IC_Z\bQ^H$ determined by $(X,Z,x)$ in a small analytic neighborhood of $x$ correspond to their obvious counterparts in a small analytic neighborhood of $(X',Z',x')$, where $i:Z\to X$,  $j:X\setminus Z\to Z$ are the natural closed and, respectively, open embeddings.  Since the two triplets are algebraic, the algebraic mixed Hodge modules $j'_*\bQ_{X'\setminus Z'}^H$, $\HH^k(i'_*(i')^!\bQ_{X'}^H)$, $\IC_{Z'}\bQ^H$ determine the algebraic mixed Hodge modules $j_*\bQ_{X\setminus Z}^H$, $\HH^k(i_*i^!\bQ_X^H)$, $\IC_Z\bQ^H$ not only in a small analytic neighborhood of $x$ but also in a small Zariski open neighborhood.

We apply this to our case, where we have an analytic isomorphism $(\cM,\cV_k,E)\simeq (M,M_k,\bz)$ and information about the three algebraic mixed Hodge modules  from above is available from Theorem \ref{thrmGen} (\ref{itPr1}), (\ref{itPr2}). Recall that $(M,M_k)$ here is the product of $\bA^{m}$  with $(M,M_k)$ from Theorem \ref{thrmGen}, where $m=\dim\cM-ll$. Thus one needs to apply $(\_)\boxtimes \bQ_{\bA^m}^H[m]$ to the mixed Hodge modules involved in Theorem \ref{thrmGen} (\ref{itPr1}), (\ref{itPr2}) to obtain those for $(\cM,\cV_k)$ in a suitable open Zariski neighborhood of $E$. This does not shift the $F$ filtration since we consider the underlying left $\cD$-modules, cf. Remark \ref{rmkDm1} and Remark \ref{rmkDm2}. $\hfill\Box$


\subs{\bf Proof of Proposition \ref{propXWQ}.} Log canonical thresholds cannot increase under specialization \cite[9.5.41]{Laz}. 
By specialization to the tangent cone, cf. Section \ref{appSI}, we obtain $
\lct_E(\cM,\cV_k) \ge \lct_\bz(\HH^1(C,E\otimes E^\vee),TC_E\cV_k)$. 
For the second inequality, we have seen in the proof of Theorem \ref{thmGenToCone}, more precisely in the proof of Theorem \ref{thm1GDik}, that the ideal defining the tangent cone $TC_E\cV_k$ in the completion at $E$ contains the ideal generated by the $(l-k+1)$-minors of the matrix  of linear forms determined by the Petri map $\pi_E$. The span $N$ of the entries is the dual of the vector subspace $\im(\pi_E)$ of $\HH^1(C,E\otimes E^\vee)$. Since bigger ideal implies bigger log canonical threshold cf. \cite[Prop. 9.2.31]{Laz}, we obtain the second inequality $\lct_\bz(\HH^1(C,E\otimes E^\vee),TC_E\cV_k)\ge \lct_\bz(N,N_k)$.
$\hfill\Box$

\subs{\bf Proof of Lemma \ref{thmW0d}.}
 More generally, any invariant of the local \'etale embedded structure stays the same. Since $N_1$ is a cone in $N$, the local lct and the local minimal exponent at the origin are the same as the global counterparts.
$\hfill\Box$

\subs{\bf Proof of Proposition \ref{thmHyE}.} Before we can prove Proposition \ref{thmHyE} we need a preliminary result.

\begin{lemma}\label{lemRexd} With the conditions as in Proposition \ref{thmHyE},
\begin{enumerate}[(a)]
\item\label{jiti0} $Sing(W^{k-1}_d)=Sing((W^{k-1}_d)_{red})$.

\item\label{jiti} There is a local $K$-isomorphism for the \'etale topology between $(\pic^d(C),L)$ and  $(\HH^1(C,\cO),\bz)$ inducing  for  $1\le k\le h^0(L)$ local $K$-isomorphisms for the \'etale topology between $((W^{k-1}_d)_{red},L)$ and the closed 
  subscheme defined by the ideal generated by the minors of size $h^0(L)-k+1$ of the  $h^1(L)\times h^0(L)$ matrix of linear forms on $\HH^1(C,\cO)$ given by the Petri map $\pi_{L}$ of $L$, if $h^0(L)\le h^1(L)$.
  \end{enumerate}
\end{lemma}
\begin{proof}
By assumption there is a local embedded \'etale isomorphism $$(\pic^d(C),W^0_d, L)\simeq (N,N_1,\bz)\times \bA^{g-d-1}$$ with $(N,N_1)$ as in Lemma \ref{thmW0d}. Here $N$ is the space of $l'\times l$ Hankel matrices by  Proposition \ref{propHyl}, and $\codim (\im\pi_L)^\vee =g-d-1$.  The reduced singular locus of $N_k$ is $N_{k+1}$ by Theorem \ref{thrmHankel} {\color{hot}(ii)}, {\color{hot}(iv)}. Hence one recovers $N_{k+1}$ inductively from $N_1$ by taking successively the reduced singular locus. We have that (\ref{jiti0}) and (\ref{jiti}) are true for $k=1$. We proceed by induction on $k>1$ simultaneously for (\ref{jiti0}) and (\ref{jiti}). We assume that $Sing(W^{k-2}_d)=Sing((W^{k-2}_d)_{red})=(W^{k-1}_d)_{red}$ and that the  local \'etale isomorphism from above induces one between $((W^{k-2}_d)_{red},L)$ and $(N_{k-1},\bz)\times \bA^{g-d-1}$.  Taking  the reduced singular locus, we obtain a local \'etale isomorphism $(Sing((W^{k-2}_d)_{red}),L) = ((W^{k-1}_d)_{red},L)\simeq(N_k,\bz)\times \bA^{g-d-1}$. Taking again the reduced singular locus we obtain $(Sing((W^{k-1}_d)_{red}),L) \simeq(N_{k+1},\bz)\times \bA^{g-d-1}$. Obviously $Sing((W^{k-1}_d)_{red})$ is a reduced closed subscheme of the reduced and irreducible scheme $Sing(W^{k-1}_d)=(W^k_d)_{red}$, the latter having dimension $d-2k$ by Proposition \ref{propHell}. On the other hand, by Theorem \ref{thrmHankel} we have $\codim N_{k+1}=g-d+2k$, hence $Sing((W^{k-1}_d)_{red})$ has dimension $d-2k$ as well. Therefore $Sing((W^{k-1}_d)_{red})=Sing(W^{k-1}_d)=(W^k_d)_{red}$ as well.
\end{proof}

\medskip

 We can now finish the proof of Proposition \ref{thmHyE}. By Lemma \ref{lemRexd}, there is an embedded local \'etale isomorphism $(\pic^d,(W^{k-1}_d)_{red},L)\simeq (\HH^1(C,\cO), TC_L((W^{k-1}_d)_{red})),\bz)$ and the latter is given by the  minors of size $l-k+1$ of  $\pi_L$, viewed as a matrix of linear forms.   On the other hand, the deformation theory with cohomological constraints gives that $(\pic^d,W^{k-1}_d,L)$ is formally determined by the minors of size $l-k+1$ of a matrix $\tilde\pi_L$ of formal power series, denoted $d_{univ}$ in the proof of Theorem \ref{thm1GDik}, such that $\pi_L$ is the linear part of $\tilde\pi_L$. Hence the initial ideal of the ideal defining $W^{k-1}_d$ contains the ideal defining $(W^{k-1}_d)_{red}$. We therefore have a chain of inclusions 
  $$
 (TC_LW^{k-1}_d)_{red}\subset TC_LW^{k-1}_d \subset TC_L((W^{k-1}_d)_{red}),
 $$
 of formal neighborhoods of $\bz$, with the first inclusion by trivial reasons. The dimension at $\bz$ of the tangent cone $TC_LW^{k-1}_d$ equals that of $(TC_LW^{k-1}_d)_{red}$, and also equals
  the dimension at $L$ of $W^{k-1}_d$.  Since $(W^{k-1}_d)_{red}$ is reduced and irreducible, it follows that the above chain of inclusions is a chain of equalities. Hence   $TC_LW^{k-1}_d$ is reduced, which implies that $W^{k-1}_d$ is reduced at $L$. This proves (\ref{iti0}). By Lemma \ref{lemRexd} (\ref{jiti}), this also implies (\ref{iti}).

For the next items, we proceed using (\ref{iti}) as in the proof of Theorem \ref{thrmLCThrk} for
the corresponding statements:

\begin{itemize} 
\item (\ref{iti2}) follows from  Theorem \ref{thrmHankel} {\color{hot}(v)} since $l'-l=g-d-1$ by Riemann-Roch.
\item (\ref{iti3}) follows from  Theorem \ref{thrmHankel} {\color{hot}(vi)}.
\item (\ref{iti4}) follows from  Theorem \ref{thrmHankel} {\color{hot}(vii)}.
\item (\ref{iti5}) follows from  Theorem \ref{thrmHankel} {\color{hot}(viii)}.
\item (\ref{iti6}) follows from  Theorem \ref{thrmHankel} {\color{hot}(ix)}.$\hfill\Box$
\end{itemize}

\section{Review of $k$-generic matrices}\label{secDet}

In this section we review  some results on singularities of spaces of generic, and more generally, $k$-generic matrices. For the terminology from singularity theory we refer to Section \ref{appSI}.

\subs{\bf Generic matrices.} Fix non-zero natural numbers $a, b$. We regard the affine space $\bA^{ab}$ over the field $K$ as the space of $b\times a$ matrices over $K$.  By $\bz$ we denote the zero matrix in $\bA^{ab}$.  Without loss of generalization, we assume that $0<a\le b$.

\begin{defn}\label{defnGen}$\;$ 
\begin{enumerate}[(1)]

\item The {\it generic matrix} is the matrix $X=(x_{ij})$  of algebraically independent variables $x_{ij}$ with $1\le i\le b$, $1\le j\le a$.

\item For $k\in\bN$ let $J_k= J_k(a,b)$ be the ideal generated by the minors of size $a-k+1$ of the matrix $X=(x_{ij})$. We set $J_0=0$, and  $J_k=(1)$ if  $k\ge a+1$, and this is compatible with convention on minors from \ref{subNot}. The ideals $J_k$ are called {\it generic determinantal ideals}.
 
\item
Let
$$M_k=M_k(a,b):=\{ A\in \bA^{ab}\mid \rank(A)\le a-k\}.$$ 
The spaces $M_k$ are called {\it generic determinantal varieties}. 
\end{enumerate}

\end{defn}

It is well-known that $M_k$ is indeed an affine subvariety of $\bA^{ab}$ and that $J_k$ is the associated radical ideal \cite[II.3]{A+}. Here are some known results about the singularities of generic determinantal varieties. The terminology is recalled in Section \ref{appSI}.

\begin{theorem}\label{thrmGen} Let $1\le k\le a\le b$ be natural numbers and $M_k$ the space of $b\times a$ matrices of rank $\le a-k$. Then:

\begin{enumerate}[(i)]

\item {\rm{(}}\cite[II.2]{A+}{\rm{)}} The variety $M_k$ is isomorphic its tangent cone at $\bz$, it has dimension $(a-k)(b+k)$, and its singular locus is $M_{k+1}$.

\item {\rm{(}}\cite[II.5.2]{A+}{\rm{)}} The multiplicity of $M_k$ at $\bz$ is 
$$
\prod_{i=0}^{k-1}\frac{(b+i)!i!}{(a-k+i)!(b-a+k+i)!}.
$$

\item {\rm{(}}\cite[Prop. 2]{Ke}{\rm{)}} $M_k$ has rational singularities. 

\item {\rm{(}}\cite{Jo}, \cite{Roi}{\rm{)}} The log canonical threshold at $\bz$ of the pair $(\bA^{ab},M_k)$ equals  the global log canonical threshold and equals
$$
\min\left\{\frac{(a-i)(b-i)}{a -k+1-i}\mid i=0,\ldots ,a-k\right\}.
$$

\item {\rm{(}}\cite{Jo}{\rm{)}} The multiplier ideal of the pair $(\bA^{ab},M_k)$ with coefficient $c\in \bR_{>0}$ is the intersection of symbolic powers of generic determinantal ideals
$$
\cJ (\bA^{ab},c\cdot M_k)=\bigcap_{j=0}^{a-k} {J_{k+j}}^{(\lfloor c(j+1)\rfloor+1-(k+j)(b-a+k+j))}.
$$
Each symbolic power $J_l^{(\_)}$ has an explicit list of generators in terms of monomials in certain minors of $(X_{ij})$, see \cite{DEP}. If $k=1$, the formula simplifies to
$$
\cJ (\bA^{ab},c\cdot M_1)= {J_{1}}^{\lfloor c\rfloor +a-b}.
$$

\item  {\rm{(}}\cite{Roi}{\rm{)}} The number of irreducible components of the $n$-th jet space of $M_k$ for $n\in\bN$ is
$$
\left\{
\begin{array}{cl}
1 & \text{ if }k=1, a\; ;\\
n+2-\displaystyle{\left\lceil \frac{n+1}{a-k+1}\right\rceil} & \text{ if }1<k<a.
\end{array}\right.
$$

\item {\rm{(}}\cite{Lor}{\rm{)}} The Bernstein-Sato polynomial of the generic determinantal ideal $J_1$ is
$$
\prod_{i=b-a+1}^b (s+i),
$$
and Conjecture \ref{conjMC} holds for $(\bA^{ab},M_1)$. The same holds locally at $\bz$.

\item {\rm{(}}\cite{Roi}{\rm{)}} If $a=b$, the topological zeta function of the pair $(\bA^{ab},M_k)$ equals  the local one at the origin and is
$$
\prod_{\al\in \Omega}\frac{1}{1-\al^{-1}s}
$$
where $\Omega$ is the set of poles:
$$
\Omega=\left\{ -\frac{a^2}{a-k+1}, -\frac{(a-1)^2}{a-k}, -\frac{(a-2)^2}{a-k-1},\ldots, - k^2 \right\}.
$$


\item {\rm{(}\cite[4.3]{Jo}, \cite[\S 3]{Sta}\rm{)}} Consider $f_{a-k}:Y_{a-k}\to \bA^{ab}$ the composition of blowups of (strict transforms of) $M_a$, $M_{a-1}$, $\ldots$, $M_{k}$,  in this order.  At each stage this is the blowup of a smooth center in a smooth variety, such that $f_{a-k}$ is a log resolution $(\bA^{ab},M_{k})$. Moreover, the pullback of the ideal $I_{a-k+1}$ defining $M_{k}$ is $\cO_{Y_{a-k}}(-\sum_{i=0}^{a-k}(a-k+1-i)E_i)$, where $E_i$ is the (strict transform of the) divisor introduced by blowing up the (strict transform of) $M_{a-i}$.

\item {\rm{(}\cite{Gaf}\rm{)}} The stratification of $M_k$ given by $M_{t}\setminus M_{t+1}$ with $k\le t$ is a Whitney stratification, and the local Euler obstruction at $\bz$ of $M_k\subset \bA^{ab}$ is $\binom{a}{a-k}$.

\item\label{itPr1} {\rm{(}\cite{Per, PR}\rm{)}} If $0<k\le a$ and $j\ge 0$, the simple holonomic $\cD_M$-module composition factors of the local cohomology $\cD_M$-module $$\cH^j_{M_{k}}(\cO_M)$$ are known, each is equal to the intersection homology module $\cL(M_{k'},M)$ for some $k'$ with $k\le k'\le a$, in which case they have weight $ab+k'-k+j$ with respect to the weight filtration on $\cH^j_{M_{k}}(\cO_M)$. The Hodge filtrations on $\cH^j_{M_{k}}(\cO_M)$ and $\cL(M_k,M)$ are combinatorially determined, as well as their generating level. If $a=b$, the generating levels are $(a^2-a+k-j)/2$ and $k^2$, respectively.

\item\label{itPr2} {\rm{(}\cite{PR}\rm{)}} If $a=b$, $\gr^W_w\cO_M(*M_1) =0$ if $w<a^2$ or $w>a^2+a$, and $\gr^W_{a^2+k}\cO_M(*M_1)=\IC_{M_k}\bQ^H(-\binom{k+1}{2})$ for $0\le k\le a$. The Hodge filtration on $\cO_M(*M_1)$ is combinatorially determined and generated in level $\binom{a}{2}$. For $p\ge 0$ the $p$-Hodge ideal of $M_1$ is 
$$
I_p(M_1) = \bigcap_{k=1}^{a-1} J_{k+1}^{\left(k(p-1)-\binom{k}{2}\right)}.
$$

\item\label{itMal} {\rm{(}\cite{Mal}\rm{)}} If $a=b$ the minimal discrepancies of $M_k$ along $M_{k+1}$ and, respectively, along a point $w\in M_{k'}\setminus M_{k'+1}$ with $k\le k'\le a$ are:
$$
\mld(M_{k+1};M_k) = k+1,\quad \mld(w;M_k) = a^2-kk'.
$$

\end{enumerate}

\end{theorem}

\begin{rmk}
For (viii) only the formula for the global topological zeta function is  given in \cite{Roi}. However, his description in terms of pre-partitions of the strata of jet schemes allows the computation of the local topological, in fact even motivic, zeta function at the origin as well. We thank R. Docampo for pointing out the equality between the local and the global topological zeta functions in this case.
\end{rmk}

\subs{\bf $k$-generic matrices.}\label{subskGen}  For the rest of  this section we assume that the characteristic-zero field $K$ is  algebraically closed.  Let $V, W$ be finite dimensional vector spaces with $0<a=\dim V\le b=\dim W$.   So $\homo(V,W)\simeq\bA^{ab}$ as varieties.  The following data are equivalent:
\begin{itemize}
\item a subspace $N\subset \Hom(V,W)$,
\item a surjective pairing $\mu:V\otimes W^\vee \to N^\vee$,
\item a $b\times a$ matrix $A$ of linear forms on $N$, up to change of bases, such that the span of the entries is $N^\vee$.
\end{itemize}

\begin{defn} Let $k\ge 1$.  We define $N_k := N\cap M_k$ scheme-theoretically, where $M_k\in \Hom(V,W)$ is the subvariety of matrices of rank $\le a-k$. That is, $N_k$ is defined by the minors of size $a-k+1$ of an associated matrix $A$ of linear forms on $N$. 
 \end{defn}

\begin{propdef}\label{propdef} {\rm{(}}\cite{Eis}{\rm{)}} A subspace $N\subset \Hom(V,W)$, or the associated pairing $\mu:V\otimes W^\vee = \Hom(V,W)^\vee\to N^\vee$, or an associated matrix $A$ of linear forms on $N$, 
 is {\it $k$-generic} if any of the following equivalent conditions hold:

\begin{enumerate}

\item[(1)] The kernel of $\mu$ does not contain any sums of $k$ or fewer pure non-zero tensors $v\otimes w$.

\item[(2)] Even after arbitrary invertible row and column operations, any $k$ entries of the  matrix $A$ are linearly independent (so, non-zero if $k=1$).

\item[(3)] $(N^\perp)_{a-k}=0$, where $N^\perp=\{\psi\in \Hom(W,V)\mid \Trace(\phi\psi)=0\text{ for all }\phi\in N\}$.
\end{enumerate}
\end{propdef}

More generally, the following data are equivalent for a finite-dimensional vector space $N$:
\begin{itemize}
\item a linear map $N\to \Hom(V,W)$,
\item a  pairing $\mu:V\otimes W^\vee \to N^\vee$,
\item a $b\times a$ matrix $A$ of linear forms on $N$, up to change of bases.
\end{itemize}

Since the conditions (1) and (2) in the above definition do not depend on $N$ being a subspace of $\Hom(V,W)$, one gets:

\begin{propdef} Let $k\ge 1$. A pairing $\mu:V\otimes W^\vee \to N^\vee$, or an associated matrix $A$ of linear forms on $N$
 is {\it $k$-generic} if any of the  equivalent conditions (1) and (2) of  \ref{propdef} hold.

\end{propdef}

\begin{rmk}
It is easy to see that $(k+1)$-generic implies $k$-generic. If $k\ge a$, then $k$-generic is equivalent to generic and is further equivalent to the pairing $\mu$ being injective.
\end{rmk}

We add a reformulation that seems popular in the literature on Brill-Noether loci:

\begin{lemma}\label{lemLi} Let $1\le k\le a$. A pairing $\mu:V\otimes W^\vee\to N^\vee$ is $k$-generic if and only if  $\mu$  is injective on  $S\otimes W^\vee$ for every subspace $S\subset V$ of dimension $k$.
\end{lemma}
\begin{proof}
Using the formulation (1) from the definition of $k$-genericity, the implication $\Rightarrow$  is easy to see. We prove now the other implication. Let $\psi\in\ker \mu$ and assume it is a sum of $\le k$ pure tensors. Viewing $\psi$ as a linear map in $V\otimes W^\vee=\Hom(V,W)^\vee=\Hom(W,V)$, this is equivalent to  $\rank \psi\le k$. Let $S=\im\psi\subset V$ be the image, necessarily of dimension $\le k$. By assumption, the restriction $\mu_S$ of $\mu$ to $S\otimes W^\vee = \Hom (W,S)$ is injective. On the other hand, $\psi$ itself lies in $\Hom(W,S)\subset \Hom(W,V)$ and $\psi\in\ker\mu_S$. Thus $\psi=0$.
\end{proof}

\begin{rmk}
We do not need that the field $K$ is algebraically closed for Lemma \ref{lemLi}.
\end{rmk}

The following is essentially due to Kempf:

\begin{thrm}\label{thmkGen}
If $N\subset \Hom(V,W)$ is $k$-generic for $1\le k\le a$, then $N_k\subset N$ is a variety with at most rational singularities,  it is isomorphic with its tangent cone at $\bz$, and it has the same codimension and multiplicity at $\bz$ as the generic determinantal variety $M_k\subset \bA^{ab}$ (for these see Theorem \ref{thrmGen}).
\end{thrm}
\begin{rmk} The isomorphism to the  tangent cone is due to the associated matrix having only linear forms as entries. For $k=1$ this theorem is due to Kempf \cite{Ke}. The generalization to $k\ge 1$ essentially followed his ideas, see \cite[Lemma, p.242]{A+} which however does not mention rational singularities. The proof that $N_k$ has rational singularities is, following Kempf's,  the same as that of \cite[Teorema 3.7]{AC}: one applies \cite[Lemma 2]{Ke} to the resolution of singularities of $N_k$  found in \cite[Lemma, p.242]{A+} by taking the fiber product of $N_k$ with the canonical resolution of singularities for $M_k$.  The formulation for $k$-genericity from Lemma \ref{lemLi} is used in these references.
\end{rmk}

\begin{thrm}\label{thmResil}{\rm{(Resiliency Theorem }}\cite[Thm. 2.1]{Eis}{\rm{)}} If $N'\subset \Hom(V,W)$ is $k$-generic and $N\subset N'$ is an arbitrary subspace then:

\begin{enumerate}[(1)]

\item If $\codim_{N'}N\le a-k$, then $\codim_NN_k=\codim_{\bA^{ab}}M_k=k(k+b-a)$ and $N_k$ is Cohen-Macaulay.

\item If $\codim_{N'}N\le a-k-1$, then $N_k$ is a variety.

\item If $k>1$ (conjecturally $k$ can be 1 as well) and  $\codim_{N'}N\le a-k-2$, then $N_k$ is normal.
 
\item The singular locus of $N_k$ is contained in
 $$
 N_{k+1}\cup\bigl\{ \phi\in N_k\setminus N_{k+1}\mid \codim_N\{\psi\in N\mid \psi V\subset\phi V\}<ka  \bigr\}.
 $$
 \end{enumerate}

\end{thrm}

\begin{thrm}\label{thrmEisMart}{\rm{(}}\cite[Cor. 3.3]{Eis}{\rm{)}} If $N\subset \Hom(V,W)$ is $k$-generic and $k+h\le a$, then every component of $N_h$ has codimension $\ge k(b-a+2h-k)$ in $N$. If $N_h$ has a component of that codimension, then its singular locus is contained in $N_{h+1}$.
\end{thrm}

\begin{thrm}\label{thmEisCod}{\rm{(}}\cite[Cor. 2.2]{Eis}{\rm{)}} Let $A$ be a $k$-generic matrix of linear forms in variables $x_1,\ldots, x_m$. Let $\bar A$ be the matrix $A$ modulo the ideal generated by a fixed set of $c$ linear forms in the same variables. If $c\le a-k$ then the $(a-k+1)$-minors of $\bar A$ are linearly independent forms of degree $a-k+1$, in particular they are non-zero. If $c\le a-k-1$, then each of these minors is prime. 
\end{thrm}

We also add the following observation:

\begin{prop}\label{thrm1Gen} If $a=b$ and $N\subset\Hom(V,W)$ is $1$-generic,
  then the pair $(N,N_1)$ is log canonical. If in addition $a>1$ and $\dim N=2a-1$  then the minimal exponent of $(N,N_1)$ lies in the  interval $(1,2-1/a]$.
\end{prop}
\begin{proof} If $a=b$ then  $N_1$ is a hypersurface in $N$. Since $N_1$ is a hypersurface with rational singularities, it follows from parts (5) and (6) of Theorem \ref{thmBfct}  that the log canonical threshold of $(N,N_1)$ is 1 and the minimal exponent is $>1$.  By the statement about the multiplicity from Theorem \ref{thmkGen} and Theorem \ref{propMinE},  the minimal exponent of $(N,N_1)$ is $\le (\dim N)/a$. Supposing  that $\dim N=2a-1$, the claim follows.
 \end{proof}
 
 Note that $2a-1$ is the smallest possible dimension for an 1-generic space $N$ of square $a\times a$ matrices. This holds for example for square Hankel matrices.

\subs{\bf Hankel  matrices.}\label{subHkl} We continue with the assumption that $1\le a\le b$.

\begin{defn}
The {\it Hankel matrix} of size $b\times a$  is 
$$
H= H(a,b):=\left(
\begin{array}{ccccc}
x_1 & x_2 & x_3 & \ldots & x_a\\
x_2 & x_3 & \ldots & \ldots & x_{a+1}\\
x_3 & \ldots & \ldots & \ldots & x_{a+2}\\
\ldots & \ldots & \ldots & \ldots &\ldots \\
x_b & \ldots & \ldots & \ldots & x_{a+b-1}
\end{array}
\right)
$$
where $x_i$ are independent variables.
\end{defn}

 Hankel matrices are particular cases of catalecticant matrices. 
  Hankel matrices are  1-generic but not generic, cf. \cite[Prop. 4.2]{Eis}. 
  

\begin{theorem}\label{thrmHankel} Let $1\le k \le a\le b$.
Let $N\subset \bA^{ab}$  be the 1-generic linear subspace of matrices corresponding to the Hankel matrix $H$, so that $N= \bA^{a+b-1}$.

\begin{enumerate}[(i)]

\item {\rm{(}\cite[Lemme 2.3]{GP}, \cite{Co2}\rm{)}} If $k>1$ consider the re-embedding of $N$ as the linear subspace $N'$ of  $\bA^{(a-k+1)(b+k-1)}$ corresponding to the  Hankel matrix  $H'=H(a-k+1, b+k-1)$. The isomorphism $N\simeq N'$ restricts to an isomorphism $N_k\simeq N'_1$, with $N'_1$   defined by the maximal minors of $H'$.

\item {\rm{(}\cite[Prop. 4.3]{Eis}, \cite{Con}\rm{)}} $N_k$ is a variety with at most rational singularities and has codimension $b-a-1+2k$ in $N$.

\item {\rm{(}\cite[Prop. 4.3]{Eis}, attributed to Gundelfinger by \cite{IK}\rm{)}} If $k<a$ then the projectivization  ${\bf N}_k\subset\bP(N)=\bP^{a+b-2}$ of $N_k$ is the $(a-k)$-secant variety of ${\bf N}_{a-1}$, and ${\bf N}_{a-1}$ is the rational normal curve of degree $a+b-2$. 

\item {\rm{(}\cite[p.440]{Ber}\rm{)}} The singular locus of $N_k$ is $N_{k+1}$.

\item {\rm{(}\cite[Prop. 5.11]{EN}\rm{)}} If  $k\le m\le a$ then the multiplicity of $N_k$ at any point in $N_m\setminus N_{m+1}$ is $$\binom{b-a-1+m+k}{m-k}.$$
Hence if $a=b$ then $N_k\setminus N_{k+1}$ is the locus of points of $N_1$ with multiplicity exactly $k$.


\item {\rm{(}\cite[Cor. 2.4]{Ber}\rm{)}} Consider $f_{a-k}:Y_{a-k}\to N$ the composition of blowups of (strict transforms of) $N_a$, $N_{a-1}$, $\ldots$, $N_k$, in this order. At each stage this is the blowup of a smooth center, such that $f_{a-k}$ is a log resolution of $(N,N_k)$.

\item If $a=b$ then $$f^*_{a-1}(N_1)= \sum_{i=0}^{a-1}(a-i)E_i$$
 where $E_i$ is the (strict transform of the) divisor introduced by blowing up the (strict transform of) $N_{a-i}$. 
 
\item {\rm{(}\cite[4.6]{Con} for $k=1$\rm{)}} $$
\lct (N, N_k)=
\left\{
\begin{array}{ll}
1 & \text{ if }a=b\text{ and }k=1,\\
1+\displaystyle{\frac{b+k-2}{a-k+1}} & \text{ if }a<b. 
\end{array}
\right.
$$

\item  If $a=b>1$, the minimal exponent of $N_1$ is 3/2.

\end{enumerate}

\end{theorem}

\begin{proof} Some of the assertions above are not literally covered by references. Before we tie the loose ends, we give a  shorter proof of $N_k$ having rational singularities than the one in \cite{Con}. By part  {\color{hot}(i)} the ideal defining $N_k$ is given by the maximal minors of  another Hankel matrix. Since every Hankel matrix is 1-generic,    {\color{hot}(ii)} follows from Theorem \ref{thmkGen}.

Parts {\color{hot}(iv)}, {\color{hot}(v)}, {\color{hot}(vi)} are phrased in \cite{Ber,EN} in terms of degrees and multiplicities of the secant varieties from {\color{hot}(iii)}. The translation to the current form is immediate.

Part {\color{hot}(vii)} follows  from parts {\color{hot}(v)}, {\color{hot}(vi)}. 

 Part {\color{hot}(viii)} is stated in \cite{Con} only for $k=1$. This case also follows immediately from {\color{hot}(ii)} and {\color{hot}(vii)}. For $k>1$, we have $(N,N_k)\simeq (N',N'_1)$ by part {\color{hot}(i)}. Hence $\lct(N,N_k)=\lct(N',N'_1)$, which reduces the computation to $k=1$.
 
 For part {\color{hot}(ix)} we follow an argument communicated to us by M. Musta\c{t}\u{a} in the context of \cite{SY}. The codimension of $N_2$ in $N$ is 3. Hence cutting down $N$ by generic hypersurfaces we obtain a smooth 3-dimensional variety $N'$ with a surface $N'_1=N'\cap N_1$ such that the singular locus of $N'_1$ is a point $N'_2=N'\cap N_2$ of multiplicity 2 in $N'_1$. By Theorem \ref{propMinE} {\color{hot}(iii)}, there is an inequality of minimal exponents $\al(N,N_1)\le \al(N',N'_1)$. By  Theorem \ref{propMinE} {\color{hot}(i)}, there is also an inequality $\al(N',N'_1)\le 3/2$. Hence $\al(N,N_1)\le 3/2$. On the other hand, by  Theorem \ref{propMinE} {\color{hot}(ii)} we have $$\al(N,N_1)\ge\min_{0\le i\le a-2}\left\{\frac{2(a-i)-1}{a-i}\right\}=3/2$$ where the numerical data of the log resolution from {\color{hot}(vi)} required by Theorem \ref{propMinE} {\color{hot}(ii)} is obtained from parts {\color{hot}(ii)} and {\color{hot}(vii)} above.
 \end{proof}

\begin{rmk}\label{rmkHa}  
 Any 1-generic linear subspace of matrices $N\subset \bA^{ab}$ has dimension $\ge a+b-1$, with equality for Hankel matrices for example. 
\end{rmk}

\section{Review of Brill-Noether loci}\label{apxBN}

This section reviews some known results on the singularities of Brill-Noether loci of stable vector bundles on curves, without attempting to be exhaustive. We also address Theorem \ref{thmGG2}, a fact which seems missing in this generality from the literature.

Fix $C$  a smooth projective curve of genus $g$ over an algebraically closed field $K$ of characteristic zero. Let $\omega_C$ be the canonical bundle of $C$. Let $n\ge 1, d\ge 0, k\ge 1$ in $\bZ$. Fix a vector bundle $F$ on $C$. We will use the following notation.

\begin{defn}\label{defBNP} $\;$
\begin{enumerate}[(1)]

\item We let $\cM_{n,d}$ be  the moduli space of stable vector bundles on $C$ of rank $n$ and degree $d$. The isomorphism class in $\cM_{n,d}$ of a stable vector bundle $E$ will be denoted  $E$ also.

\item We let
$$
\cV_{n,d,k}(F):=\{E\in \cM_{n,d}\mid h^{0}(C,E\otimes F)\ge k\}
$$
 endowed with the natural structure of closed subscheme of $\cM_{n,d}$, see \cite{A+, Li,CT}. We call these schemes {\it  Brill-Noether loci}.
We set $\cV_{n,d,k}=\cV_{n,d,k}(\cO_C)$. When $\cM_{n,d}$ is fixed from the context, we set $\cV_k(F)=\cV_{n,d,k}(F)$ and $\cV_k=\cV_k(\cO_C)$. 

\item Let $E\in\cM_{n,d}$. The {\it Petri map} 
$$
\pi_{E,F}: \HH^0(C,E\otimes F)\otimes \HH^0(C,E^\vee\otimes F^\vee\otimes \omega_C) \to \HH^0(C,E\otimes E^\vee\otimes\omega_C)
$$
is defined to be the composition of the multiplication map 
$$
\HH^0(C,E\otimes F)\otimes \HH^0(C,E^\vee\otimes F^\vee\otimes \omega_C) \to \HH^0(C, E\otimes E^\vee\otimes F\otimes F^\vee\otimes\omega_C)
$$
and the trace map
$$
\HH^0(C, E\otimes E^\vee\otimes \enmo (F) \otimes\omega_C)\to \HH^0(C,E\otimes E^\vee\otimes\omega_C)
$$
via  $F\otimes F^\vee=\enmo (F)$, the vector bundle of endomorphisms.
We set $\pi_{E}=\pi_{E,\cO}$. Set 
$l=h^0(C,E\otimes F)$, $l'=h^1(C,E\otimes F)$.   If $E$ is a line bundle, we denote it by $L$ to stress this fact.

\item
Let $
\rho_{n,d,k}(F) := n^2(g-1)+1-k\bigl( k - n\deg(F)+\rank (F)(n(g-1)-d) \bigr).
$
More conceptually, this  equals
$\dim \cM_{n,d} - h^0(E\otimes F)\cdot h^1(E\otimes F)$ for $E\in \cV_{n,d,k}(F)\setminus \cV_{n,d,k+1}(F)$. It also equals
$ \dim \cM_{n,d} -k(k-\chi(E\otimes F))$ for all $E\in\cV_{n,d,k}(F).$ We set $\rho_{n,d,k}=\rho_{n,d,k}(\cO_C)$.
\end{enumerate}

\end{defn}

\subs\label{subAss}
{\bf Assumption.} In this section we will  assume (\ref{eqAssu}), namely, that $l\le l'.$ 
Equivalently, $\chi(E\otimes F)\le 0$. Since 
\be\label{eqChi}
l-l'=\chi(E\otimes F) =n\deg(F)-\rank (F)(n(g-1)-d)
\ee the assumption does not depend on $E$. 

If (\ref{eqAssu}) is not satisfied, one can always reduce to this case by replacing the tuple $(E,F, k, d)$ with $(E^\vee\otimes\omega_C,F^\vee, k-n\deg(F)+\rank(F)(n(g-1)-d), 2(g-1)n-d)$ using Serre duality. By this we mean that the isomorphism $\cM_{n,d}\xa{\sim}\cM_{n,2(g-1)n-d}$ given by $E\mapsto E^\vee\otimes\omega_C$ induces an isomorphism $\cV_{k}(F)\xa{\sim} \cV_{k-n\deg(F)+\rank(F)(n(g-1)-d)}(F^\vee)$.

\begin{rmk} For convenience we give a dictionary between classical notation and ours.

\begin{itemize}
\item The Picard variety is $\pic^d(C)=\cM_{1,d}$, the Jacobian variety is $\pic^0(C)=\cM_{1,0}$.
\item The classical Brill-Noether variety is $W^r_d=\cV_{1,d,r+1}$.
\item Any classical theta divisor  $\Theta\subset \pic^0(C)$ is the image of $W_{g-1}^0=\cV_{1,g-1,1}$ under an isomorphism $\pic^{g-1}(C)\simeq \pic^0(C)$  by translation with a fixed degree $g-1$ divisor class.
\end{itemize}

\end{rmk}

\begin{theorem}\label{thmLePot}$\rm{(}$\cite[\S 8]{LeP}$\rm{)}$
If the  space $\cM_{n,d}$ is not empty, it is a smooth variety of dimension $n^2(g-1)+1=h^1(C,E\otimes E^\vee)$ for any $E\in\cM_{n,d}$. If it is empty then $g=0$ or 1.
\end{theorem}

We have seen in Theorem \ref{thrmInj} that genericity of the Petri maps is attained in some cases.
The other extreme of somewhat good behaviour of the Petri maps, 1-genericity, is always attained when $F=\cO_C$:

\begin{lemma}\label{lem1genP}
For any curve $C$, the Petri map $\pi_E$ is 1-generic  for every $E\in\cM_{n,d}$.
\end{lemma}
\begin{proof}
More generally, for two vector bundles $E_1$, $E_2$ on a smooth variety $X$, the multiplication map on global sections
$$
\HH^0(X,E_1)\otimes \HH^0(X,E_2)\to \HH^0(E_1\otimes E_2)
$$
is 1-generic. Indeed, if $s_i\in \HH^0(X,E_i)$ is a non-zero section, then its zero locus $Z(s_i)\subset X$ is closed and properly contained in $X$. Since $Z(s_1s_2)=Z(s_1)\cup Z(s_2)$ set-theoretically, it follows that $Z(s_1s_2)\subsetneq X$. Thus $s_1s_2\neq 0$.
\end{proof}

\begin{thrm}\label{thmKVF} Assume that $0\le \rho_{n,d,k}(F)<\dim \cM_{n,d}$,   $E\in\cV_k(F)\subset\cM_{n,d}$, $F$ and $E\otimes F$ are stable, and the Petri map $\pi_{E,F}$ is $k$-generic. Then the tangent cone at $E$ of $\cV_k(F)$ is isomorphic to the subscheme of the affine space $\HH^1(C,E\otimes E^\vee)$ given by the ideal generated by the minors of size $l-k+1$ of the $l'\times l$ matrix of linear forms on $\HH^1(C,E\otimes E^\vee)$ associated to $\pi_{E,F}$.
\end{thrm}

\begin{rmk}\begin{enumerate}[(1)]

\item
 Since $\pi_L$ is always 1-generic for $E=L\in\pic^d(C)$, the theorem describes in particular the tangent cones to all $\cV_1=W^0_d$ in $\pic^d(C)$. This is the main  result of Kempf \cite{Ke}. The case of $W^r_d$  is treated similarly, see  \cite{Ke-abIn}, \cite{AC}, \cite[VI 2.1]{A+}.  
 
\item For $F=\cO_C$ the statement appeared in  \cite[Thm. 1, Prop. 5 (d)]{Li} under the extra assumption that $\cV_k\neq \cV_{k+1}$. 
The statement as in the theorem is \cite[Thm. 3.4]{CT}, and for $F\neq \cO_C$  see \cite[Rmk. 3.8]{CT}. The proof is essentially the same as in \cite{Ke}. In all these references,  $k$-genericity is phrased in the equivalent form from Lemma \ref{lemLi}

\item Since our standing assumption is that $l\le l'$, instead of assuming that $0\le \rho_{n,d,k}(F)<\dim \cM_{n,d}$, it is enough to ask that $1\le k\le l$. Then $k$-genericity implies the condition on $\rho$ by Theorem \ref{thmkGen}.
\end{enumerate}
\end{rmk}

\begin{thrm}\label{thrmLCTa} 
Let $1\le k\le l$ and  $L\in \cV_k(F)\subset\pic^d(C)$.
Assume one of the two holds:

\begin{enumerate}
\item[(a)] $C$ is generic among curves with genus $g$ and either $F=\cO_C$ or $F$ is generic among vector bundles on $C$ of same rank and degree; or more generally,
\item[(b)] $\pi_{L,F}$ is injective.
\end{enumerate}
Then the following hold (in a Zariski open neighborhood of $L$ in $\pic^d(C)$ in case (a)):

\begin{enumerate}[(i)]

\item {\rm{(}}\cite[IV.4]{A+}, \cite[Thm. 1.1]{TiBtw}{\rm{)}}  In a Zariski open neighborhood of $L$: $\cV_k(F)$  is reduced, it has dimension $\rho_{1,d,k}(F)$, and  the singular locus of $\cV_k(F)$ is $\cV_{k+1}(F)$.

\item {\rm{(}}\cite[VI.2]{A+}, \cite[Thm. 3.4, Rmk. 3.8]{CT}{\rm{)}} The multiplicity of $\cV_k(F)$ at $L$ is 
$$
\prod_{i=0}^{k-1}\frac{(l'+i)!i!}{(l-k+i)!(l-l'+k+i)!}.
$$

\item {\rm{(}}\cite{Ke} for $k=1$, \cite[3.7]{AC} for $k\ge 1${\rm{)}} If $F=\cO_C$, $\cV_k$ has rational singularities. 

\item {\rm{(}}\cite[Thm. B]{Zhu}{\rm{)}} If $F=\cO_C$, the log canonical threshold  of the pair $(\pic^d(C),\cV_k)$ at $L$ is 
$$
\min\left\{\frac{(l-i)(l'-i)}{l -k+1-i}\mid i=0,\ldots ,l-k\right\}.
$$
\end{enumerate}
\end{thrm}

\begin{rmk}\label{rmkProofs}  (1) The proofs of (i)-(iii) work as follows.  The pair $(T_L\pic^d(C),TC_L\cV_k(F))$ consisting of a tangent space and tangent cone at $L$, is by Theorem \ref{thmKVF}   isomorphic up to a smooth factor to the pair $(\bA^{ll'},M_k)$ where $M_k$
is the generic determinantal variety  of $l'\times l$ matrices of rank $\le l-k$. Then Theorem \ref{thrmGen} (i)-(iii) applies and these properties of tangent cones pass to the local properties of the original scheme. 

(2)
This proof does not work on the nose for (iv), since in general  the log canonical threshold cannot  be read from the tangent cone. It was remarked a posteriori as a curiosity in \cite{Zhu}  that the log canonical threshold at $L$ of $(\pic^d(C),\cV_k)$  equals that of $(\bA^{ll'},M_k)$. The proof from \cite{Zhu} used  the description of $\cV_k\subset\pic^d(C)$ as a degeneracy locus and jets. Below we give a shorter proof. A different proof is given by Theorem \ref{thrmLCT}  which also explains the curiosity.

\end{rmk}

\subs\label{subZhu}{\bf Proof of Theorem \ref{thrmLCTa} {\color{hot} (iv)}.}
Log canonical thresholds cannot increase under specialization \cite[9.5.41]{Laz}. Hence
$$
\lct_L(\pic^d(C),\cV_k)\ge \lct_L(T_L\pic^d(C),TC_L\cV_k)= \lct_\bz(\bA^{ll'},M_k)=\lct(\bA^{ll'},M_k),
$$
by using the specialization to the tangent cone, cf. the comment after Theorem \ref{thmElik}.
To show the reverse inequality, one uses the same argument as in \cite[p.3156]{Zhu}. Namely, there is a  description of $\cV_k$ as a degeneracy locus. It  has the property that locally at $L$ on $\pic^d(C)$ there is a map to $\bA^{ll'}$ such that the pullback of $M_k$ is $\cV_k$. Log canonical thresholds cannot increase under pullbacks via morphisms between smooth varieties \cite[9.5.8]{Laz}. Hence
$$
\lct_L(\pic^d(C),\cV_k)\le \lct(\bA^{ll'},M_k).
$$
Thus $\lct_L(\pic^d(C),\cV_k)= \lct(\bA^{ll'},M_k)$ and the claim follows from Theorem \ref{thrmGen} (iv).
$\hfill\Box$

\begin{thrm}\label{thrmLCTb}
For any curve $C$ :
\begin{enumerate}[(i)]

\item {\rm{(}}\cite[Cor. 3.6]{CT}{\rm{)}}  If $\rho_{n,d,1}\ge 0$ then every non-empty irreducible component  of $\cV_{1}\subset \cM_{n,d}$ is reduced and has dimension $\rho_{n,d,1}$.


Let $E\in \cV_{1}\subset \cM_{n,d}$  (as in \ref{subAss}, as always in this section). Then:

\item {\rm{(}}Riemann's Singularity Theorem  for $n=1=g-d$;  \cite{Ke} for $n=1$; \cite{La} for $d=n(g-1)$; \cite[Cor. 2]{Li} and \cite[Cor. 3.6]{CT} for $n> 1${\rm{)}} The multiplicity of $\cV_{1}$ at $E$ is 
$
\binom{l'}{l-1}.
$

\item {\rm{(}}\cite{Ke}{\rm{)}} If $n=1$, $\cV_{1}\in\pic^d(C)$ has rational singularities at $L=E$.
\end{enumerate}

 \end{thrm}

 \begin{rmk}\label{rmkOle} The proofs work as follows. By Theorem \ref{thmKVF} the pair $(T_E\cM_{n,d},TC_E\cV_1)$ consisting of a tangent space and tangent cone at $E$, is isomorphic up to a smooth factor to the pair $(N,N_1)$ where $N\subset \bA^{ll'}$ is the 1-generic subspace of $l'\times l$ matrices associated to the Petri map $\pi_E$, and 
  $N_1=N\cap M_1$ is the subvariety of matrices of rank $\le l-1$. Then Theorem \ref{thmkGen} applies to prove (i), (ii), (iii) for the tangent cone, and these properties of the tangent cone pass to local properties of the original scheme. 
   \end{rmk}

\subs{\bf Proof of Theorem \ref{thmGG2}.} The same proof as in Remark \ref{rmkOle} works, since all we needed was the 1-genericity of the Petri map. In this more general case,  1-genericity holds by Lemma \ref{lem1genP}.
$\hfill\Box$

\medskip

\begin{thrm}\label{thmDimSing} $\;$

(i) {\rm{(}}\cite{A+, GB}, cf. Proposition \ref{propOurLePot} {\rm{)}} If $1\le k$, the dimension of $\cV_k(F)$ at every point  is at least $\rho_{n,d,k}(F)$. If $\cV_{k}(F)\neq \cM_{n,d}$, then $\cV_{k+1}(F)_{red}\subset Sing(\cV_k(F))$, where $Sing$ denotes the reduced singular locus.

(ii) {\rm{(}}\cite[IV, Cor. 4.5]{A+}{\rm{)}} If $n=1$ and $F=\cO_C$, $\cV_1=W^0_d$ is irreducible and $Sing(\cV_1)=(\cV_2)_{red}$.

(iii) {\rm{(}}Martens \cite[IV, Thm. 5.1]{A+}{\rm{)}} If $g\ge 3$, $2\le d<g$, and $0<2r\le d$, then:

\begin{itemize}
\item if $C$ is not hyperelliptic,  $\dim W^r_d< d-2r$;
\item if $C$ is hyperelliptic, $\dim W^r_d=d-2r$.
\end{itemize}

\end{thrm}

When $n=1$, positivity of the vector bundles defining the degeneracy locus structure of Brill-Noether loci led to the following due to Kempf, Kleiman-Laskov, Griffiths-Harris, Fulton-Lazarsfeld, see \cite{A+}:

\begin{thrm} Let $d,k \ge 1$. Consider $\cV_k\subset \pic^d(C)$.

(i) If $\rho_{1,d,k}\ge 0$ then $\cV_k\neq\emptyset$ is non-empty. If $\rho_{1,d,k}>0$ then $\cV_k$ is connected.

(ii) For a general curve $C$: If $\rho_{1,d,k}<0$ then $\cV_k$ is empty. If $\rho_{1,d,k}>0$, then $\cV_k$ is irreducible.

\end{thrm}

A similar result is available for $F\neq\cO_C$, see \cite[Thm. 2.1]{Hitch}. For $n>1$, see \cite{GB, Hitch}.

\subs{\bf Hyperelliptic curves.} For this subsection we do not make the assumption \ref{subAss}. 
Recall the definition of Hankel matrices from \ref{subHkl}.

\begin{prop}\label{propHyl}
If $C$ is hyperelliptic, $d<2g$, $L\in \pic^d(C)$ with $0\neq h^0(L)h^1(L)$, then the Petri map $\pi_L$ is represented by a Hankel matrix in a suitable base.
\end{prop}
\begin{proof}
Let $f:C\to\bP^1$ be the hyperelliptic pencil and let $A=f^*(\cO(1))$. Then $A$ has degree 2 and $V=\HH^0(A)$ is 2-dimensional base-point free. If $z_0,z_1$ denote two generators of the homogeneous coordinate ring of $\bP^1$, then $s_0, s_1$ with $s_i=f^*z_i$ generate $V$. Moreover  $\HH^0(kA)=S^kV$, and the $k$-symmetric power and the multiplication map $\HH^0(k_1A)\otimes \HH^0(k_2A)\to \HH^0((k_1+k_2)A)$ is identified with the usual multiplication on symmetric powers of $V$, for $k,k_1,k_2\in\bN$ with $k_1+k_2\le g$ and $k\le g$. Since $V$ is 2-dimensional, the matrix of linear forms representing any of these multiplication maps is Hankel, up to change of bases.

If $L\in W^r_d\setminus W^{r+1}_d$, it is known that 
\be\label{eqLAP}
L\simeq rA+P_1+\ldots +P_{d-2r}\ee and $\omega_C\otimes L^{-1}\simeq (g-1-d+r)A+Q_1+\ldots +Q_{d-2r}$ for some points $P_i, Q_i$ such that each $P_i+Q_i$ is a fiber of $f$, and no two $P_i$ lie in the same fiber, cf. \cite[D9 on p.41]{A+}. We have $2r\le d$ by Clifford Theorem. Let $l_i(s_0,s_1)$ denote the linear form defining $f(P_i)=f(Q_i)$ in $\bP^1$. 

Let $l_i^P$ denote be a generator of the 1-dimensional space of global sections of $\cO_C(P_1)$. Using  $P_1\simeq A-Q_1$, the image of $l_i^P$ under $\HH^0(A-Q_1)\to \HH^0(A)$ is $l_i$, up to a non-zero constant which we can take to be 1. We define $l_i^Q$ similarly by replacing $P_i$ with $Q_i$, so that its image in $\HH^0(A)$ is also $l_i$. Then the product $l_i^Pl_i^Q$ is non-zero and must map to $l_i$ under the multiplication map $\HH^0(\cO_C(P_1))\otimes \HH^0(\cO_C(Q_1))\to \HH^0(\cO_C(P_1+Q_1))\simeq \HH^0(A)$.

Consider now the tensor product of 1-dimensional vector spaces $\HH^0(\cO_C(P_1))\otimes\ldots\otimes \HH^0(\cO_C(P_k))$, where $k\le d-2r$. Under the multiplication map to the 1-dimensional vector space $\HH^0(\cO_C(P_1+\ldots+P_k))$, the image $l_1^P\ldots l_k^P$ is a generator. Under the inclusion $$\HH^0(\cO_C(P_1+\ldots+P_k))\simeq \HH^0(kA-Q_1-\ldots-Q_k)\to \HH^0(kA)$$ the image of $l_1^P\ldots l_k^P$ is $l_1\ldots l_k$. Similarly, $\HH^0(\cO_C(Q_1+\ldots+Q_k))$ is generated by $l_1^Q\ldots l_k^Q$, whose image in $\HH^0(kA)$ is $l_1\ldots l_k$.

Consider the multiplication map
$$
\HH^0(rA)\otimes \HH^0(\cO_C(P_1+\ldots + P_{k})) \to \HH^0(rA+P_1+\ldots +P_{k}).
$$
By dimensional reasons, it must be an isomorphism. Hence
$$
\HH^0(rA+P_1+\ldots +P_{k}) = l_1^P\ldots l_k^P\cdot S^rV
$$
as a subspace of 
$$
\HH^0((r+k)A) = S^{r+k}V
$$
by mapping $l_i^P$ to $l_i$. Similarly, 
$$
\HH^0((g-1-d+r)A+Q_1+\ldots +Q_{k}) = l_1^Q\ldots l_k^Q\cdot S^{g-1-d+r}V
$$
as a subspace of $$\HH^0((g-1-r+k)A)=S^{g-1-r+k}V$$
by mapping $l_i^Q$ to $l_i$. It follows that the Petri map $\pi_L$ is the map
$$
l_1^P\ldots l_k^P\cdot S^rV \otimes l_1^Q\ldots l_k^Q\cdot S^{g-1-d+r}V \to  l_1\ldots l_k\cdot S^{g-1-d+2r}V
$$
defined by the usual multiplication on symmetric powers together with $l_i^Pl_i^Q=l_i$. That is, as vector spaces, this is the same as the usual multiplication
$
S^rV\otimes S^{g-1-d+r}V\to S^{g-1-d+2r}V.
$ Hence the  matrix of linear forms representing it is Hankel, up to a change of coordinates.
\end{proof}

\begin{prop}\label{propHell}{\rm{}(}\cite[Prop. 2.4]{BK}{\rm{)}} Let $C$ be a hyperelliptic curve, $g\ge 2$, $d< g$, $r\ge 0$. Then
 $W^r_d$ is an irreducible scheme of dimension $d-2r$,  $Sing(W^r_d)= (W^{r+1}_d)_{red}$, and $(W^{r}_d)_{red}\simeq W^0_{d-2r}$. Same is true for $d=g$ and $r>0$, in which case $W^0_g=\pic^g(C)$.
  \end{prop}
\begin{proof}
The Abel-Jacobi proper morphism $C^{(d-2r)}\to \pic^d(C)$ given by $D\to rA+D$ surjects set-theoretically onto $W^r_d$ and is one-to-one generically by (\ref{eqLAP}). Hence $(W^{r}_d)_{red}\simeq W^0_{d-2r}$ since the latter is reduced and is the scheme theoretic image of the Abel-Jacobi morphism. Thus $W^r_d$ 
 is irreducible, and $\dim W^r_d=d-2r$.  Proposition \ref{propOurLePot} implies the general fact  that  $Sing(W^r_d)\subset (W^{r+1}_d)_{red}$ as reduced algebraic sets, and that for $L'\in W^r_d\setminus W^{r+1}_d$, $\dim T_{L'}W^r_d = g- \dim \im \pi_{L'}$. By  Proposition \ref{propHyl}, $\pi_{L'}$ is a Hankel matrix and hence $\dim \im \pi_{L'}=h^0(L')+h^1(L')-1 = 2r-d+g$. It follows that $\dim T_{L'}W^r_d =d-2r$. Since $\dim W^r_d=d-2r$, we must have $Sing(W^r_d)= (W^{r+1}_d)_{red}$. 
\end{proof}

\section{Review of singularity invariants}\label{appSI}

In this section we recall some definitions and facts from singularity theory used in the article. There are no new results here.
We take $K=\bC$ for simplicity. Let $X$ be $K$-scheme of finite type and $x\in X(K)$ a point.

\begin{defn}
 The {\it tangent cone} $TC_xX$ of $X$ at $x$ is the spectrum of associated graded ring $\gr_m\cO_{X,x}=\oplus_{i\ge 0}\; m^i/m^{i+1}$ of the local ring $(\cO_{X,x},m)$ with respect to the $m$-adic filtration.
\end{defn}

In particular, if $X\subset \bA^n$ is a closed subscheme given by an ideal $I\subset K[x_1,\ldots,x_n]$, then the tangent cone at the origin $TC_0X$ is given by the {\it initial ideal} $in(I)$, the homogeneous ideal generated by the smallest-degree homogeneous components of the polynomials $f\in I$. In this setup one can always define a flat specialization of $X$ to $TC_0X$, cf. \cite[Thm. 15.17]{E}.

\begin{defn}
The {\it multiplicity}  of $X$ at $x$ is the unique integer $e(X,x)$ such that 
$$
\length_{\cO_{X,x}} m^n/m^{n+1} = e(X,x)\frac{n^{d-1}}{(d-1)!} + \text{ lower order terms in }n
$$
for $n$ big enough, where $d$ is the Krull dimension of $\cO_{X,x}$.
\end{defn}

The multiplicity and the dimension of $X$ at $x$ equal those of the tangent cone $TC_xX$ at the vertex \cite[12.1]{E}. In fact, they are equal to the degree and the dimension plus one, respectively, of the projectivized tangent cone $\bP(TC_xX)$ inside the projectivization of the Zariski tangent space $\bP(T_xC)$.

\begin{defn}
We say $X$ has {\it rational singularities}  if it is normal and there exists a proper birational morphism $f:Y\to X$ such that $Y$ is a regular $K$-scheme (that is, $f$ is a {\it resolution} of $X$) and $R^if_*\cO_Y=0$ for $i>0$. We say $X$ has rational singularities at $x$ if a Zariski open neighborhood of $x$ has rational singularities.
\end{defn}

The following is due to Elkik \cite{El}:

\begin{theorem}\label{thmElik}  1) A deformation of a rational singularity is a rational singularity.

2) If $f:X\to S$ is a flat morphism and $x\in X$ is such that $s=f(x)$ is a rational singularity in $S$
 and $x$ is a rational singularity of the fibre $f^{-1}(s)$, then $x$
 is a rational singularity in $X$.
\end{theorem}

In particular, if the tangent cone  $TC_x$ has rational singularities then $X$ has rational singularities at $x$, since there exists a flat specialization of a Zariski open affine neighborhood of $x$ to the tangent cone of $X$ in $x$.
The converse is not true:  $X=(x^2+y^3+z^4=0)$ in $\bA^3$ has rational singularities, but $TC_0X=(x^2=0)$ is not even reduced.

\begin{defn}
The scheme $X$ is {\it Cohen-Macaulay} if every local ring $\cO_{X,x}$ admits a regular sequence of elements in the maximal ideal of length equal to $\dim \cO_{X,x}$.
\end{defn}

One has the following, see \cite[Cor. 18.14]{E} and \cite{Ke}, respectively.

\begin{theorem}\begin{enumerate}[(1)]

\item If $X$ is Cohen-Macaulay then $X$ has no embedded component and its irreducible components have the same codimension.

\item If $X$ has rational singularities then $X$ is Cohen-Macaulay.
\end{enumerate}

\end{theorem}

\begin{defn} If $Z$ is a reduced closed subscheme of a smooth variety $X$, the {\it symbolic powers} of its (radical) ideal $I\subset\cO_X$ are
$$
I^{(n)}:=\{f\in\cO_X\mid \ord_x(f)\ge m\text{ for general }x\in Z\}.
$$ 
\end{defn}

\begin{defn}\label{defnLR} Let $Z$ be a closed subscheme of a smooth variety $X$, defined by a sheaf of ideals $I\subset \cO_X$. 
\begin{enumerate}[(1)]

\item A {\it log resolution} of $(X,Z)$ is a resolution $f:Y\to X$ such that $f^{-1}(Z)$ and its union with the support of $K_{Y/X}$ is a divisor with simple normal crossings. Write $I\cdot\cO_Y=\cO_Y(-\sum_{j\in J}N_jE_j)$ and $K_{Y/X}=\sum_{j\in J}(\nu_j-1)E_j$ where $E_j$ are prime divisors on $Y$, and $N_j, \nu_j\in\bN$. 

\item The {\it log canonical threshold} of $(X,Z)$ at $x\in Z$ is
$$
\lct_x(X,Z) := \min\{{\nu_j}/{N_j}\mid f(E_j)\cap U\neq \emptyset \}
$$
where $U$ is a small Zariski open neighborhood of $x$ in $X$. 

\item The {\it multiplier ideal} with coefficient $c\in\bR_{>0}$ is
$$
\cJ_x(X,cZ):=f_*\cO_Y(K_{Y/X}-\lfloor c\sum_{j}N_jE_j  \rfloor) \subset\cO_X
$$
where the sum is over $j\in J$ with $f(E_j)\cap U$. 

\item Taking $U=X$ one obtains by definition the global versions $\lct(X,Z)$ and $\cJ(X,cZ)$ of the log canonical threshold and multiplier ideals, respectively. 

\item The pair $(X,Z)$ is {\it log canonical} if $\lct(X,Z)=1$.

\item The {\it topological zeta function} of $(X,Z)$ at $x\in Z$ is the rational function
$$
Z^{top}_{X,Z,x}(s):=\sum_{\emptyset\neq I\subset J}\chi(E_I^\circ\cap f^{-1}(x))\prod_{j\in I}\frac{1}{N_js+\nu_j},
$$
where $E_I^\circ=\cap_{i\in I}E_i\setminus \cup_{j\in J\setminus I}E_j$, and $\chi$ is the topological Euler characteristic. Taking $E_I^\circ$ instead of $E_I^\circ\cap f^{-1}(x)$ one obtains the global version, which we denote $Z^{top}_{X,Z}(s)$.

\item 

\end{enumerate}

\end{defn}

The following can be found in \cite[Part III]{Laz}:

\begin{theorem} \begin{enumerate}[(1)]
\item There is an equality $$\lct_x(X,Z)=\min\{c\in\bR_{>0}\mid \cJ_x(X,cZ)\subsetneq\cO_X\}.$$ 

\item The multiplier ideals $\cJ_x(X,cZ)$ are independent of the choice of log resolution.

\item Let $f_i\in\cO_X$ be a finite set of local generators for the ideal $I$ of $Z$ around $x$. Then, in terms of analytic functions,
$$
\cJ_x(X,cZ)^{an}=\{g\in\cO_{X}^{an}\mid \frac{|g|^2}{(\sum_i |f_i|^2)^c}\text{ is integrable locally around }x\}.$$

\item There is a finite stratification into locally closed subsets of $Z$ such that $\lct_x(X,Z)$ and $\cJ_x(X,Z)$ are piecewise constant as functions of $x\in Z$.

\item One has $\lct(X,Z)=\min_{x\in Z}\lct_x(X,Z)$ and $\cJ(X,cZ)=\cap_{x\in Z}\cJ_x(X,Z)$.
\end{enumerate}

\end{theorem}

The following is due to Denef-Loeser \cite{DL}:

\begin{theorem}
The topological zeta function of $(X,Z)$ (resp. at $x\in Z$) is independent of the choice of log resolution.
\end{theorem}

\begin{defn}\label{defnbfct} If $I\subset\cO_X$ is an ideal of regular functions on a smooth affine variety $X$, the {\it Bernstein-Sato polynomial} $b_I(s)$ of $I$ is the non-zero polynomial $b(s)\in\bC[s]$ of minimal degree satisfying
$$
b(s_1+\ldots+s_r)\prod_{i=1}^rf_i^{s_i}=\sum_{k=1}^rP_kf_k\prod_{i=1}^rf_i^{s_i}
$$
for some $P_k$ in $\cD_X[s_{i,j}]_{i,j}$, where: $f_1,\ldots, f_r$ is a set of generators of $I$; $s_i$ are independent variables; $\cD_X$ is the ring of algebraic linear partial differential operators on $X$ acting naturally on $\cO_X[\prod_if_i^{-1},s_1,\ldots,s_r]\prod_if_i^{s_i}$; on the latter there is also an action $s_{i,j}=s_it_i^{-1}t_j$ with $t_k$ acting $\cD_X$-linearly by $t_k(s_i)=s_i$ if $i\neq k$, and $t_k(s_k)=s_k+1$. 

If $x\in X$ and one replaces algebraic functions by germs of analytic functions at $x$, one obtains by definition the {\it local Bernstein-Sato polynomial} $b_{I,x}(s)$ of $I$ at $x$.
\end{defn}

\begin{theorem}\label{thmBfct} Let $I\subset\cO_X$ be an ideal of regular functions on a smooth affine variety $X$, and let $Z$ is the associated subscheme. Then:
\begin{enumerate}[(1)]
\item The polynomials $b_I(s)$, $b_{I,x}(s)$ do not depend on the choice of generators for $I$, are non-zero, and all their roots are negative rational numbers. Moreover, $b_{I,x}$ is principal and generated by $s+\codim_{(X,x)}(Z,x)$ if and only if $Z$ smooth at $x$.

\item There is a finite stratification into locally closed subsets of $Z$ such that $b_{I,x}(s)$ is piecewise constant for $x\in Z$.

\item The polynomial $b_I(s)$ is the lowest common multiple of all $b_{I,x}(s)$ with $x\in Z$. 

\item The polynomial $b_Z(s):=b_I(s-\codim_XZ)$ depends only on the scheme $Z$.

\item The negative of the maximal root of  $b_{I,x}(s)$ (resp. $b_I(s)$) equals the log canonical threshold $\lct_x(X,Z)$ (resp. $\lct(X,Z)$).

\item If $Z$ is a reduced complete intersection in $X$ (resp. at $x\in Z$), then $Z$ has rational singularities (resp. at $x\in Z$) if and only if the maximal root of $b_I(s)$ (resp. $b_{I,x}(s)$) is the negative of the codimension of $Z$ in $X$ (resp. local codimension at $x$) and has multiplicity one. 
\end{enumerate}
\end{theorem}

Parts (1) and (4)-(6) for $b_I(s)$ are shown in  \cite{BMS}, however, the theorem for principal ideals $I$ has a longer history, see \cite{Mon}. By reduction to the principal ideal case \cite{Mus-b}, the rest of the results follows. 

\begin{defn}
With $X, Z, I, x$ as in (6), and assuming that $Z$ is not smooth at $x$,  the {\it minimal exponent} $\al(X,Z)$ of $(X,Z)$ is the negative of the maximal root of the polynomial $b_I(s)/(s+\codim_XZ)$. Similarly, $\al_x(X,Z)$ is defined using $b_{I,x}(s)$.
\end{defn}

\begin{thrm}\label{propMinE} With the same setup, assume $Z$ is an effective divisor on $X$.
\begin{enumerate}[(1)]

\item {\rm{(}\cite[Thm. E, (3)]{MP2}\rm{)}}  If the multiplicity $e$ of $Z$ at $x$ is $\ge 2$, then  $\al_x(X,Z)\le (\dim X)/e$.

\item {\rm{(}\cite[Cor. D]{MP2}\rm{)}} If $f:Y\to X$ is a log resolution of $(X,Z)$ as in Definition \ref{defnLR} such that it is an isomorphism over $X\setminus Z$ and the strict transforms of the irreducible components of $Z$ are mutually disjoint, then the minimal exponent of $\al(X,Z)\ge \min_{j}\{\nu_j/N_j\}$, where the minimum runs over the exceptional divisors $E_j$.

\item {\rm{(}\cite[Lemma 7.5]{MP3}\rm{)}} If $H$ is a general smooth hypersurface in $X$ then  $\al(X,Z)\le\al(H,Z|_{H})$.

\end{enumerate}

\end{thrm}

The following is a version of the Monodromy Conjecture, see \cite{DL}:

\begin{conj}\label{conjMC}
Let $X$ be a smooth affine variety, $Z$ a closed subscheme given by an ideal $I\subset\cO_X$, and $x\in Z$. Then $b_I(s)\cdot Z^{top}_{X,Z}(s)$ and $b_{I,x}(s)\cdot Z^{top}_{X,Z,x}(s)$ have no poles. 
\end{conj}

\begin{defn}
If $X$ is a normal variety with $K_X$ a $\bQ$-Cartier divisor, and $\ord_E$ is a divisorial valuation of the function field of $X$ with non-empty center $c_X(E)$ on $X$, let $a_E(X):=1+\ord_EK_{Y/X}$ where $Y\to X$ is a resolution of $X$ such that $c_Y(E)$ is a divisor. Then $a_E(X)$ is independent of the choice of resolution. If $W\subset X$ is a closed subset, the {\it minimal log discrepancy} of $X$ along $W$ is 
$$
\mld(W;X):=\inf_E\{a_E(X)\mid c_X(E)\subset W\}.
$$
\end{defn}

\begin{defn}{\rm{(}\cite{MP2}\rm{)}} If $X$ is a smooth variety, $Z$ a reduced divisor, and $k\in\bN$, the {\it $k$-Hodge ideal} $I_k(Z)\subset\cO_X$ of $(X,Z)$ is determined by the equality $$F_k(\cO_X(*Z))=I_k(Z)\otimes_{\cO_X}\cO_X((k+1)Z)$$
where $F$ is Saito's Hodge filtration.
\end{defn}

\begin{rmk}\label{rmkDm1}
The quasi-coherent $\cO_X$-module $\cO_X(*Z)$ is a regular holonomic left $\cD_X$-module. It underlies the mixed Hodge module $j_*\bQ_U^H[d]$ where $j:U=X\setminus Z\to X$ is the open embedding of the complement of $Z$ and $d=\dim X$. Thus Saito's theory of mixed Hodge modules endows $\cO_X(*Z)$ with a Hodge filtration $F_\lbul$ by coherent $\cO_X$-modules and a finite weight filtration $W_\lbul$ by holonomic $\cD_X$-modules.
\end{rmk}

\begin{defn}
Let $X$ be a smooth variety, $Z$ a proper closed subscheme, and $k\in \bN$.

\begin{enumerate}
\item The {\it $k$-local cohomology sheaf} $\cH^k_Z(\cO_X)$ is  the $k$-th derived functor of the functor assigning the subsheaf of $\cO_X$ of local sections with support in $Z$.  The sheaf $\cH^k_Z(\cO_X)$ only depends on the support of $Z$ and not on the scheme structure.
\item If $Z$ is a variety of codimension $c$, the {\it intersection homology} $\cD_X$-module $\cL(Z,X)$ is the smallest $\cD_X$-submodule of $\cH^c_Z(\cO_X)$ that coincides with $\cH^c_Z(\cO_X)$ generically.
\end{enumerate}
\end{defn}

\begin{rmk}\label{rmkDm2} If $Z$ is a reduced divisor then $\cH^1_Z(\cO_X)\simeq \cO_X(*Z)/\cO_X$. In general, $\cH^k_Z(\cO_X)$ is a regular holonomic left $\cD_X$-module. It underlies the mixed Hodge module $\HH^k(i_*i^!\bQ_X^H[d])$. Thus $\cH^k_Z(\cO_X)$ is also endowed with a Hodge filtration $F_\lbul$ by coherent $\cO_X$-modules and a finite weight filtration $W_\lbul$ by holonomic $\cD_X$-modules. The intersection homology module $\cL(Z,X)$ underlies the intersection complex pure Hodge module $\IC_{Z}\bQ^H$.
\end{rmk}

\begin{defn}
Let $m\in\bN$. The {\it $m$-jet space of  $\bA^n$} is the affine space $$\cL_m(\bA^n):=\Hom_{K-\text{alg}}(K[x_1,\ldots,x_n],K[t]/t^{m+1})\simeq \bA^{n(m+1)}.$$
If $X$ is a closed subscheme of $\bA^n$ defined by an ideal $I$, the {\it $m$-jet space of $X$} is the closed subscheme of $\cL_m(\bA^n)$ of $m$-jets vanishing on $I$,
$$
\cL_m(X):=\{\gamma\in \cL_m(\bA^n)\mid \gamma(I)= 0 \text{ in }K[t]/t^{m+1}\}.
$$
The scheme structure on $\cL_m(X)$ is as follows. Let $x,x',x'',\ldots,x^{(m)}$ be the coordinates on $\cL_m(\bA^n)$, where  $x^{(k)}=(x_1^{(k)},\ldots,x_n^{(k)})$. Let $x(t)= x+x't+x''t^2+\ldots+x^{(m)}t^m$. Let $f_j(x)$ be a set of generators of $I$. Plugging $x(t)$ instead of $x$,  set
$f_j(x(t))= f_j+f_j't+f_j''t^2+\ldots+f_j^{(m)}t^m$ mod $t^{m+1}$. The scheme $\cL_m(X)$ is cut out by the ideal generated by the polynomials $f_j^{(k)}$ with $1\le k\le m$ and it represents the functor $S\mapsto \Hom(S\times\spec K[t]/t^{m+1},X)$ from $K$-schemes to sets, see \cite[Ch. 3, \S 2]{CNS}.
 \end{defn}

\bigskip

\part{Other applications of $\Linf$ pairs}\label{part3}

\bigskip

\section{Compact K\"ahler manifolds} 

In this section we review applications of deformation theory with cohomology constraints to objects on  compact K\"ahler manifolds from \cite{BW}. The common point of these applications is that the controlling dgl pair $(C,M)$ is {\it formal}. We recall that this means that the dgl pair is equivalent to its cohomology dgl pair $(HC,HM)$ endowed with zero differentials. Equivalently, the controlling cohomology $\Linf$ pair structure on $((HC,l_*),(HM,m_*))$ has $l_n=0$ and $m_{n}=0$ for $n\neq 2$. Hence we only have the products $l_2:(HC)^{\otimes 2}\to HC$ and $m_2:HC\otimes HM\to HM$ as part of the $\Linf$ pair structure. The equations (\ref{eqLMC}) and (\ref{eqLJI}) simplify thus drastically.

\subs{\bf Stable holomorphic vector bundles with zero total Chern class.} Consider the moduli space $\cM$ of stable rank $n$ holomorphic vector bundles $E$ with vanishing Chern classes on a compact K\"ahler manifold $X$. These holomorphic vector bundles are the ones that admit flat unitary connections. In $\cM$ consider the cohomology jump loci
$$
\cV^{pq}_k(F)=\{E\in\cM \mid h^q(X, E\otimes F \otimes  \Omega^p_X)\geq k \}
$$
with the natural scheme structure, for fixed $p$ and fixed poly-stable bundle $F$ with vanishing Chern classes. The tensor products are over $\cO_X$. Fix $E$ as above. The task is to describe the formal completion of $\cV^{pq}_k(F)$ at $E$, which we denote by $ \cV^{pq}_k(F)_{(E)}$. This deformation problem with cohomology constraints is controlled by the formal cohomology $\Linf$ pair
$$
(\HH^\ubul(X,\enmo(E)),\HH^\ubul(X,E\otimes F \otimes  \Omega^p_X)).
$$
Stability is needed to simplify the equivalence relation in (\ref{eqLMC}) and (\ref{eqLJI}).
Define 
\begin{align*}
\sQ(E) &=\{\eta\in \HH^1(X, \enmo(E))\mid
\eta\wedge\eta=0\in \HH^2(X, \enmo(E))\}, \\
\sR^{pq}_k(E; F) & =\{\eta\in \sQ(E) \mid  \dim \HH^q(\HH^\ubul(X, E\otimes F\otimes \Omega^p_X),\eta \wedge\cdot )\geq k\},
\end{align*}
endowed with the natural scheme structures.

\begin{thrm} Let $X$ be a compact K\"ahler manifold. Let $E$ and $F$ be a stable and, respectively, a poly-stable holomorphic vector bundle with vanishing Chern classes on $X$. Then:
\begin{enumerate}
\item There is an isomorphism of formal schemes $\sM_{(E)}\cong \sQ(E)_{(0)}$ inducing for every $k$ an isomorphism of formal schemes
$$
 \cV^{pq}_k(F)_{(E)}\cong \sR^{pq}_k(E;F)_{(0)}.
$$ 
\item If $k=h^q(X, E\otimes F\otimes \Omega^p_X)$, then $\sV^{pq}_k(F)$ has  quadratic algebraic singularities at $E$. 
\item\label{eqFs} If $F=\cO_X$ and $n=1$ then $Sing (\cV^{pq}_k) \subset (\cV^{pq}_{k+1})_{red}$.
\end{enumerate}
\end{thrm}

This generalized some results of   \cite{n, GM, gl1,gl2, M-a, ma, w}. Last two parts follow a general pattern in presence of formality and trivial equivalence relations in (\ref{eqLMC}) and (\ref{eqLJI}): for $E\in \sV^{pq}_k(F)\setminus \sV^{pq}_{k+1}(F)$, the locus $\sV^{pq}_k(F)$ is locally around $E$ cut  by linear forms out of the moduli space, so it as as singular as the moduli space itself. 

Note that in (\ref{eqFs}) the reverse inclusion $  (\cV^{pq}_{k+1})_{red}\subset Sing (\cV^{pq}_k)$ is true Zariski-locally at $E$ if ``generic vanishing" holds, that is, $\cV^{pq}_{1}\subsetneq \cM=\pic^\tau(X)$ Zariski-locally at $E$. This follows by the very general Theorem \ref{thmTgxxy} since $\cM$ is smooth in this case. 
For global results in the situation of (\ref{eqFs}) see \cite{BWr}. Due to these global results one can remove ``Zariski-locally at $E$" from the preceeding discussion. Hence in presence of generic vanishing in (\ref{eqFs}), for example in the situation of \cite[Thm. 2]{gl1},   we have $Sing (\cV^{pq}_k) = (\cV^{pq}_{k+1})_{red}$.

\subs{\bf Irreducible complex local systems.}\label{subIrc} Consider the moduli space $\mb$ of irreducible rank $n$ complex local systems $L$ on a compact K\"ahler manifold $X$. Consider the cohomology jump loci
$$
\cV^i_k(W)=\{ L\in \mb\mid \dim_\bC \HH^i(X,L\otimes_\bC W)\ge k \}
$$
with the natural scheme structure, for a fixed semi-simple local system $W$ of any rank. This deformation problem with cohomology constraints is controlled by the formal cohomology $\Linf$ pair
$$
(\HH^\ubul(X,\enmo(L)),\HH^\ubul(X,L \otimes_{\bC}  W)).
$$
Define
$$
\sQ(L) =\{\eta\in \HH^1(X, \enmo(L)) \mid 
\eta\wedge\eta=0\in \HH^2(X, \enmo(L))\},$$
$$
\sR^i_k(L;W)  =\{ \eta\in \sQ(L)\mid \dim \HH^i(H^\ubul(X, L\otimes W), \eta\wedge \cdot)\geq k \},
$$
endowed with the natural scheme structures. 

\begin{thrm}\label{thmIrrLS} Let $X$ be a compact K\"ahler manifold. Let $L$ be an irreducible local system on $X$, and let $W$ be a semi-simple local system. Then:
\begin{enumerate}
\item There is an isomorphism of formal schemes
$
(\mb)_{(L)}\cong \sQ(L)_{(0)}
$
inducing for every $k$ an isomorphism of formal schemes
$$
\cV^i_k(W)_{(L)}\cong\, \sR^i_k(L;W)_{(0)}. 
$$
\item If $k=h^i(X, L\otimes W)$, then $\cV^{i}_k(W)$ has  quadratic algebraic singularities at $L$. 
\item\label{eqWl} If $W=\bC_X$ and $n=1$ then $Sing (\cV^{i}_k) \subset (\cV^{i}_{k+1})_{red}$, with equality locally at $L$ when ``generic vanishing" $\cV^i_1\subsetneq \cM_B = \Hom(\HH_1(X,\bZ),\bC^*)$ holds locally at $L$.
\end{enumerate}
\end{thrm}
This  generalized a result of \cite{PS} to which we  refer for generic vanishing. For global results see \cite{BWa}. Due to these global results one can remove ``locally at $L$" from (\ref{eqWl}). 
All the results here hold more generally for semi-simple local systems $L$.

\begin{rmk} If  $X$ is a smooth projective complex variety,  the nonabelian Hodge theory of Simpson implies that  $\mb$ and $\mdr$ are isomorphic as analytic spaces, where $\mdr$ is the moduli space of stable flat bundles of rank $n$ on $X$. Since this induces isomorphisms on the cohomology jump loci, the deformation problems with cohomology constraints are the same for irreducible local systems and stable flat bundles. 
\end{rmk}

\subs{\bf Stable Higgs bundles with zero total Chern class.} We assume here that $X$ is a smooth projective complex variety. Consider the moduli space $\mdol$ of stable Higgs bundles  $E=(E,\theta)$ of rank $n$ with $c(E)=0$. One has the Dolbeault cohomology $\HH^\ubul_{\rm{Dol}}(X,E):={\HH}^\ubul(X, (E\otimes\Omega_X^\ubul,\theta\wedge \cdot))$ associated to $(E,\theta)$, see \cite{s1}. Let $F=(F,\phi)$ is a poly-stable Higgs bundle with vanishing Chern classes and
\begin{align*}
\sV^i_{k}(F) =\{  E  \in \mdol\mid \dim \HH^i_{\rm{Dol}}(X, E\otimes F)\ge k\}
\end{align*}
where the tensor product is of Higgs bundles. This deformation problem with cohomology constraints is controlled by the formal cohomology $\Linf$ pair
$$
(\HH^\ubul_{\rm{Dol}}(X,\enmo(E)),\HH^\ubul_{\rm{Dol}}(X,E\otimes F)).
$$
Define
$$
\sQ(E) =\{\eta\in \HH^1_{\rm{Dol}}(X, \enmo(E)) \mid 
\eta\wedge\eta=0\in \HH^2_{\rm{Dol}}(X, \enmo(E))\},$$
$$
\sR^i_k(E;F)  =\{ \eta\in \sQ(E)\mid \dim \HH^i(\HH^\ubul_{\rm{Dol}}(X, E\otimes F), \eta\wedge \cdot)\geq k \},
$$
endowed with the natural scheme structures. 

\begin{thrm}\label{thmHiggs} Let $X$ be a smooth projective complex variety. Let $E$ be a stable Higgs bundle with $c(E)=0$, and let $F$ be a poly-stable Higgs bundle with $c(F)=0$. Then:
\begin{enumerate}
\item There is an isomorphism of formal schemes
$
(\mdol)_{(E)}\cong \sQ(E)_{(0)}
$
inducing for every $k$ an isomorphism of formal schemes
$$
\sV^i_k(F)_{(E)}\cong\, \sR^i_k(E;F)_{(0)}. 
$$
\item If $k=\dim \HH^i_{\rm{Dol}}(X, E\otimes F)$, then $\sV^{i}_k(F)$ has  quadratic algebraic singularities at $E$. 
\item\label{eqWER} If $F=(\cO_X,0)$  and $n=1$ then $Sing (\sV^{i}_k) \subset (\sV^{i}_{k+1})_{red}$, with equality locally at $E$ when ``generic vanishing" $\sV^i_1\subsetneq \mdol$ holds locally at $E$.
\end{enumerate}
\end{thrm}
This  generalized a result of \cite{s1}. For some global results we refer to \cite{BWr}; this allows one to remove ``locally at $E$" from (\ref{eqWER}).

\subs{\bf Representations of the fundamental group.}
Consider the moduli space $\mathbf{R}(X)=\homo(\pi_1(X, x), GL(n, \mathbb{C}))$ with the natural scheme structure. Here $x\in X$ is a fixed point. Every closed point $\rho\in\mathbf{R}(X)$ corresponds to a rank $n$ local system $L_\rho$ on $X$. Let $W$ be a semi-simple complex local system of any rank on $X$. Consider the cohomology jump loci 
$$\ti\cV^i_k(W)=\{\rho\in \mathbf R(X)\,|\, \dim \HH^i(X, L_\rho\otimes_\bC W)\geq k\}
$$
with the natural scheme structure. This deformation problem 
with cohomology constraints for a semi-simple representation $\rho$  is closely related to that of the semi-simple local system $L_\rho$. The results in \ref{subIrc} extend to  semi-simple local systems. The controlling formal cohomology $\Linf$ pair here is
$$
(\HH^\ubul(X,\enmo(L_\rho)),\HH^\ubul(X,L_\rho \otimes_{\bC}  W)) \times \mathfrak g/\mathfrak h,
$$ 
with the $\Linf$ structure extended trivially over the vector space  $\mathfrak g/\mathfrak h$, with $\mathfrak g =\enmo(L_\rho)|_x$  the fiber of the endomorphism local system and $\mathfrak h\subset\mathfrak g$  the image under restriction to $x$ of the vector space of global sections $\HH^0(X,\enmo(L_\rho))$.

\begin{thrm}\label{thmRPP}
Let $X$ be a compact K\"ahler manifold,  $\rho\in\mathbf{R}(X)$ be a semi-simple representation, and $W$ a semi-simple local system on $X$. Then, with the notation as in \ref{subIrc}:
\begin{enumerate}
\item There is an isomorphism of formal schemes
$
\mathbf{R}(X)_{(\rho)}\cong (\sQ(L_\rho)\times \mathfrak g/\mathfrak h)_{(0)} 
$
inducing for every $k$  isomorphisms of formal schemes
$$
\ti\cV^i_k(W)_{(\rho)} \cong\cV^i_k(W)_{(L_\rho)} \cong (\sR^i_k(L_\rho;W)\times \mathfrak g/\mathfrak h)_{(0)} . 
$$
\item If $k=\dim \HH^i(X, L_\rho\otimes_\bC W)$, then $\ti\cV^{i}_k(W)$ has  quadratic algebraic singularities at $\rho$. 
\item\label{eqWAQ} If $W=\bC_X$  and $n=1$ then $Sing (\ti\cV^{i}_k) \subset (\ti\cV^{i}_{k+1})_{red}$, with equality locally at $\rho$ when ``generic vanishing" $\cV^i_1\subsetneq \mdol$ holds locally at $L_\rho$.
\end{enumerate}
\end{thrm}

This generalized some results of   \cite{GM, s1, dp}. As before, global results \cite{BWr} allow one to remove ``locally at $\rho$" from (\ref{eqWAQ}).

\section{Other topological restrictions}

Theorems \ref{thmIrrLS} and \ref{thmRPP} impose restrictions on the homotopy types of compact K\"ahler manifolds. We review now applications of deformation theory with cohomology constraints to topological restrictions on other types of spaces from \cite{BR}.  The main idea here is that Deligne's weight filtration from mixed Hodge theory is compatible with the higher order multiplication maps, e.g. Massey products on $\HH^\ubul(X,\bC)$ if $X$ is complex algebraic variety \cite{BR, CSo}. It is mentioned in the introduction of \cite{DGMS} that this idea led to the formulation of their  result that the de Rham complex of compact K\"ahler manifolds is a formal dga.

If $\Linf$ pair structure on
$
(\HH^\ubul(X,\bC), \HH^\ubul(X,\bC)),
$ is obtained from the de Rham complex, then the pair controls the deformations  with cohomology constraints of the constant sheaf $\bC_X$.  If $W_0\HH^1(X,\bC)=0$ then the compatibility with the weight filtration implies that only finitely many $\Linf$ module multiplication maps are non-zero by degree reasons. We will see below that this has major consequences. The condition $W_0\HH^1(X,\bC)=0$ is known to be a topological condition on complex algebraic varieties, by M. Saito, and it is satisfied if the singularities of $X$ are not too wild, e.g. the condition holds for normal, or even unibranch singularities.

More generally, let $X$ be a  connected topological space  having the homotopy type of a finite CW-complex.
 The $\Linf$ pair $
(\HH^\ubul(X,\bC), \HH^\ubul(X,L))
$ controls the deformations with cohomology constraints of a rank one complex local system $L$. Here, 
 if $\mb (X)$  denotes the space of all rank one $\bC$-local systems on $X$, then $\mb(X)$ is identified with the group $\Hom (\pi_1(X),\bC^*)$ of rank one representations of the fundamental group $\pi_1(X)$ based at a fixed point of $X$. This is an algebraic group, the product of a finite abelian group with the complex affine torus $(\bC^*)^b$ where $b$ is the first Betti number of $X$.
The {\it cohomology jump loci} are defined by
$$
\cV^i_k(X)=\{L\in \mb (X)\mid \dim \HH^i(X,L)\ge k\} 
$$ 
with the natural structure of closed subschemes of $\mb (X)$. The cohomology jump loci are homotopy invariants of the topological space $X$. Moreover,  $\cV^1_k(X)$ depends only on $\pi_1(X)$ and $k$.

\begin{theorem}
\label{thrmLoc} Let $X$ be a connected topological space, homotopy equivalent to a finite CW-complex.
Let $L$ be a rank one local system on $X$. Let $(\HH^\ubul(X,\bC), \HH^\ubul(X,L))$ be endowed with an $\Linf$ pair structure via homotopy transfer from the dgl pair consisting of Sullivan's de Rham complexes for $\bC_X$ and $L$, respectively. If there exists $n_0$ such that the $L_\infty$ module structure maps $m=(m_n)_{n\ge 2}$  of  $\HH(X,L)$ over  $\HH(X,\bC)$ satisfy $$m_n(\omega,\ldots,\omega,{\eta})=0$$  for all $n>n_0$, $\omega\in \HH^1(X,\bC)$, ${\eta}\in \HH^\ubul(X,L)$, then every irreducible component of the algebraic set $\cV^i_k(X)$ passing through $L$ is a translated complex affine subtorus of $\mb(X)$.
\end{theorem} 

\noindent
{\it Sketch of the proof.} 
Since we deal with rank one local systems, there is no equivalence relation in  (\ref{eqLMC}) and (\ref{eqLJI}) to mod out by.  Since the de Rham complex of $X$ is a cdga, it has zero Lie bracket as dgla. This implies that the $\Linf$ algebra structure on $\HH^\ubul(X,\bC)$ is trivial, that is, all products are zero. This simpifies (\ref{eqLMC}), so that $\Def(HC)$ is pro-represented by the formal neighborhood of the origin in $\HH^1(X,\bC)$.

On the other hand, the $\Linf$ module structure on $\HH^\ubul(X,L)$ is non-trivial even if $L=\bC_X$, 
in which case it is induced by wedging of  forms, that is, the $\Linf$ module structure remembers the $A_\infty$-algebra structure on $\HH^\ubul(X,\bC)$. Since there are only finitely many $\Linf$ module multiplications, the functors (\ref{eqLJI}) are pro-represented by the formal neighborhood at the origin of  the closed subschemes of $\HH^1(X,\bC)$ given by cohomology jump ideals of the universal complex, with finitely many terms, interpolating the complexes (\ref{eqLJI}). Let us denote by $\sR^i_k(X,L)$ these affine schemes. We have thus a  commutative diagram with vertical arrows isomorphisms of formal germs:
$$
\xymatrix{
\mathcal{R}^i_k (X,L)_{(\mathbf{0})} \ar@{^{(}->}[r]  \ar[d]^*[@]{\sim}& \HH^1(X,\bC)_{(\mathbf{0})} \ar[d]^*[@]{\sim} \\
\cV^i_k(X)_{(L)} \ar@{^{(}->}[r] & \mb(X)_{(L)}.
}
$$
Moreover, the right-most isomorphism is induced by the exponential map 
$$\exp: \bC^b=H^1(X,\bC)\ra (\bC^*)^b$$
for the connected component $(\bC^*)^b$ of $\mb(X)$ containing the constant sheaf. One applies now the following Ax-Lindemann type result.

\begin{prop}\label{propAx}
Suppose $(W, 0)$ and $(V, 1)$ are analytic germs of two algebraic sets in $\bC^n$ and $(\bC^*)^n$, respectively. If the exponential map $\exp: \bC^n\to (\bC^*)^n$ induces an isomorphism between $(W, 0)$ and $(V, 1)$, then $(V, 1)$ is the germ of a finite union of complex affine subtori. 
\end{prop}

Combining the weight condition with the above theorem one obtains the next results which hold for spaces admitting mixed Hodge structures, not only for complex algebraic varieties. We recall the following definition.
Let $W$ be a complex projective variety, $Z$ and $Z'$ closed subschemes, $Y=Z\cup Z'$, and assume that the singular locus of $W$ is contained in $Y$. The {\it link of $Z$ in $W$ with $Y$ removed} is the complement $\cL=\cL(W,Y,Z):=T-Y$ for a nice neighborhood $T$ of $Z$ in $W$. If $Z=\{x\}$ is an isolated singularity of $W$ and $Z'$ is empty, then $\cL$ is the usual link of the singularity $(W,x)$.

\begin{theorem}\label{corW}
Let $X$ be: 
\begin{itemize}
\item a connected complex algebraic variety, possibly reducible,
\item a connected component of the link $\cL(W,Y,Z)$, or
\item a connected component of the Milnor fiber of the germ of a holomorphic function $f:(\bC^n,0)\ra (\bC,0)$.
\end{itemize}
If $W_0\HH^1(X,\bC)=0$, then each irreducible component of the algebraic set $\cV^i_k(X)$ containing the constant sheaf is a complex affine subtorus of $\mb(X)$.
\end{theorem}

For more global results of this type for complex algebraic varieties, obtained by other methods, we refer to \cite{BWq, BWa, ek}.

\bigskip

\part{$\Linf$ structures} \label{part2}

\bigskip

This part has two sections. Section  \ref{apxLinf} is dedicated to a  review of $\Linf$ structures. It contains the technical details behind the black box of Section \ref{secDFT} on deformation theory with cohomology constraints. In Section \ref{secTecCore} 
we give a proof of Theorem \ref{thmGenToCone} different than the proof in \cite{theta}.

\section{Review of $\Linf$ structures}\label{apxLinf}

In this section, we review   $L_\infty$ structures, deformation functors,  cohomology jump deformation subfunctors, and homotopy transfer theorems from \cite{Mane, KS, BR}. The material from this section is used in Part \ref{part1} of the article and in Section \ref{secTecCore} where we give another proof of Theorem \ref{thmGenToConeEt}.

\subs{\bf Signs, notation, d\'ecalage.} We work over a field $K$ of characteristic zero. Graded means $\bZ$-graded. If $V$ is a graded $K$-vector space, $T(V)$ denotes the graded tensor algebra and $\otimes$  its product, $\rS(V)=T(V)/\Span\{u\otimes v-(-1)^{|u||v|}v\otimes u\}$ denotes the graded commutative algebra and $\vee $  its product, $\Lambda(V)=T(V)/\Span\{u\otimes v+(-1)^{|u||v|}v\otimes u\}$ denotes the graded exterior algebra and $\wedge$  its product, where the spans are over  homogeneous elements $u$, $v$. 
For homogeneous elements $v_1,\ldots,v_n$ in $V$ and a permutation $\sigma\in \mathcal{S}_n$, the {\it Koszul sign} $\epsilon(\sigma)$ is defined by   
$v_1\vee\ldots \vee v_n =\epsilon(\sigma) v_{\sigma(1)}\vee\ldots \vee v_{\sigma(n)}$, so it also depends on  $|v_i|$. The {\it anti-symmetric Koszul sign} is $\chi(\sigma)=\mathrm{sign}(\sigma)\epsilon(\sigma)$,  equivalently, $v_1\wedge\ldots \wedge v_n =\chi(\sigma) v_{\sigma(1)}\wedge\ldots \wedge v_{\sigma(n)}$.  For homogeneous elements $v_1,\ldots,v_n$ in $V$ by $v=(v_1,\ldots,v_n)$ we  mean  $v_1\otimes \ldots \otimes _n$ if the context is clear to ease notation.

A permutation $\sigma\in \mathcal{S}_n$ on the set of $n$ elements is said to be an {\it$(i,n-i)$-unshuffle} if $
 \sigma(1)<\ldots < \sigma(i)$ and $
 \sigma(i+1)<\ldots < \sigma(n).
$
For $0\le i\le n$ denote by $\mathrm{Sh}(i,n-i)$ the set of $(i,n-i)$-unshuffles, consisting only of the identity  if $i=0$ by convention, 
and set 
 $\mathfrak{S}_n :=\lbrace (i,j,\sigma)\mid \sigma \in  \mathrm{Sh}(i,n-i), i\geq 1, i+j=n+1 \rbrace.$
 More generally, let
$\mathfrak{S}_{j,n}$ be the set of tuples $(k_1,\cdots,k_j, \tau)$ such that $k_i\geq 1,$ $k_1+\ldots+k_j=n$, and  $\tau\in \mathcal S_n$ is a permutation preserving the order within each block of length $k_i$.

 A graded  multilinear map $f: V^{\otimes n}\rightarrow V$  is {\it symmetric}, respectively {\it anti-symmetric}, if $$f(v_{\sigma(1)}, \ldots, v_{\sigma(n)} )=\eps(\sigma)f(v_{1}, \ldots, v_{n} ),\text{ respectively }
 \chi(\sigma)f(v_{1}, \ldots, v_{n} ),$$
for all $\sigma \in \mathcal{S}_n$ and homogeneous  $v_i\in V$. Equivalently, $f$ induces a graded linear map $\rS^n(V)\to V$, respectively $\Lambda^n(V)\to V$.

 The association $v\mapsto v\otimes 1+1\otimes v$ induces a coproduct $\Delta$ and a coalgebra structure on $\rS(V[1])$. One has the notion of coderivations on coalgebras. 
Set $\rrS(V[1])=\oplus_{i\geq 1}\rS^n(V[1])$ and consider it with the induced reduced symmetric coalgebra structure. A \emph{codifferential} on $\rrS(V[1])$ is a linear map $Q : \rrS(V[1]) \rightarrow \rrS(V[1])$ of degree $1$  such that $Q$ is a coderivation and $Q^2=0$. 

D\'ecalage allows one to pass from graded symmetric to graded anti-symmetric multilinear maps, see \cite[Prop. 10.6.2, Lemma 10.6.4, Ex. 11.8.12]{Mane}. Denote by $s:V[1]\to V$ the identity as a set map viewed as a graded linear map of degree 1.

\begin{prop} \label{decalage}   For every $n\ge 0$ there exists a  linear isomorphism of degree $n$, called d\'ecalage, 
$$
s^{\otimes n}: \rS^n(V[1]) \longrightarrow \Lambda^n(V)[n]$$
$$v_1\vee\cdots \vee v_n \mapsto  (-1)^{\sum_{i=1}^n(n-i)\left | v_i \right |}\left (sv_1\wedge\cdots \wedge sv_n  \right )
$$
for $v_i\in V[1]$ homogeneous of degree $|v_i|$.
\end{prop}

\begin{cor}\label{corDeca} If $V, V'$ are graded vector spaces, $i,n\in\bZ$, $n\ge 0$, then there is a linear isomorphism
$$
\dec:\Hom^i(\rS^n(V[1]), V') \xa{\sim} \Hom^{i-n}(\Lambda^n(V), V')$$
$$
 f  \mapsto (f\circ (s^{\otimes n})^{-1})[-n].
$$
Explicitly, for homogeneous $v_j\in V$ of degree $|v_j|$, $$ \dec(f)(v_1\wedge \cdots \wedge v_n) = (-1)^{\sum_{j=1}^n(n-j)(\left | v_j \right |-1)} f(s^{-1}v_1\vee\cdots \vee s^{-1}v_n).$$
\end{cor}

\begin{rmk}\label{rmkDecW}
A slight generalization involves another graded vector space $W$. There is a linear isomorphism $$
\dec:\Hom^i(\rS^n(V[1])\otimes W, V') \xa{\sim} \Hom^{i-n}(\Lambda^n(V)\otimes W, V')$$ defined by $f  \mapsto (f\circ (s^{\otimes n}\otimes \id_W)^{-1})[-n]$. Explicitly, $$ \dec(f)(v_1\wedge \cdots \wedge v_n\otimes w) = (-1)^{\sum_{j=1}^n(n-j)(\left | v_j \right |-1)} f(s^{-1}v_1\vee\cdots \vee s^{-1}v_n\otimes w).$$
\end{rmk}

\subs{\bf $\Linf$ algebras and modules.} There are  equivalent ways to define $\Linf$ structures.

\begin{propdef} \label{equivdeflalg}An {\it $L_\infty$    algebra}, or {\it strong homotopy Lie algebra}, is a graded vector space $L$ together with one of the following equivalent data:

(1) a codifferential $Q$ on $\rrS(L[1])$;

(2) a collection of graded symmetric multilinear maps $q_n:L[1]^{\otimes n} \rightarrow L[1]$, ${n\ge 1}$,
of degree $1$ such that
  for all $n\ge 1$ and homogeneous $a_i\in L[1]$,  
$$
   0= \sum_{(i,j,\sigma ) \in \mathfrak{S}_n } \epsilon(\sigma)q_j(q_i(a_{\sigma(1)} , \cdots , a_{\sigma(i)}), a_{\sigma(i+1)} , \cdots , a_{\sigma(n)}),
$$
and  we note that $\eps(\sigma)$ takes into account the degrees of $a_i$ in $L[1]$, not in $L$;

(3)   a collection of graded anti-symmetric multilinear maps
$ l_n: L^{\otimes n}\rightarrow L $, ${n\geq 1}$, such that $l_n$ has degree $2-n$, and
  for all $n\ge 1$ and homogeneous $a_i\in L$,  
$$
 0= \sum_{(i,j,\sigma ) \in \mathfrak{S}_n } \chi(\sigma)(-1)^{(j-1)}l_j(l_i(a_{\sigma(1)}, \ldots, a_{\sigma(i)}),a_{\sigma(i+1)}, \ldots, a_{\sigma(n)}).
$$
\end{propdef}

\begin{rmk}
In (2) one can replace $q_n$ by the induced linear map $q_n:\rS^n(L[1])\to L[2]$ and the commas (tensor products) in the identity by $\vee$. Similarly, in (3) one can replace $l_n$ by the induced linear map $l_n:\Lambda^nL\to L[2-n]$ and the commas in the identity by $\wedge$. 
\end{rmk}

\begin{rmk}
The equivalence between the definitions goes as follows, see  \cite[\S10, \S12]{Mane}, \cite[\S3.1]{KS}. 
The equivalence between the symmetric and the anti-symmetric formulations is given by setting $l_n=-\dec(q_n)$ where $\dec$ is the d\'ecalage isomorphism of Corollary \ref{corDeca} with $i=1$, $V=L$,  $V'=L[1]$. The sign convention is made to generalize the dgla case, see Remark \ref{dglala}. For the equivalence between the first two formulations,
given a codifferential $Q$ on $\rrS(L[1])$, set for $1\le j, n$, $$Q_n^j:= \mathrm{pr}_{\rS^j(L[2])}\circ Q_{\mid \rS^n(L[1])}:\rS^n(L[1]) \rightarrow \rS^j(L[2]). $$
 The  collection of maps
$Q_n^1:\rS^n(L[1]) \rightarrow L[2]$   gives the collection of graded symmetric linear maps $q_n$ as in (2).
Conversely, given  $ q_n$  as in (2), let $Q_n^1$ be the induced linear maps on the symmetric algebra of $L[1]$. Then the map defined for homogeneous vectors $a_i$ of $L[1]$ by
$$Q(a_1\vee \cdots \vee a_n): = \sum_{(i,j,\sigma ) \in \mathfrak{S}_n } \epsilon(\sigma)Q_i^1(a_{\sigma(1)} \vee \cdots \vee a_{\sigma(i)}) \vee a_{\sigma(i+1)}\vee \cdots \vee a_{\sigma(n)}
$$
defines a codifferential on $\rrS(L[1])$. 

\end{rmk}

  \begin{rmk}\label{rmkSignDiff} There is a different formulation of (3) for $L_\infty$ algebras in the literature, e.g. \cite{Al, L, LM, LS}. This is due to a different sign convention for the d\'ecalage isomorphism. The definition here follows \cite{KS, Mane}. The maps $l_n$ differ by a factor $(-1)^{n(n-1)/2}$ in these two definitions.  The definition of $l_n$ in \cite{Ge} differs by a factor $(-1)^n$ from ours. 
  \end{rmk}

  \begin{rmk} \label{dglala} Given an $L_\infty$ algebra $(L,l_1,l_2,\ldots)$, unwinding the definition, we have that $l_1$ is of degree $1$ and satisfies $l_1\circ l_1=0$, so  $(L,l_1)$ is a cochain complex. Any $L_\infty$ algebra with $l_n=0$ for $n\ge 3$ is a dgla, where the Lie bracket is $l_2$. Conversely, 
  any dgla is an $L_\infty$ algebra $(L,l_1,l_2, 0, 0,\ldots )$. 
   \end{rmk}

\begin{propdef}\label{lmalgebra}  A {\it morphism of $L_\infty$ algebras} between $(L,Q)$, or $(L,q_1,q_2,\ldots)$, or $(L,l_1,l_2,\ldots)$, and $(L',Q')$, or $(L',q'_1,q'_2,\ldots)$, or $(L',l'_1,l'_2,\ldots)$, is one of the following  equivalent data:

 \begin{enumerate}[(1)]

\item A morphism of coalgebras $F:\rrS(L[1])\to \rrS(L'[1])$ such that $F\circ Q=Q'\circ F$.

\item A  collection  of linear maps 
$\sf_n:\rS^n(L[1]) \rightarrow L'[1]$, $n\ge 1$,   such that for all $a_i$ homogeneous in $L[1]$
\begin{align*}
\sum_{i=1}^n q'_iF^i_n(a_1\vee\ldots \vee a_n)=\sum_{(i,j,\sigma ) \in \mathfrak{S}_n}\eps(\sigma)\sf_j(q_i(a_{\sigma(1)}\vee\ldots\vee a_{\sigma(i)})\vee\ldots\vee a_{\sigma(n)})
\end{align*}
where $F^i_n:\rS^n(L[1])\to \rS^i(L'[1])$ are defined recursively by $F^1_n=\sf_n$ and
$$
F^i_n(a_1\vee\ldots \vee a_n) =\frac{1}{i} \sum_{k=1}^{n-i+1}\sum_{\sigma\in \mathrm{Sh}(k,n-k)}\eps(\sigma)\sf_k(a_{\sigma(1)}\vee\ldots\vee a_{\sigma(k)})\vee F^{i-1}_{n-k}(a_{\sigma(k+1)}\vee\ldots\vee a_{\sigma(n)}).
$$

\item A collection of  graded anti-symmetric multilinear maps
$ f_n: L^{\otimes n} \rightarrow L' $, $n\geq 1$, such that $f_n$ has degree  $1-n$ and the maps $f_n$ satisfy a certain identity (which we do not write down, but see Remark \ref{expmor}.)
\end{enumerate}

\end{propdef}

\begin{rmk}\label{expmor}
The identity in (3) can be obtained by d\'ecalage from (2) by setting $f_n=\dec(\sf_n)$, see \cite[12.3]{Mane}. This identity is more involved: see \cite[Def. 2.3]{Al}, but note that $l_n$ there differs by $(-1)^{n(n-1)/2}$ from here, cf. Remark \ref{rmkSignDiff}, and the grading is reversed by a negative sign to work with chain complexes instead of cochain complexes as here. For the equivalence between the first two formulations, given a morphism of coalgebras $F$ compatible with the codifferentials, set for $1\le i,n$, $$F_n^i:= \mathrm{pr}_{\rS^i(L'[1])}\circ F_{\mid \rS^n(L[1])}:\rS^n(L[1]) \rightarrow \rS^i(L'[1]). $$ Then $\sf_n:=F^1_n$, as well as the other $F^i_n$, are as in (2). In particular $F_n^i$ depends solely on $F_k^1$ for $1\leq k \leq n-i+1$. Conversely, given $\sf$ and $F^i_n$ as in (2), then $F:=\sum_{i,n}F^i_n$ is as in (1).  See \cite[Prop. 12.2.3]{Mane} for details.
\end{rmk}

\begin{rmk}\label{Taylorcoefmor}
The morphism $F:(L,Q)\to (L',Q')$ of $\Linf$ algebras can be reconstructed from the maps $F^1_n=\sf_n$ in one shot by
$$F(a_1\vee\ldots \vee a_n) =  \sum_{j=1}^n\sum_{(k_1,\ldots , k_j,\tau ) \in \mathfrak{S}_{j,n} } \frac{\epsilon(\tau)}{j!} F_{k_1}^1(a_{\sigma(1)}\vee\ldots\vee a_{\sigma(k_1)}) \vee\ldots\vee F_{k_j}^1(a_{\sigma(n-k_j+1)}\vee\ldots\vee a_{\sigma(n)})$$
for $a_i\in L[1]$ homogeneous, \cite[Prop. 3.7]{KS}. We  refer to $F^1_n$ as the Taylor coefficients of $F$.
\end{rmk} 

\begin{rmk} Unwinding the definition, $F^1_1=\sf_1$ and $f_1=\dec(\sf_1)$ are morphisms of cochain complexes. Moreover,
the category of dgla's is a subcategory of the category on $L_\infty$ algebras. Nevertheless, it is not a full subcategory as there are $L_\infty$  morphisms between two dgla's that are not morphisms of dgla's. 
\end{rmk}

\begin{prop} \label{inverse} $\rm{(}$\cite[Cor. 11.5.5]{Mane}, \cite[Prop. 3.9]{KS}$\rm{)}$   \label{coalgiso} A  morphism  between $\Linf$ algebras $F:(L,Q)\to(L',Q')$ is an isomorphism if and only if $F_1^1$ is an isomorphism. In this case, the inverse $G$ of $F$ is determined by the  recursive formula: 
\be \label{inverseformula}
\begin{cases}
 G^1_1=(F_1^1)^{-1} &  \\ 
 G_n^1=-\Big(\sum_{i=1}^{n-1}G_i^1F^i_n\Big)(F_n^n)^{-1}& \text{ for } n\geq 2. 
\end{cases}
\ee
\end{prop}

\begin{defn} A morphism of $L_\infty$ algebras $f: L\rightarrow L'$ is a {\it weak equivalence} if the map of complexes $f_1: (L, l_1) \rightarrow (L', l_1')  $ is a quasi-isomorphism. If there exists a zig-zag of weak equivalences between two $\Linf$ algebras $L$ and $L'$, we say that $L$ and $L'$ are {\it homotopy equivalent}.
\end{defn}

We review now $\Linf$ modules, see \cite[\S7]{KS}. 

\begin{propdef} \label{lmodule} Let  $(L,Q)$, or $(L,q_1,q_2,\ldots)$, or  $(L,l_1,l_2,\ldots)$, be an $L_\infty$ algebra. An {\it $L_\infty$  module} over $L$ is a graded vector space $V$ together with any of the following equivalent data:
 \begin{enumerate}[(1)]

\item A  codifferential $\phi$ of degree $1$ on the cofreely cogenerated comodule $\rS(L[1])\otimes V$ over $(\rS(L[1]),Q)$ where the definition of $Q$ is extended by setting $Q_0=0$.  (A codifferential $\phi$ must satisfy by definition the relation \cite[(7.2)]{KS} with $Q$, and $\phi\circ\phi=0$.)

\item A collection of graded linear maps $\phi_n:\rS^{n-1}(L[1])\otimes V\to V[1]$, ${n\geq 1}$, satisfying a certain identity.
 
\item  A collection of graded multilinear maps 
$ m_n:  L^{\otimes (n-1) } \otimes V \rightarrow V  $, ${n\geq 1}$, such that $m_n$ has degree $2-n$,  $m_n$ are anti-symmetric with respect to  $L^{\otimes (n-1)}$ and satisfy a certain identity. \end{enumerate}

\end{propdef}

\begin{rmk}\label{equivLmod} If $\phi$ is a codifferential as in (1) set
$$
\phi_n:=\mathrm{pr}_{V[1]}\circ F_{\mid \rS^{n-1}(L[1])\otimes V}.
$$
The fact that $\phi$ is a coderivative implies that  one reconstructs $\phi$ from $\phi_n$ by
\be
\begin{aligned}\label{eqcomo}
 \phi(a_1\vee \cdots \vee  a_{n-1} \;\otimes   & \;v)= Q(a_1\vee \ldots  \vee a_{n-1} \otimes v) +\\ + \sum_{i=0}^{n-1}\sum_{\sigma\in \mathrm{Sh}(i,n-1-i)}  &(-1)^{\sum_{s=1}^i \left | a_{\sigma(s)} \right |}  \epsilon(\sigma) a_{\sigma(1) } \vee \cdots \vee  a_{\sigma(i) }  \otimes  \phi_{n-i}(a_{\sigma(i+1)} \vee \ldots  \vee a_{\sigma(n-1)} \otimes v ) 
\end{aligned}
\ee
for homogeneous $a_i\in L[1]$ and $v\in V$, see \cite[Rem. 7.2]{KS}. The extra condition  $\phi\circ\phi=0$ gives the desired identity that $\phi_n$ have to satisfy in (2). D\'ecalage gives the identity that the maps $m_n=-\dec(\phi_n)$ have to satisfy in (3), with $\dec$ as in Remark \ref{rmkDecW}, where the sign convention here is made to recover the dgl module case, see  Remark \ref{remDlgmod}.  The explicit identity in (3) is written down in \cite[Def. 2]{L}, \cite[Def. 2.2]{Al}, where $m_n$ differ from here by $(-1)^{n(n-1)}$, see Remark \ref{rmkSignDiff}.
\end{rmk}

\begin{rmk}\label{remDlgmod}
Unwinding the definition, $ m_1$ has degree one and satisfies $m_1\circ m_1=0$,  so $(V,m_1)$ is a cochain complex. If $(L,l_1,l_2)$ is a dgla then the $\Linf$ module $(V,m_1,m_2,\ldots)$ over $L$ is a dgl $L$-module if $m_n=0$ for $n\ge 3$, see \cite[Ex. 7.3]{KS}. Conversely, every dgl $L$-module is an $\Linf$ $L$-module with vanishing  multiplication maps $m_n$ for $n\ge 3$.
\end{rmk}

There is another equivalent definition of $\Linf$ modules, see \cite[Prop. 7.5]{KS}:

\begin{prop}
If $(L,Q)$ is an $\Linf$ algebra and $V$ is a graded vector space, there is a natural dgla structure on $\Hom(\rS(L[1])\otimes V,V)$ with differential $\partial\phi=-(-1)^{|\phi|}\phi\circ(Q\otimes\id)$ and bracket induced by the product $\phi\bullet\psi=\phi\circ(\id\otimes\psi)\circ(\Delta\otimes \id)$, such that the Maurer-Cartan elements $\phi$ of this dgla can be identified with the $\Linf$ module structures on $V$.
\end{prop}

\begin{rmk}
If $F:(L,Q)\to (L',Q')$ is a morphism of $\Linf$ algebras and $(V,\phi')$ is an $\Linf$ module over $(L',Q')$, then $\phi=\phi'\circ(F\otimes \id)$ defines a natural $\Linf$ module structure  on $V$ over $(L,Q)$.
\end{rmk}

\begin{propdef}\label{defLmod} Let  $(L,Q)$, or $(L,q_1,q_2,\ldots)$, or  $(L,l_1,l_2,\ldots)$, be an $L_\infty$ algebra. A {\it morphism of $L_\infty$  modules} between $(V,\phi)$, or $(V,\phi_1,\phi_2,\ldots)$, or $(V,m_1,m_2,\ldots)$, and $(V',\phi')$, or $(V',\phi'_1,\phi'_2,\ldots)$, or $(V',m'_1,m'_2,\ldots)$,  is any of the following equivalent data:

 \begin{enumerate}[(1)]

\item A morphism of comodules $\kappa:\rS(L[1])\otimes V\to \rS(L[1])\otimes V'$ such that $\kappa\circ \phi=\phi'\circ\kappa$.

\item A collection of graded linear maps $\kappa_n:\rS^{n-1}(L[1])\otimes V\to V'$, $n\ge 1$, satisfying a certain identity.

\item A collection of graded multilinear maps 
$ g_n: L^{\otimes (n-1) } \otimes V \rightarrow V' $, ${n\geq 1}$, such that $g_n$ has degree  $1-n$,  $g_n$ are   anti-symmetric with  respect to $L^{\otimes (n-1)}$, and
satisfy a certain identity.
 \end{enumerate}

\end{propdef}

\begin{rmk}
Given $\kappa$ as in (1), set
$$
\kappa_n:=\mathrm{pr}_{V'}\circ F_{\mid \rS^{n-1}(L[1])\otimes V}.
$$
Then $\kappa$ can be reconstructed from $\kappa_n$ by
$$
\kappa(a_1\vee\ldots\vee a_{n-1}\otimes v) = \sum_{i=0}^{n-1}\sum_{\sigma\in \mathrm{Sh}(i,n-1-i)}\eps(\sigma)a_{\sigma(1)}\vee\ldots\vee a_{\sigma(i)}\otimes \kappa_{n-i}(a_{\sigma(i+1)}\vee\ldots\vee a_{\sigma(n-1)}\otimes v)
$$
for homogeneous $a_i\in L[1]$ and $v\in V$, see \cite[(7.12)]{KS}. Compatibility with the codifferentials gives the condition that $\kappa_n$ have to satisfy in (2). D\'ecalage gives the identity that $g_n=\dec(\kappa_n)$ have to satisfy in (3). The explicit identity in (3) is written down in \cite[Def. 3.3]{Al}. 
\end{rmk}

\begin{rmk} The map $\kappa_1:(V,m_1)\to (V',m_1')$ is a morphism of cochain complexes.
A morphism of $\Linf$ modules $\kappa$ is an isomorphism if and only if $\kappa_1$ is an isomorphism, see \cite[Prop. 7.10]{KS}.
\end{rmk}

\begin{defn}
 We say that  a morphism between $\Linf$ modules $V$ and $V'$ over the $\Linf$ algebra $L$ is a {\it weak equivalence} if $\kappa_1:(V,m_1)\to (V',m_1')$ is a quasi-isomorphism. If there exists a zig-zag of weak equivalences  between the $\Linf$ modules $V$ and $V'$, we say that $V$ and $V'$ are
  {\it homotopy equivalent} $\Linf$ modules over $L$.
\end{defn}

An equivalent definition of morphisms of $\Linf$ modules is the following, see \cite[Prop. 7.11]{KS}:

\begin{prop}
If  $(V,\phi)$, $(V',\phi')$ are two $\Linf$ modules over the $\Linf$ algebra $(L,Q)$, there exists a natural abelian dgla structure on $\Hom(S(L[1])\otimes V,V')[-1]$ whose Maurer-Cartan elements are identified with the morphisms of $\Linf$ modules $(V,\phi)\to(V',\phi')$.
\end{prop}

\begin{rmk}
There are two equivalent ways to reduce the definitions of $\Linf$ modules and their morphisms in terms of $\Linf$ algebras and their morphisms, cf. \cite[3.2]{Tsy}, \cite[Thm. 1]{L}. One way is the following. A structure of $\Linf$ module on a cochain complex $(V,m_1)$ over an $\Linf$ algebra $L$ is equivalent to a morphism of $\Linf$ algebras $L\to \End (V,V)$, where $\End(V,V)$ is endowed with a natural dgla structure with the differential induced by $m_1$. \end{rmk}

The second way is given in the next proposition, see \cite[Rem. 3.2.1]{Tsy}, \cite[Thm. 1]{L}, \cite[Prop. 2.11, Prop. 2.14]{BR}.

\begin{prop}\label{rmkRed}$\;$
\begin{enumerate}[(1)]
\item An $\Linf$ module $(V,m)$ over an $\Linf$ algebra $(L,l)$ determines an $\Linf$ algebra structure on $L\oplus V$ with anti-symmetric multilinear maps $j_n:(L\oplus V)^{\otimes n}\to L\oplus V$,
$$
j_n((a_1,v_1),\ldots,(a_n,v_n))=\Bigl(l_n(a_1,\ldots,a_n), \sum_{i=1}^n(-1)^{\theta(n,i)}m_n(a_1,\ldots,\hat a_i,\ldots,a_n,v_i)\Bigr),
$$
where $a_i\in L$ and $v_i\in V$ are homogeneous, $\theta(n,i)=n-i+|v_i|(|a_{i+1}|+\ldots+|a_n|)$, and $\hat a_i$ refers to omitting $a_i$ from the list. 

Define a second grading on $L\oplus V$ such that $L$ has degree zero and $V$ has degree 1. Then this construction  gives an equivalence between the $\Linf$ $L$-module structures $m$ on $V$ and the $\Linf$ algebra structures $j$ on $L\oplus V$ satisfying: $L$ is an $\Linf$ subalgebra, $j_n$ have degree zero with respect to the second grading, and $j_n=0$ on $V^{\otimes n}$ for all $n\ge 2$. 

\item Let $f:L\to L'$ be a morphism of $\Linf$ algebras, $V$  an $\Linf$ $L$-module, and $V'$  an $\Linf$ $L'$-module. A morphism $g:V\to V'$ of $\Linf$ modules over $L$, with $V'$ viewed as an $L$-module via $f$, determines  a morphism $k: (L\oplus V,j)\to (L'\oplus V',j')$ of associated $\Linf$ algebras with anti-symmetric components
$k_n: (L\oplus V)^{\otimes n}\to L'\oplus V'$,
$$
k_n((a_1,v_1),\ldots,(a_n,v_n)) = \Bigl(f_n(a_1,\ldots,a_n), \sum_{i=1}^n(-1)^{\theta(n,i)}g_n(a_1,\ldots,\hat a_i,\ldots,a_n,v_i)\Bigr).
$$
This construction gives an equivalence between the set of pairs of morphisms $f:L\to L'$, $g:V\to V'$ as above and $\Linf$ algebra morphisms $k:L\oplus V\to L\oplus V'$ satisfying: $k_n$ have degree zero with respect to the second grading, and  $k_n=0$ on $V^{\otimes n}$ for $n\ge 2$.
 \end{enumerate}

\end{prop}

\begin{rmk}
There are sign issues in \cite{BR}. In \cite[Def. 2.4]{BR} an extra sign  in terms of the degrees of the elements for each of the summands in the identity is missing, cf. \cite[Def 2.3, Rem. 2.4]{Al}. This translates into missing signs in \cite[Def. 2.9]{BR}, cf. \cite[Def. 3.3]{Al}. In \cite[Def. 3.1]{BR}, $l_n$ misses a sign $(-1)^{n(n-1)/2}$, or, equivalently, the definition of $l_n$ from here has to be used, cf. Remark \ref{rmkSignDiff}.
\end{rmk}

\subs{\bf Homotopy transfer theorem.}  Let $(C,d,\left[\_\;,\_\right])$ be a dgla. We denote by $HC$ the cohomology of $(C,d)$. One can always equip $C$ with a {homotopy retract}
 \[
\xymatrix{
C \ar@(lu,ld)_h \ar@<.5ex>[r]^p & \ar@<.5ex>[l]^\iota HC
}
\] where $p:(C,d)\ra(HC,0)$ and $\iota:(HC,0)\ra (C,d)$ are morphisms of cochain complexes,  $\iota$ is a quasi-isomorphism, and  $h:C\ra C[-1]$ is a graded linear map such that $\id_C-\iota p=dh+hd$.
The homotopy transfer theorem for $\Linf$ algebras  states, see  \cite[Thm. 10.3.5]{LV}, \cite[Thm. 14.4.2]{Mane}:

\begin{thrm}\label{thmTTL}
There is an $L_\infty$ algebra structure of graded anti-symmetric multilinear maps $\{l_n\}_{n\ge 1}$ on $HC$ with $l_1=0$,  together with a weak equivalence of $L_\infty$ algebras \[(HC,0,l_2,l_3, l_4, \dots) \xrightarrow{\sim} (C,d,\left[\_\;,\_\right])\] 
such that $l_n= \sum_{\phi}\frac{1}{|\Aut(\phi)|}l_\phi$ is a sum over the isomorphism classes of rooted  binary trees with $n$ leaves and the operation $l_\phi$ is defined below.
\end{thrm}

\begin{rmk} A rooted tree with $n$ leaves is a  graph (that is, a set of vertices with edges between them such that every vertex belongs to an edge) that is connected with no cycles, with exactly $n+1$ external vertices (that is, a vertex contained by only one edge), with one preferred external vertex called the root, and the other $n$ external vertices called leaves. The non-external vertices are called internal vertices.
 The edges joining two internal vertices are called internal edges.  The root determines one outgoing edge at each vertex. The rest are called incoming edges. There is no preferred total order for the incoming edges joined at a vertex. A rooted tree is binary if every internal vertex has two incoming edges. We picture the rooted trees with leaves up and root down.  The incoming directions are from the top, the outgoing direction is towards the root. For example, here are two  different isomorphism classes of rooted binary trees with 4 leaves:
 
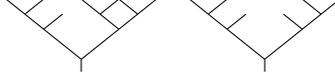
\vspace{0.2cm}

\begin{center}
\begin{tikzpicture}
[ level 1/.style={level distance=0.3cm,sibling distance=1cm},
  level 2/.style={level distance=0.4cm,sibling distance=1cm},
  level 3/.style={level distance=0.2cm,sibling distance=0.5cm},
  level 4/.style={level distance=0.2cm,sibling distance=0.5cm}]
  \node {}  [grow=up]
child{
	child{	
		child{
			child{}
			child{edge from parent[opacity=0]}			
		}
		child{
			child{}
			child{}
		}		
	}
	child{
		child{edge from parent[opacity=0]}
		child{
			child{edge from parent[opacity=0]}
			child{}
		}		
	}
};
\end{tikzpicture}\hspace{0.3cm}
\begin{tikzpicture}
[ level 1/.style={level distance=0.3cm,sibling distance=1cm},
  level 2/.style={level distance=0.4cm,sibling distance=1cm},
  level 3/.style={level distance=0.2cm,sibling distance=0.5cm},
  level 4/.style={level distance=0.2cm,sibling distance=0.5cm}]
  \node {}  [grow=up]
	child{
    		child {
      			child {
				child{}
				child{}			
			}
     			child{edge from parent[opacity=0]}
    		}
    		 child {
      			child{edge from parent[opacity=0]}
			child{
				child{}
				child{} 
			}   
		}
};
\end{tikzpicture}	
\end{center}
For a rooted binary tree with $n$ leaves $\phi$,   fix  a total order of the incoming edges at every internal vertex, or, equivalently,  fix a planar embedding $\phi'$ of the tree. This defines an operation $l_{\phi'}:HC^{\otimes n}\to HC$ by the composition scheme $\phi'$ where the leaves are labeled by $\iota$ and take the input from $HC$, every internal vertex is decorated by the Lie bracket with the two incoming directions serving as input and the outgoing direction as output,  $h$ decorates the internal edges, and the edge connecting to the root is labeled by $p$. For example, for the first planar tree $\phi'$ pictured above with the order on the leaves increasing from left to right, one obtains 
$$l_{\phi'}:v_1\otimes v_2\otimes v_3\otimes v_4\mapsto  p([\iota(v_1),h([h([\iota(v_2),\iota(v_3)]),\iota(v_4)])])$$
for homogeneous $v_i\in HC$.  To obtain the anti-symmetric version one composes with the symmetrization map $\Lambda^n (HC)\subset HC^{\otimes n}$. That is, one has a well-defined map
$$
l_\phi (v_1\wedge\ldots\wedge v_n):=\sum_{\sigma\in \mathcal S_n}\chi(\sigma)l_{\phi'}(v_{\sigma(1)},\ldots, v_{\sigma(n)}),
$$
see \cite[14.4.2]{Mane}. Equivalently, one can obtain the graded symmetric version $\tilde l$ of the transferred $\Linf$ algebra structure by composing with the symmetrization map $\rS^n(HC)\subset HC^{\otimes n}$, that is, by setting
$$
\tilde l_\phi (v_1\vee\ldots\vee v_n):=\sum_{\sigma\in \mathcal S_n}\eps(\sigma) l_{\phi'}(v_{\sigma(1)},\ldots, v_{\sigma(n)}).
$$ Then $\tilde l_n=\sum_\phi \frac{\tilde l_{\phi'}}{|\Aut(\phi)|}$. The d\'ecalage map of Proposition \ref{decalage} relates the symmetric and the anti-symmetric versions. 
\end{rmk}

\begin{rmk} The Lie bracket $\left[\_\;,\_\right]_C$ on $C$ induces a Lie bracket $\left[\_\;,\_\right]_{HC}$ on $HC$.
Since there is only one rooted binary tree with $2$ leaves, one can easily check that $ l_2=\left[\_\;,\_\right]_{HC}$. 
\end{rmk}

\begin{rmk}\label{rmkHigher}
If $v\in HC$ and $\iota(v)$ commutes with all elements of $C$, that is, $\left[\iota(v),u\right]_C=0$ for all $u\in C$, then $ l_n(v,u_1,u_2,\ldots,u_{n-1})=0$ for all $u_i \in HC$.  This is because for each planar binary tree $\phi'$ as above, there is a vanishing bracket involving $\iota(v)$, and this vanishing propagates to the whole $ l_{\phi'}(v,u_1,u_2,\ldots,u_{n-1})$ by composition.
\end{rmk}

\subs{\bf Maurer-Cartan equations and deformation functors for $\Linf$ algebras.} Let $\Art$ denote the category with objects local Artinian finite type $K$-algebras together with local morphisms. Let $\Set$ denote the category of sets. If $(A,\mm_A)$ is in $\Art$, where $\mm_A$ denotes the maximal ideal, and $(L,l)$ is a $\Linf$ algebra, then the scalar extension maps $l_n^A:=l_n\otimes\id_A$ define an $\Linf$ algebra structure on $L\otimes\mm_A$. 

\begin{defn}\label{defnMC}
The {\it Maurer-Cartan functor} of the $\Linf$ algebra $(L,l)$ is the covariant functor $\MC_L:\Art\to\Set$ defined by
$$
\MC_L(A)=\Bigl\{\omega\in L^1\otimes\mm_A\mid \sum_{n\ge 1}\frac{1}{n!}l_n^A(\omega^{\otimes n})=0\Bigr\}.
$$
\end{defn}

\begin{rmk}\label{rmkSymM0}
Equivalently, one can replace $\omega^{\otimes n}$ by $\omega^{\wedge n}$, due to the anti-symmetry of $l_n$. 
If $(L,q)$ is the symmetric formulation of the $\Linf$ algebra structure, then
$$
\MC_L(A)=\Bigl\{\omega\in L^1\otimes\mm_A\mid \sum_{n\ge 1}\frac{1}{n!}q_n^A(\omega^{\vee n})=0\Bigr\}.
$$
\end{rmk}

\begin{rmk}
Consider the commutative differential graded  algebra $K[t,dt]$ where $t$  and $dt$ are of degree $0$ and $1$, respectively. If $(A,\mathfrak{m}_A)$ in $\Art$ then $\mathfrak{m}_A\otimes K[t,dt]$ is a finite-dimensional nilpotent  cdga  and its tensor product $L\otimes \mm_A[t,dt]$  with $L$ is naturally endowed with an $L_\infty$ algebra structure which we denote by $l^{A[t,dt]}$, see \cite[Rem. 10.4.7]{Mane}:
$$
l^{A[t,dt]}_1(v\otimes a) =l_1(v)\otimes a+(-1)^{|v|}v\otimes d_{A[t,dt]}(a),
$$ 
$$
l^{A[t,dt]}_n(v_1\otimes a_1,\ldots, v_n\otimes a_n) = (-1)^{\sum_{i<j}|a_i||v_j|}l_n(v_1,\ldots,v_n)\otimes a_1\ldots a_n,
$$
for homogeneous $v, v_i\in L$, $a, a_i\in \mm_A[t,dt]$. Thus
\begin{equation}\label{eqHMC}
\begin{split}
\MC_{L[t,dt]}( A)=&\Biggl\{\omega\in  (L^1\otimes\mathfrak{m}_A[t])\oplus (L^0\otimes \mathfrak{m}_A[t]dt)   \mid  \sum_{n\ge 1}\frac{1}{n!}l_n^{A[t,dt]}(\omega^{\otimes n})= 0\Biggr\}.
\end{split}
\end{equation}
\end{rmk}

\begin{defn} Two Maurer-Cartan elements $\omega_0, \omega_1\in MC_{L}(A)$ are  {\it homotopy equivalent} if there exists an element $\omega(t,dt) \in MC_{L[t,dt]}(A)$ such that $\omega(0,0)=\omega_0$ and $\omega(1,0)=\omega_1$. \end{defn}

This is indeed an equivalence relation by \cite[Lemma 13.1.3]{Mane}. In the dgla case, homotopy equivalence agrees with gauge equivalence \cite[Prop. 10.5.5]{Mane}.

\begin{defn}\label{defMCequi}
The {\it deformation functor} of the $\Linf$ algebra $L$ is the covariant functor $\Def(L):\Art\to\Set$ given by the quotient
$\Def(L;A)=\MC_L(A)/\sim$ by homotopy equivalence.
\end{defn}

The main theorem of deformation theory is the following, see \cite[13.1.4]{Mane}:

\begin{thrm}\label{thmLDis}
Every morphism of $\Linf$ algebras $f:L\to L'$ induces a natural transformation of functors $\MC_L\to\MC_{L'}$ that factors to a natural transformation $\Def(L)\to \Def(L')$. If $f$ is a weak equivalence of $\Linf$ algebras, then $\Def(L)\to \Def(L')$ is an isomorphism of functors.
\end{thrm}

Together with the homotopy transfer theorem it gives:

\begin{thrm}\label{corFKM}
If $C$ is a dgla and $HC$ is the cohomology endowed with a transferred $\Linf$ algebra structure by Theorem \ref{thmTTL}, then $\Def(C)$ and $\Def(HC)$ are isomorphic functors.
\end{thrm}

We record here a result we need later:

\begin{lemma}\label{MCnequi} Let $C$ be a dgla and let $HC$ be endowed with a transferred $\Linf$ algebra structure as in Theorem \ref{thmTTL} via a homotopy retract diagram.
If $H^1C\otimes \mathfrak{m}_A=\MC_L(A)$  and  for every  $v$ in $H^0C\otimes \mathfrak{m}_A$ the image $(\iota\otimes\id_A)(v)$ commutes with all elements of $C\otimes\mm_A$, 
 then no two different elements in $H^1C\otimes \mathfrak{m}_A$ are homotopy equivalent. \end{lemma}

\begin{proof} Let $z(t,dt) =z_1(t)+z_2(t)dt \in (H^1C\otimes\mathfrak{m}_A[t])\oplus (H^0C\otimes \mathfrak{m}_A[t][dt]). $ Then by assumption, the Maurer-Cartan condition (\ref{eqHMC}) is reduced to 
$(\mathrm{id}_{HC}\otimes d_{A[t,dt]})(z_1)=0$ by Lemma \ref{rmkHigher}. It means in particular that $z_1$ is constant with respect to $t$. So  $z$ cannot  define a homotopy equivalence between different Maurer-Cartan elements.
\end{proof}

\subs{\bf $\Linf$ pairs and cohomology jump deformation functors.} The following terminology was introduced in \cite{BR} and is convenient to use in deformation theory with cohomology constraints. 

\begin{defn}
An {\it $\Linf$ pair}  $(L,V)$ is an $\Linf$ algebra $L$ together with an $\Linf$ $L$-module $V$. A {\it morphism of $\Linf$ pairs} between $(L,V)$ and $(L',V')$ is a pair $(f,g)$ where $f:L\to L'$ is a morphism of $\Linf$ algebras and $g:V\to V'$ is a morphism of $\Linf$ modules over $L$, where $V'$ is regarded as an $L$-module via $f$. We say that a morphism $(f,g)$ of $\Linf$ pairs is a {\it weak equivalence} if $f$ and $g$ are weak equivalences.
\end{defn}

These structures extend the corresponding notions for dgl pairs defined in \cite{BW}. Using a reduction from module to algebra structures, cf. Proposition \ref{rmkRed}, the transfer theorem was updated to pairs  in \cite[Thm. 2.25]{BW}:

\begin{thrm}\label{thrmTLP}
Let $(C,M)$ be a dgl pair. Then there exists an $\Linf$ pair structure on the cohomology pair $(HC,HM)$ with zero differentials, and second order operations inherited from $(C,M)$, together with a weak equivalence of $\Linf$ pairs $(HC,HM)\to (C,M)$.
\end{thrm}

\begin{defn}
Let $(L,V)$ be an $\Linf$ pair. Let $A\in\Art$ and $\omega\in\MC_L(A)$. Define the graded linear map $d_\omega:V\otimes A\to V\otimes A$ of degree 1 by
$$
d_\omega(\_):=\sum_{n\ge 0}\frac{1}{n!}m_{n+1}^A(\omega^{\otimes n}\otimes\_)
$$
where $m_n$ are the $\Linf$ $L$-modules structure maps on $V$ and $m_n^A=m_n\otimes \id_A$.
\end{defn}

\begin{rmk}\label{rmkSymM} 
Equivalently, one can replace $\omega^{\otimes n}$ by $\omega^{\wedge n}$, due to the anti-symmetry of $m_n$. 
If $(M,\phi)$ is the symmetric formulation of the $\Linf$ module structure, then
$$
d_\omega(\_):=\sum_{n\ge 0}\frac{1}{n!}\phi_{n+1}^A(\omega^{\vee n}\otimes\_)
$$
cf. Remark \ref{rmkSymM0}.
\end{rmk}

The following was shown in \cite[Thm. 3.7]{BR}:

\begin{lemma} Let $(L,V)$ be an $\Linf$ pair with $L$ and $V$ $\bN$-graded and $V$ bounded above as a cochain complex. Let $A\in\Art$ and $\omega\in\MC_L(A)$. Then $(V\otimes A,d_\omega)$ is a complex of $A$-modules with finitely generated cohomology. If $\omega'\in \MC_L(A)$ is homotopy equivalent to $\omega$, then the complexes $(V\otimes A,d_\omega)$ and $(V\otimes A,d_\omega')$ are homotopy equivalent. 
\end{lemma}

In particular, the cohomology jump ideals $J^i_k(V\otimes A,d_\omega)\subset A$ are well-defined and independent of the homotopy equivalence class of $\omega$.

\begin{defn}\label{defDik} (\cite[Def. 3.6]{BR})
Let $(L,V)$ be an $\Linf$ pair with $L$ and $V$ $\bN$-graded and $V$ bounded-above as a cochain complex. For $i,k\in\bN$ the {\it cohomology jump deformation functors} $\Def^i_k(L,V)$ 
are given for $A\in \Art$ by the quotient
\begin{equation}\label{eqLJI2}
\Def^i_k(L,V;A)=\{\omega\in  \MC_L(A)\mid J^i_k(V\otimes A,d_\omega)=0\}/\sim
\end{equation}
by homotopy equivalence.
\end{defn}

It is shown in \cite{BR} that $\Def^i_k(L,V)$ are subfunctors of $\Def(L)$. In the case of dgl pairs they agree with the cohomology jump subfunctors defined already in \cite{BW}. Moreover, one has the following result extending the dgl pairs case, see \cite[Thm. 3.8]{BR}:

\begin{theorem}\label{cjdf pairs invariance} Let $(L,V)$, $(L',V')$ be two $\Linf$ pairs, with $L,L
,V,V'$ $\bN$-graded and $V, V'$ bounded above as cochain complexes. If $(f,g): (L,V)\rightarrow (L',V')$ is a weak equivalence of $\Linf$ pairs, then for all $i, k\in\bN$ there is an isomorphism of subfunctors $\Def^i_k(L,V)\to\Def^i_k(L',V')$ compatible with the isomorphism of functors $\Def(L)\ra \Def(L')$ from Theorem \ref{thmLDis}.
\end{theorem}

\begin{rmk} In general if $(f,g): (L,V)\rightarrow (L',V')$ is only a morphism of $\Linf$ pairs, there is no natural transformation of subfunctors $\Def^i_k(L,V)\to \Def^i_k(L',V')$ although there is a natural transformation of deformation functors $\Def(L)\to\Def(L')$. (This is stated incorrectly at the beginning of \cite[3.4]{BR}, although it does not affect  the proof there.)
\end{rmk}

Together with the homotopy transfer theorem for pairs one has, see \cite[Thm. 1.6]{BR}:

\begin{theorem}\label{thmHCHMa}
\label{thmMainLDef} Let $(C,M)$ be a dgl pair with $C, M$ $\bN$-graded and $M$ bounded above as a cochain complex. If $(HC,HM)$ is the cohomology endowed with a transferred $\Linf$ pair structure from Theorem \ref{thrmTLP}, the cohomology jump subfunctors 
$
\Def^i_k(HC,HM)\subset \Def(HC)
$
are naturally isomorphic to the cohomology jump subfunctors 
$
\Def^i_k(C,M)\subset \Def(C)
$
for all $i,k\in\bN$.
\end{theorem}

\section{Partial formality} \label{secTecCore}

The goal in this section is to give another proof of Theorem \ref{thmGenToCone}, closer to the proof of Polishchuk of \cite[Thm 0.1]{Po}. For this we prove Theorem \ref{thrmLinf1Gen} which is an exact analog for $\Linf$ pairs of \cite[Thm. 3.1]{Po}. 
We  use the notions recalled in Section \ref{apxLinf}. We begin with  the analog of  \cite[Lemma 1.1]{Pob} where the case of $A_\infty$ algebras is treated.

\begin{thrm}\label{thrmLinf1Gen} Let $(M,V)$ be an $L_\infty$ algebra together with a module, both of finite dimension over a field $K$ of characteristic zero, such that:
\begin{itemize} 
\item $M^i=0$ and $V^i=0$ for $i\neq 0,1$,
\item the differentials on $M$ and $V$ are zero, 
\item  the linear map $\pi:V^0\otimes (V^1)^\vee\to (M^1)^\vee$ induced from the multiplication map $m_2:M^1\otimes V^0\to V^1$ is injective.
\end{itemize} 
Then there exists an $L_\infty$   algebra structure on $M$  isomorphic to the original one such that for the induced module structure on $V$, $m_2:M\otimes V\to V$ is the original one and the higher  multiplication maps
$
m_n: (M^1)^{\otimes (n-1)}\otimes V^0\ra V^1
$ 
vanish for $n>2$.
\end{thrm}

\begin{lemma}\label{trans} Let $(L, l_1,l_2,\ldots)$ be an $L_\infty$ algebra and $f=\lbrace f_n: L^{\otimes n} \rightarrow L \rbrace_{n \geq 1}$  a collection of graded anti-symmetric multilinear maps of degree $1-n$ such that $f_1=\id$. There is an unique $L_\infty$ algebra structure $l'= \lbrace l_n' \rbrace_{n\ge 1}$ on $L$ such that $f$ is an isomorphism of $\Linf$ algebras between $(L,l)$ and $(L,l')$. 
\end{lemma}

\begin{proof}  Using d\'ecalage to switch to the graded symmetric setting, one has a collection of graded linear  maps $ \sf_n=\dec^{-1}(f_n):\rS^n(L[1]) \to L[1]$, ${n \geq 1}$, by Corollary \ref{corDeca}. This gives rise to a morphism of reduced coalgebras $F: \; \rrS(L[1]) \rightarrow \rrS (L[1])$. For this see \cite[Cor. 11.5.4]{Mane} for a recursive formula for $F$, or equivalently, $F$ can be reconstruced in one shot by the same formula as in Remark \ref{Taylorcoefmor}. Since $\sf_1$ is an isomorphism of graded vector spaces, $F$ is an isomorphism of coalgebras, see \cite[Cor. 11.5.5]{Mane}. Thus we can transfer the  codifferential $Q$ on  $\rrS(L[1])$  associated to $ (L, l)$ by Proposition-Definition \ref{equivdeflalg} to another one, denoted by $Q'$, on $\rrS(L[1])$ along the isomorphism $F$. Then $F$  commutes with $Q$ and $Q'$. The  codifferential  $Q'$  corresponds to an $L_\infty$ structure $l'$ on $L$, such that $(L,Q)$ and $(L,Q')$, equivalently $(L,l)$ and $(L,l')$, are isomorphic $\Linf$ algebras. 
\end{proof}

\begin{rmk}
 If $(M,m)$ is an $L_\infty$   module over $(L, l)$, there is an induced module structure on $M$ over $(L, l')$. Then the $\Linf$ pairs $((L,l),(M,m))$ and $((L, l'),(M,m')$ are isomorphic.
\end{rmk}

The following explicits part of the new module structure in a particular case of Lemma \ref{trans}. 

\begin{lemma} \label{particulartrans}  Let $((L,l),(M,m))$ be an $\Linf$ pair such that the differentials $l_1$ and $m_1$ are zero. Fix  $2\leq k \in \mathbb{N}$. Let  $f=\lbrace f_n: L^{\otimes n} \rightarrow L \rbrace_{n \geq 1}$ be a collection of graded anti-symmetric multilinear maps such that $f_n$ has degree $1-n$ and: $f_1=\id$, $f_k\neq 0$, and $ f_n=0$ if $1\neq n\neq k$. 
Let $l'$ and $m'$ be the new $\Linf$ structures transferred  along $f$ as in Lemma \ref{trans}.  Then   $
m'_n =m_n$ if $1 \leq n\leq k$, 
and, if $n=k$,
$$m'_{k+1}(x_1\wedge  \cdots \wedge  x_k \otimes v)=m_{k+1} (x_1\wedge  \cdots \wedge  x_k \otimes v)- m_2(f_k(x_1\wedge  \cdots \wedge  x_k )\otimes v)$$
for homogeneous $x_i\in L$, $v\in M$.
\end{lemma}

\begin{proof} Let $Q$ and  $\phi$  be the codifferentials associated  to $l$ and $m$ as in Proposition-Definition \ref{equivdeflalg} and Proposition-Definition \ref{lmodule}, respectively.  By minimality of the pair $((L,l),(M,m))$ one has $Q_1^1=0$ and $\phi_1=0$. Let $F$ be the isomorphism of differential graded coalgebras
$F: \; (\rrS (L[1]),Q) \rightarrow (\rrS(L[1]),Q')$ associated to $f$ as in the proof of Lemma \ref{trans}. Let $F_n^i$ be defined as in Remark \ref{expmor}.  The Taylor coefficients $F^1_n=\dec^{-1}(f_n)$ satisfy $F_1^1=\rm{id}$, $F_k^1 \neq 0$,  $F_n^1=0$ for $1\neq n\neq k$, by the explicit description of the d\'ecalage isomorphism from Corollary \ref{corDeca} and the assumptions on $f$. By the formula reconstructing $F$ from its Taylor coefficients from Remark \ref{Taylorcoefmor}, one has $F^1_n= 0$ for $ 2 \leq n \leq k-1$, and $F_n^n=\mathrm{id}$.

The Taylor coefficients of the inverse $G$ of the isomorphism of $\Linf$ algebras $F$ 
are given by   Proposition \ref{coalgiso}. It follows that 
 $$\begin{cases}
 G^1_1=(F^1_1)^{-1}=\mathrm{id}, &  \\ 
 G_n^1= -\Big(\sum_{i=1}^{n-1}G_i^1F^i_n\Big)(F^n_n)^{-1}=0& \text{ for } 2 \leq n \leq k-1, \\
 G_k^1=-\Big(\sum_{i=1}^{k-1}G_i^1F^i_k\Big)(F^k_k)^{-1}= - F_k^1.&
\end{cases}$$
 Hence, for homogeneous $x_1,\ldots,x_n \in L[1]$ with $n\leq k$, 
$$G(x_1 \vee \cdots  \vee x_n )=\begin{cases}
x_1\vee \cdots \vee x_n  +G_n^1(x_1 \vee \cdots \vee x_n ) & \text{ if } n=k \\ 
x_1 \vee \cdots \vee x_n  & \text{ if } 1 \leq  n \leq k-1.
\end{cases}$$
Now, the new module structure $\phi'$ is defined in terms of the old one $\phi$ and $G$ by  $$\phi'(x_1 \vee \cdots  \vee x_n  \otimes v)=\phi(G(x_1\vee  \cdots \vee  x_n) \otimes v)$$
with $v\in M$ homogeneous. If $1\leq n \leq k-1$, then $\phi'(x_1\vee  \cdots \vee x_n \otimes v)=\phi(x_1\vee  \cdots  \vee  x_n \otimes v)$. In particular, $\phi_n'=\phi_n $, and equivalently $m_n'=m_n$, for $2\leq n \leq k-1$ by Remark \ref{equivLmod}. If $n=k$ then 
$$
\phi'(x_1\vee  \cdots \vee  x_k \otimes v)  = \phi(x_1\vee  \cdots \vee  x_k \otimes v)+\phi(G_k^1(x_1\vee  \cdots \vee  x_k )\otimes  v).
$$
Using the Taylor expansion (\ref{eqcomo}) of $\phi$, we compute the last term
\begin{align*}
 \phi(G_k^1(x_1\vee  \cdots \vee  x_k ) \otimes v) &= Q_1^1(G_k^1(x_1\vee  \cdots \vee x_k ))\otimes v +\phi_2(G_k^1(x_1, \ldots , x_k ) \otimes v)\\
 &\quad\quad +(-1)^{\left | G_k^1(x_1\vee  \cdots \vee  x_k ) \right |}\cdot G_k^1(x_1\vee \cdots \vee  x_k ) \otimes \phi_1(v)\\ 
 &= \phi_2(G_k^1(x_1 \vee  \cdots \vee  x_k ) \otimes v)
\end{align*} where the last line follows from the fact that the differentials $Q_1^1$ and $\phi_1$ are trivial. It implies 
\begin{align*}
\phi'_{k+1}(x_1\vee  \cdots \vee  x_k \otimes v) &= \phi_{k+1} (x_1\vee  \cdots \vee  x_k \otimes v)+\phi_2(G_k^1(x_1\vee  \cdots \vee  x_k)\otimes v).\\ 
 &= \phi_{k+1}(x_1\vee  \cdots \vee  x_k \otimes v)-\phi_2(F_k^1(x_1\vee  \cdots \vee  x_k)\otimes v).
\end{align*} which  by d\'ecalage, see Remark \ref{equivLmod}, is equivalent to $$m'_{k+1}(x_1\wedge  \cdots \wedge  x_k \otimes v)=m_{k+1} (x_1\wedge  \cdots \wedge  x_k \otimes v)- m_2(f_k(x_1\wedge  \cdots \wedge  x_k )\otimes v)$$
for homogeneous $x_i\in L$. This  finishes the verification.
\end{proof}

\subs{\bf Proof of Theorem \ref{thrmLinf1Gen}.} 
The strategy is analogous to  \cite[Theorem 3.1]{Po} where the case of $A_\infty$ algebras is considered. 

{\it Step 1. Setup.}
We have $M=M^0\oplus M^1$ and $V=V^0\oplus V^1$ as graded vector spaces. Let $(M,l)$ denote the $\Linf$ algebra structure on $M$, where $l=\{l_n:M^{\otimes n}\to M\}_{n\ge 1}$ are graded anti-symmetric multilinear maps of degree $2-n$. Let $(V,m)$ denote the structure of $\Linf$ module over $M$ on $V$, where $m=\{m_n:M^{\otimes (n-1)}\otimes V\to V\}_{n\ge 1}$ are graded, anti-symmetric with respect to $M^{\otimes (n-1)}$, multilinear maps of degree $2-n$. By assumption we have that the differential $l_1$ of $M$ and the differential $m_1$ of $V$ are zero.
By $$m_n:M^{\otimes (n-1)}\otimes V^0\to V^1$$ we  also mean the restriction of the map $m_n$, slightly abusing notation if the context is clear. We denote by its partial dual by $$ 
\tilde{m}_{n}: (M^1)^{\otimes (n-1)} \ra \mathrm{Hom}( V^0,V^1).
$$
The anti-symmetry of $m_{n}$ with respect to $(M^1)^{\otimes (n-1)}$ is equivalent to the anti-symmetry of $\tilde{m}_{n}$.

{\it Step 2. The roadmap.}
The desired $\Linf$ algebra structure on $M$ will be obtained as an infinite composition 
$$
\ldots\circ f^{(4)}\circ f^{(3)} \circ f^{(2)}
$$
of $\Linf$ algebra isomorphisms  $f^{(k)}=\{f^{(k)}_n:M^{\otimes n}\to M\}_{n\ge 1}$  as in  Lemma \ref{particulartrans} with $f^{(k)}_1=\id_M$, $f^{(k)}_n=0$ if $1\neq n\neq k$, and $f^{(k)}_k\neq 0$. 
Recall that a morphism of $\Linf$ algebras is equivalent to a morphism of coalgebras compatible with the codifferentials. So here by composition we mean the composition of the associated morphisms of coalgebras.

We show now that the infinite composition of such isomorphisms is well-defined. As in Lemma \ref{trans}, each $f^{(k)}$ corresponds to a graded linear map $F^{(k)}:\rS(M[1])\to \rS(M[1])$ reconstructed from its Taylor coefficients $\{\sf^{(k)}_n=\dec^{-1}(f^{(k)}_n)\}_{n\ge 1}$ by means of the formula from Remark \ref{Taylorcoefmor}.
It is enough to show that their composition
$$
F=\ldots\circ F^{(4)}\circ F^{(3)} \circ F^{(2)}
$$
is well-defined. Equivalently, it is enough to show that the Taylor coefficients $F^1_n:\rS^n(M[1])\to M[1]$ are well-defined. We will use the notation $F^m_n:\rS^n(M[1])\to \rS^m(M[1])$ for the various components of maps $F$, as used in Remark \ref{expmor}. We also set $F^m=\sum_nF^m_n$, $F_n=\sum_mF^m_n$.

Fix $a\in \rS^n(M[1])$. Let $s\ge n$. Then it is enough to show that $(F^{(s)}\circ \ldots  \circ F^{(2)})^1_n(a)$ depends only on $F^{(2)},\ldots, F^{(n)}$.
This is obviously true for $s=n$. Assume $s>n$. Then
$$(F^{(s)}\circ \ldots  \circ F^{(2)})^1_n(a)=F^{(s),1}(F^{(s-1)}(  \ldots (F^{(3)}(F^{(2)}_n(a)))\ldots ))
$$
$$
=\sum_j \sf^{(s)}_j(F^{(s-1),j}(  \ldots (F^{(3)}(F^{(2)}_n(a)))\ldots ))
$$
$$=F^{(s-1),1}(  \ldots (F^{(3)}(F^{(2)}_n(a)))\ldots ) + \sf^{(s)}_s(F^{(s-1),s}(  \ldots (F^{(3)}(F^{(2)}_n(a)))\ldots )) 
$$ 
The first term depends only on $F^{(2)},\ldots, F^{(n)}$ by induction. The second term is zero since $F^{(2)}_n(a)$, $F^{(3)}(F^{(2)}_n(a))$, and so on, are elements of $\bigoplus_{m=1}^n\rS^m(M[1])$ by the reconstruction formula from Remark \ref{Taylorcoefmor}, whereas $\sf^{(s)}_s$ takes input only from $\rS^s(M[1])$. This shows that the infinite composition is well-defined.

Let $m^{(k)}=\{m^{(k)}_n\}_{n\ge 1}$ be the $\Linf$ module structure on $V$ over the $\Linf$ algebra structure $M$ obtained after the composition $f^{(k)}\circ\ldots \circ f^{(2)}$ of $\Linf$ algebra isomorphisms. Then $m^{(k)}$ is isomorphic to the original module structure. We will construct $f^{(k)}$ with the additional property  
\be\label{eqmk}
m^{(k)}_n:(M^1)^{\otimes (n-1)}\otimes V^0\to V^1\text{ is zero for }3\le n\le k+1.
\ee
Lemma \ref{particulartrans}  then implies that the limit module structure $m^{(\infty)}=\lim_{k\to\infty}m^{(k)}$, which converges since the infinite composition of $\Linf$ algebra isomorphisms is well-defined, satisfies the claimed properties of the theorem, namely, $m^{(\infty)}$ and $m$ are isomorphic module structures on $V$ over $M$, $m^{(\infty)}_2=m_2$ is the original one and $m_n^{(\infty)}:(M^1)^{\otimes (n-1)}\otimes V^0\to V^1$ is zero for $3\le n$.

{\it Step 3. The construction of $f^{(k)}_k$.}  
 By surjectivity of $\tilde m_{2}$, we can find a graded multilinear map $f^{(2)}_2$  filling a commutative diagram
$$
\begin{tikzcd}
(M^1)^{\otimes 2} \ar["f_2^{(2)}" ', d, dashed] \ar[dr, "\tilde{m}_{3}"]& \\
M^1 \ar[r, twoheadrightarrow, "\tilde{m}_{2}" ']& \Hom(V^0,V^1).
\end{tikzcd}
$$
Since $\tilde m_{3}$ is anti-symmetric, we can find such $f^{(2)}_2$ anti-symmetric by working with $\Lambda^2M^1$ instead of $(M^1)^{\otimes 2}$. Extend $f^{(2)}_2$  to $M^{\otimes 2}\to M$.  
By Lemma \ref{particulartrans} the resulting module structure $m^{(2)}$ satisfies the desired condition (\ref{eqmk}), namely 
$
m_{3}^{(2)}:(M^1)^{\otimes 2}\otimes V^0\to V^1
$
is zero.
We redefine $m=m^{(2)}$.

Continuing in this fashion we construct all $f^{(k)}_k$ by filling in a diagram
$$
\begin{tikzcd}
(M^1)^{\otimes k}  \ar[drr, "\tilde{m}_{k+1}"] \ar[d,  dashed, "f_k^{(k)}" ']& & \\
M^1 \ar[rr, twoheadrightarrow, "\tilde{m}_{2}" ']&& \Hom(V^0,V^1).
\end{tikzcd}
$$
Then (\ref{eqmk}) holds for all $m^{(k)}$.
$\hfill\Box$

\medskip

\subs{\bf Second proof of Theorem \ref{thmGenToCone}.} In fact we give another proof of Theorem \ref{thm1GDik}. The latter implies  Theorem \ref{thmGenToCone}. 

Let $(M,V)$ be as in Theorem \ref{thm1GDik}.
By Theorem \ref{cjdf pairs invariance}  we can compute the functors
 $\Def^0_k(M,V)$ attached to the $\Linf$ pair $(M,V)$ using the $\Linf$ structure obtained by Theorem \ref{thrmLinf1Gen}. Denote  by $m=\{m_n\}_{n\ge 1}$ the $\Linf$ module structure on $V$. 

We know from the first part of the proof of Theorem \ref{thm1GDik} that no two elements in $M^1\otimes \mm_A$ are homotopy equivalent by  Lemma \ref{MCnequi} if one considers the original $\Linf$ algebra structure on $M$. By Theorem \ref{thmLDis}, the same is true for the new isomorphic $\Linf$ algebra structure on $M$ obtained via Theorem \ref{thrmLinf1Gen}.
Thus
$
\Def(M)\simeq (\widehat{M^1})_{\bz}. 
$

By Theorem \ref{thrmLinf1Gen}, $m_n$ vanish on $(M^1)^{\otimes (n-1)}\otimes V^0$ for $n\ge 3$. Thus the degree-zero cohomology jump subfunctors simplify to 
\begin{equation}\label{eqLJIp}
\Def^0_k(M,V;A)=\{\omega\in  M^1\otimes \mm_A\mid   J^0_k(V\otimes A, m_{2}^A(\omega\otimes \_))=0 \},
\end{equation}
see Definition \ref{defDik}. 
This means that in the original proof of Theorem \ref{thm1GDik} we have $d_{univ}=B$ on the nose for the $\Linf$ structure obtained via  Theorem \ref{thrmLinf1Gen}. The rest of the proof stays the same.
$\hfill\Box$

\end{document}